\documentclass[11pt, reqno]{amsart}

\usepackage{amsmath, amssymb, amsthm, cases,  mathtools, multicol, hyperref, subcaption}
\usepackage[left=.75 in, right=.75 in,top=.75 in, bottom=.75 in]{geometry}
\usepackage[usenames, dvipsnames]{xcolor}

	\usepackage{graphicx, tikz, pgf, wrapfig, pgfplots}
	\pgfplotsset{compat=1.13}
	\usetikzlibrary{arrows.meta}
	\usepgfplotslibrary{fillbetween}
	\def\vbigsample{100} 
	\def\bigsample{100} 
	\def\smallsample{50} 
\hypersetup{
	colorlinks=true,
	linkcolor=Cerulean,
	citecolor=Thistle
}
\usetikzlibrary{cd}
\numberwithin{equation}{section}

\newcommand{\p}{w} 
\def\tayh{\mathfrak{h}} 
\def\tayg{\mathfrak{g}} 
\newcommand{\half}{\frac{1}{2}} 
\newcommand{\N}{\mathbb{N}} 
\newcommand{\Z}{\mathbb{Z}} 
\newcommand{\R}{\mathbb{R}} 
\newcommand{\T}{\mathbb{T}} 
\newcommand{\defeq}{\vcentcolon=} 
\newcommand{\eqdef}{=\vcentcolon} 
\DeclareMathOperator{\tr}{tr} 
\DeclareMathOperator{\ext}{ext} 
\DeclareMathOperator{\sym}{Sym} 
\DeclareMathOperator{\id}{id} 
\DeclareMathOperator{\cof}{cof} 
\newcommand{\rnum}[1]{
	\textup{\uppercase\expandafter{\romannumeral#1}}
}
\newcommand{\setdef}[2]{\left\lbrace #1 \;\middle|\; #2 \right\rbrace} 
\newcommand{\jap}[2]{{\left\langle #1 \right\rangle}^{#2}} 
\newcommand{\inv}{^{-1}} 
\DeclareRobustCommand{\rchi}{{\mathpalette\irchi\relax}} 
\newcommand{\irchi}[2]{\raisebox{\depth}{$#1\chi$}} 
\newcommand{\x}{\mathcal{X}} 
\newcommand{\y}{\mathcal{Y}} 
\newcommand{\jet}{\mathcal{J}} 
\theoremstyle{plain}
\newtheorem{theorem}{Theorem}[section] 

\newtheorem{prop}[theorem]{Proposition}
\newtheorem{lemma}[theorem]{Lemma}
\newtheorem{cor}[theorem]{Corollary}
\theoremstyle{definition}
\newtheorem{definition}[theorem]{Definition}
\newtheorem{remark}[theorem]{Remark}
\newtheorem{example}[theorem]{Example}
\newcommand{\brac}[1]{\left( #1 \right)} 
\newcommand{\Bigbrac}[1]{\Big( #1 \Big)} 
\newcommand{\abrac}[1]{\left\langle #1 \right\rangle} 
\newcommand{\cbrac}[1]{\left\lbrace #1 \right\rbrace} 
\newcommand{\sbrac}[1]{\left[ #1 \right]} 
\newcommand{\vbrac}[1]{\left\vert #1 \right\vert} 
\newcommand{\ocbrac}[1]{\left( #1 \right]} 
\newcommand{\cobrac}[1]{\left[ #1 \right)} 
\DeclarePairedDelimiter\abs{\lvert}{\rvert} 
\newcommand{\norm}[2]{{\left\lvert \left\lvert #1 \right\rvert \right\rvert}_{#2}} 
\newcommand{\normtypdom}[4]{{\left\lvert \left\lvert #1 \right\rvert \right\rvert}_{#2^{#3} \brac{#4}}}
\newcommand{\parabolicOrder}[1]{\abs{#1}_{t,\overline{x}^2}} 
\DeclarePairedDelimiter\ceil{\lceil}{\rceil} 
\DeclarePairedDelimiter\floor{\lfloor}{\rfloor} 
\newcommand{\nugeo}{\nu^\geo_{\partial\Omega}} 
	\newcommand{\diff}{\mathrm{d}} 
	\newcommand{\Dt}{\frac{\mathrm{d}}{\mathrm{d}t}} 
	\newcommand{\pdt}{\partial_t} 
	\newcommand{\pdtm}{D_t} 
	\newcommand{\nablatwo}{\overline{\nabla}} 
	\newcommand{\nablatwoemb}{\widetilde{\nabla}} 
	\newcommand{\diver}{\nabla\cdot} 
	\newcommand{\symgrad}{\mathbb{D}} 
		\newcommand{\divgeo}{\nabla^{\mathcal{G}}\cdot} 
		\newcommand{\geo}{\mathcal{G}} 
	\newcommand{\will}{\mathcal{W}} 
	\newcommand{\quadw}[1]{\mathcal{Q}_{#1}} 
	\newcommand{\fvr}{\delta \mathcal{W}} 
	\newcommand{\svr}[1]{\delta_{#1}^{2} \mathcal{W}} 
	\newcommand{\hovr}[2]{\delta_{#2}^{#1} \mathcal{W}} 
	\newcommand{\enimp}{\mathcal{E}} 
	\newcommand{\dsimp}{\mathcal{D}} 
	\newcommand{\eneq}{\overline{\mathcal{E}}} 
	\newcommand{\dseq}{\overline{\mathcal{D}}} 
	\newcommand{\engeo}{\widetilde{\mathcal{E}}} 
	\newcommand{\dsgeo}{\widetilde{\mathcal{D}}} 
	\newcommand{\eneqfunc}{\overline{E}} 
	\newcommand{\dseqfunc}{\overline{D}} 
	\newcommand{\engeofunc}{\widetilde{E}} 
	\newcommand{\dsgeofunc}{\widetilde{D}} 
	\newcommand{\geocor}{\mathcal{G}} 
	\newcommand{\rem}{\mathcal{N}} 

\title[Viscous surface waves and surface energies]{The viscous surface wave problem with generalized surface energies}

\author{Antoine Remond-Tiedrez}
\address{
Department of Mathematical Sciences\\
Carnegie Mellon University\\
Pittsburgh, PA 15213, USA
}
\email[A. Remond-Tiedrez]{aremondt@andrew.cmu.edu}
\thanks{This work was initiated at the Institute for Computational and Experimental Research in Mathematics (ICERM) during the Spring 2017 semester program ``Singularities and Waves In Incompressible Fluids,'' which was supported by an NSF Grant (DMS \#1439786). }

\author{Ian Tice}
\address{
Department of Mathematical Sciences\\
Carnegie Mellon University\\
Pittsburgh, PA 15213, USA
}
\email[I. Tice]{iantice@andrew.cmu.edu}
\thanks{I. Tice was supported by a Simons Foundation Grant (\#401468) and an NSF CAREER Grant (DMS \#1653161).   }



\begin{document}

\begin{abstract}
We study a three-dimensional incompressible viscous fluid in a horizontally periodic domain with finite depth whose free boundary is the graph of a function.
The fluid is subject to gravity and generalized forces arising from a surface energy.  The surface energy incorporates both bending and surface tension effects.
We prove that for initial conditions sufficiently close to equilibrium the problem is globally well-posed and solutions decay to equilibrium exponentially fast, in an appropriate norm.
Our proof is centered around a nonlinear energy method that is coupled to careful estimates of the fully nonlinear surface energy.
\end{abstract}

\maketitle

\tableofcontents
\addtocontents{toc}{\protect\setcounter{tocdepth}{1}}


\section{Introduction }

In this paper we study the dynamics of a three-dimensional periodic layer of viscous incompressible fluid bounded below by a rigid interface and above by a moving free boundary.  The free boundary is advected with the fluid, but the configuration of the free boundary gives rise to surface stresses that act as forcing terms on the fluid.  In this introductory section we discuss the origin and nature of the surface stresses and then record the equations of motion. 

\subsection{Surface energies}\label{sec:intro_se_discussion}

We will restrict our attention in this paper to surface stresses that are generated as generalized forces associated to an energy functional that depends on the configuration of the surface.  Here the generalized force is understood in the sense that it is the negative gradient of the energy.  The classical example of such a force is surface tension, which is associated to the energy functional given by a constant multiple of the area functional (the constant is known as the coefficient of surface tension).  The generalized force is then the mean curvature operator, which is the trace of the second fundamental form.  We can account for higher-order geometric effects by considering more general functionals depending on the second fundamental form itself.  The classical example of such an energy is the Willmore functional, which is the square of the mean curvature integrated over the surface.  Our goal here is to briefly survey the vast literature associated to Willmore-type energies and their relation to interfacial mechanics. 

The Willmore energy was popularized in the differential geometry literature by Willmore's initial work on it \cite{willmore_65} and in his books on Riemannian geometry \cite{willmore_book_82,willmore_book_93}.  Willmore also formulated the so-called Willmore conjecture, which predicted the minimizers of the energy among immersed tori.  The Willmore conjecture was proved recently by Marques--Neves \cite{marques-neves}.  Critical points of the Willmore energy remain an active topic of study in geometric analysis and PDE: for example,  Kuwert--Sch\"atzle \cite{kuwert_schatzle_2004} studied removable singularities, Rivi\`ere \cite{riviere-willmore-inventiones} developed a theory of weak Willmore immersions, and Bernard--Rivi\`ere \cite{bernard_riviere_2014} proved results about energy quantization and compactness.

Remarkably, energies of Willmore-type arise naturally in many areas of applied mathematics, and so such energies have received much attention outside of differential geometry.  Roughly speaking, one can justify the widespread appearance of Willmore-type energies in applications through the lens of dimension reduction in elasticity.  In many applications one considers a thin three-dimensional elastic material.  When the size of the thin direction is very small relative to the two other directions, then it is natural to seek an effective two-dimensional model, thereby reducing the dimension.  A rigorous derivation of Willmore-type energies as $\Gamma-$limits of three-dimensional elastic energies was carried out by Friesecke--James--M\"uller \cite{friesecke_james_muller_1,friesecke_james_muller_2} for plates and Friesecke--James--Mora--M\"uller \cite{friesecke_james_mora_muller} for shells.  

One major area of interest in these energies is the study of biological membranes, lipid bi-layers, and vesicles.   All of these structures can be thought of as very thin elastic materials, and should thus have some relation to Willmore-type energies.  In \cite{helfrich} Helfrich introduced such an energy to model the structure of lipid bi-layers, which led to these energies being standard modeling tools in membrane biology.  More recent advances have considered coupled models of fluid-membrane dynamics: Du--Li--Liu \cite{du_li_liu_2007} and Du--Liu--Ryham--Wang \cite{du_etal_2009} used phase field models to model fluid dynamics coupled to vesicles, Farshbaf-Shaker--Garcke \cite{farshbaf_etal_2011} developed thermodynamically consistent higher order phase field models, and Ryham--Klotz--Yao--Cohen \cite{ryham_etal} used Willmore-type energies to study the energetics of membrane fusion.  

The coupling of the full fluid equations to surface stresses generated by Willmore-type energies presents numerous analytical challenges.  Cheng--Coutand--Shkoller \cite{cheng-coutand-shkoller} proved a local existence result for a viscous fluid coupled to a nonlinear elastic biofluid shell, and Cheng--Shkoller \cite{cheng_shkoller_2010} proved local existence for a model with a Koiter shell.  Local existence results for similar models related to hemodynamics were proved by Muha--\^{C}ani\'c \cite{muha_canic_2015,muha-canic-hemodynamics}.  We refer to the work of Bonito--Nochetto--Pauletti \cite{bonito_nochetto_pauletti_2011} and Barett--Garcke--N\"{u}rnberg \cite{barrett_garcke_nurnberg_2017} and the references contained therein for a discussion of the numerical analysis of such models.

A second major area of interest in energies of this type is the study of thin layers of ice, which can be thought of as thin elastic materials.  We refer to the book by Squire--Hosking--Kerr--Langhorne \cite{squire1996moving} and the references therein for an overview of the physics specific to ice sheets.  We refer to the work of Plotnikov--Toland \cite{plotnikov_toland} for a discussion of how the Willmore functional is related to bending energies for thin sheets of ice.  The question of how fluids couple to the dynamics of ice sheets has attracted much attention in recent years, though most attention has focused on inviscid fluids.  Solitary and traveling wave solutions and effective equations were studied by Milewski--Vanden-Broeck--Wang \cite{milewski-vanden-broeck-wang}, Wang--Vanden-Broeck--Milewski \cite{wang_vandedbroeck_milewski_2013}, and Trichtchenko--Milewski--Parau--Vanden-Broeck \cite{trich_etal} in two dimensions and by Milewski--Wang \cite{milewski-wang} and Trichtchenko--Parau--Vandedn-Broeck--Milewski \cite{trich_etal_3d} in three dimensions.  For two-dimensional irrotational two-fluid flows, Liu--Ambrose \cite{liu_ambrose_2017} proved well-posedness and Akers--Ambrose--Sulon \cite{akers_ambrose_sulon_2017} constructed traveling wave solutions for a two-fluid model.  The one-fluid model was studied by Ambrose--Siegel \cite{ambrose_siegel_2017}.

Interestingly, Willmore-type energies also appear in other applications with no clear connection to thin elastic structures.  In \cite{rubinstein17} Rubinstein details how the energy appears in optics in questions related to optimal lens design.   Hawking \cite{hawking-mass} also introduced a Willmore-like energy in his study of gravitational radiation.

\subsection{Examples of surface energies.}\label{sec:introExSurfEnerDens}
 
In this paper we are concerned with periodic slab-like geometries, which in particular means that we restrict our attention to surfaces given as the graph of a function $\eta: \mathbb{T}^2 \to \mathbb{R}$, where $\mathbb{T}^2 = \mathbb{R}^2 / \mathbb{Z}^2$ is the usual $2-$torus.  This has the benefit of significantly simplifying the differential geometry of the surface.  The area element, the unit normal, and the shape operator (the matrix in coordinates whose trace is the mean curvature) are then, respectively,
\begin{equation}\label{intro_geo_quants}
	\sqrt{1+ \abs{\nabla \eta}^2},\;\;
	\frac{\brac{-\nabla \eta,1}}{\sqrt{1+ \abs{\nabla \eta}^2}}
	, \text{ and }
	\frac{1}{\sqrt{1+ \abs{\nabla \eta}^2}}
	\left(I - \frac{\nabla \eta \otimes \nabla \eta}{1+ \abs{\nabla \eta}^2} \right) \nabla^2 \eta.
\end{equation}

We consider generalized Willmore-type energies that depend both on $\nabla \eta$ and $\nabla^2 \eta$, which allows for a combination of surface stresses of surface tension and bending type.  We specify the energy functional $\will$ through the use of an energy density $f: \mathbb{R}^2 \times \mathbb{R}^{2 \times 2} \to \R$:  
\begin{equation*}
 \will\brac{\eta} = \int_{\mathbb{T}^2} f\brac{\nabla \eta, \nabla^2 \eta}.
\end{equation*}
Note in particular that we neglect to allow the energy density to depend on $\eta$ directly since this is the case for surface energies that only depend on the geometric quantities defined in \eqref{intro_geo_quants}.  We now consider various examples of energies of this type.  Along the way we will record both the first and second variations of the energies.

\textbf{Willmore energy:}
	We consider the Willmore energy, which arises in the Helfrich model of elasticity for a lipid membrane \cite{helfrich}, modeled as a surface $\Sigma$:
	\[
		\will_H = \int_\Sigma C_1 + C_2 {\brac{H-H_0}}^2 + C_3 K
	\]
	for some non-negative constants $C_1, C_2, C_3$ and $H_0$,
	where
	\begin{itemize}
		\item	$H \defeq \tr s$ is the mean curvature,
		\item	$K \defeq \det s$ the Gaussian curvature,
		\item	$h \defeq \frac{\nabla^2 \eta}{A}$ is the scalar extrinsic curvature, or scalar second fundamental form, and
		\item	$s \defeq h^\#$ is the shape operator, i.e. for any vector fields $X$, $Y$, $g\brac{s\brac{X},Y} = h\brac{X,Y}$, where $g$ is the metric on $\Sigma$.
	\end{itemize}
	Note that since $\int_\Sigma K$ is a topological invariant (due to Gauss-Bonnet), and that since $\int_\Sigma 1$ yields a lower-order differential operator (see the surface area discussion below),
	we can simply consider the energy
	\[
		\will = \int_\Sigma \half H^2.
	\]
	We may rewrite this energy as
	\[
		2 \will\brac{\eta}
		= \int_\Sigma H^2
		= \int_{\T^2} H^2 A
		= \int_{\T^2} {\vbrac{g\inv:h}}^2 A
		= \int_{\T^2} \frac{1}{\sqrt{1+\abs{\nabla\eta}^2}} {\vbrac{
			\brac{I - \frac{\nabla\eta\otimes\nabla\eta}{1+\abs{\nabla\eta}^2}} : \nabla^2 \eta
		}}^2,
	\]
	where
	\begin{itemize}
		\item	$A \defeq \sqrt{1+\abs{\nabla\eta}^2}$ is the area element, and
		\item	$g\inv \defeq I - \frac{\nabla\eta\otimes\nabla\eta}{1+\abs{\nabla\eta}^2}$ is the inverse of the metric tensor.
	\end{itemize}
	The first variation is non-trivial to compute, so we skip it here and refer to \cite{willmore_book_93}, where a detailed computation shows that
	\footnote{
		Note that our conventions differ slightly from those used by Willmore: we define the mean curvature to be the sum of the principal curvatures, and not half of that sum,
		and we define the Willmore energy to be half of the square of the mean curvature.
	}
	\[
		\fvr\brac{\eta} = \Delta_\Sigma H + \half H \brac{H^2 - 4K},
	\]
	where
	\[
		\Delta_\Sigma f \defeq -\frac{1}{A} \nabla\cdot\brac{A g\inv\cdot\nabla f}
	\]
	is the Laplace-Beltrami operator on the surface $\Sigma$.  The second variation about a flat equilibrium is the same as the linearization of $\fvr\brac{\eta}$ about a flat equilibrium, and is the bi-Laplacian:
	\[
		\svr{0} = \Delta^2.
	\]
	
\textbf{`Scalar' Willmore energy:}
	Computing the general second variation $\svr{\eta}$ of the Willmore energy presented above is a harrowing experience, and therefore we now discuss a toy model similar to the full Willmore energy but simple enough to yield tractable computations.	This is what we call the `scalar' Willmore energy, namely
	\[
		\will\brac{\eta} = \int_{\T^2} \half m\brac{\nabla\eta} {\vbrac{\Delta\eta}}^2
	\]
	for some smooth $m : \R^2 \to \brac{0,\infty}$ with $m\brac{0} > 0$.
	Simple computations then show that the variations of $\will$ are given by 
	\[
		\fvr\brac{\eta} = \Delta\Bigbrac{
			m\brac{\nabla\eta}\Delta\eta
		} - \nabla\cdot\Bigbrac{
			\half\nabla m\brac{\nabla\eta} {\vbrac{\Delta\eta}}^2
		}
	\]
	and
	\[
		\brac{\svr{\eta}}\phi
		= \Delta \Bigbrac{
			m\brac{\nabla\eta}\Delta\phi
		} + \nabla\cdot \Bigbrac{
			\nabla\brac{\nabla\eta\cdot\nabla m\brac{\nabla\eta}}\cdot\nabla\phi
			- \half {\vbrac{\Delta\eta}}^2 \nabla^2 m \brac{\nabla\eta} \cdot\nabla\phi
		}.
	\]
	In particular, the second variation at the flat equilibrium is
	\[
		\svr{0} = {\brac{\sqrt{m\brac{0}}\Delta}}^2.
	\]
\textbf{Anisotropic Willmore energy:}
	The last surface energy we discuss that yields a fourth-order differential operator is one which, by contrast with the previous two,
	does not linearize to the bi-Laplacian. This surface energy is thus a prototypical example of anisotropic bending energies.
	In particular, we consider the surface energy
	\[
		\will\brac{\eta}  \defeq \half \int_{\T^2} {\vbrac{C\brac{\nabla\eta}:\nabla^2 \eta}}^2
	\]
	for some $C : \R^2 \to \sym\brac{\R^{2\times 2}}$ such that $C\brac{0}$ is positive-definite.
	Then the linearization about the equilibrium of the first variation of $\will$ is
	\[
		\svr{0} = {\brac{C\brac{0}:\nabla^2}}^2.
	\]
	Note that for $C\brac{w} = \sqrt{m\brac{w}} I$ we recover the `scalar' Willmore energy and for
	\[
		C\brac{w}
		= \frac{1}{{\brac{1+\abs{w}^2}}^{1/4}}
		\brac{I - \frac{w \otimes w}{1+\abs{w}^2}}
	\]
	we recover the Willmore energy discussed above.

\textbf{Surface area:}
We now discuss how surface energies related to surface area yield second order differential operators that describe, for example, the forces due to surface tension.  Consider the surface energy
	\[
		\int_\Sigma 1
		= \int_{\mathbb{T}^2} A,
	\]
	where as above (in the discussion of the Willmore energy) $A = \sqrt{1+{\vbrac{\nabla\eta}}^2}$.
	It is well-known that the first variation of the area functional written above is precisely (minus) the mean curvature,
	and that it models the effect of surface tension seeking to minimize the surface area of the free surface.
	More precisely, its variations are given by
	\[
		\fvr\brac{\eta}
		= - H
		= -\nabla\cdot\brac{
			\frac{\nabla\eta}{\sqrt{1+{\vbrac{\nabla\eta}}^2}}
		}
	\]
	and
	\[
		\brac{\svr{\eta}}\phi
		= - \nabla\cdot\brac{
			g\inv \cdot \frac{\nabla\phi}{A}
		}
		= - \nabla\cdot\brac{
			\brac{
				I - \frac{\nabla\eta\otimes\nabla\eta}{1 + {\vbrac{\nabla\eta}}^2}
			}
			\cdot
			\frac{\nabla\phi}{\sqrt{1+{\vbrac{\nabla\eta}}^2}} 
		}.
	\]
	In particular, its linearization about equilibrium is
	\[
		\svr{0} = -\Delta.
	\]

\textbf{Competing effects of surface tension and flexural forces:}  Our general form of the surface energy allows for energetic contributions due to bending as well as area, and as such we will allow for surface stresses of flexural and surface tension type.  Here we record some examples of what these forces look like in terms of the local geometry of the surface.  In particular, we see that there are instances in which the bending and surface tension stresses are in opposition.

 	\begin{itemize}
 		\item	Circular arc: 
			In a circular (one-dimensional) arc surface tension and flexural forces act in opposite directions,
 			the former pushing inward and the latter pushing outward.
 			This is due to the simple observation regarding the scaling of these surface energies:
 			\begin{equation*}
				\mathcal{A} = \int_\Sigma 1	\sim R
				\quad\text{and}\quad
 				\will = \int_\Sigma H^2	\sim \frac{1}{R^2} R = \frac{1}{R}.
 			\end{equation*}
 		\item	Sigmoidal wave: 
 			Surface tension and flexural forces acting in opposite directions can also be seen locally in some more complicated geometries,
 			such as that of the sigmoidal wave shown in Figure \ref{fig:sigmoidalGaussian}.
 			In particular, these forces act in opposite directions to one another at the front and tail of the wave.
 		\item	Gaussian wave:
 			This is another example, shown in Figure \ref{fig:sigmoidalGaussian}, of a geometry in which, locally, surface tension and flexural forces may act in opposite directions.
 	\end{itemize}

 			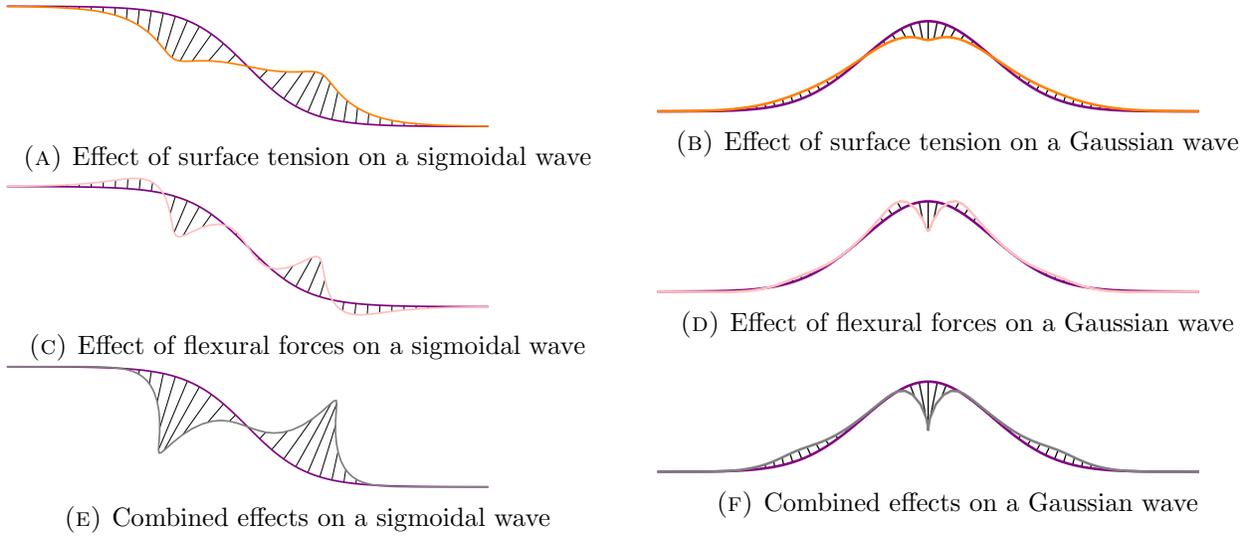
\begin{figure}[h!]
 			\centering
				\pgfmathsetmacro{\picscale}{0.8}
				\pgfmathsetmacro{\la}{0.36}
				\pgfmathsetmacro{\al}{0.58}
				\pgfmathsetmacro{\be}{3.7}
 				\begin{subfigure}{0.45\textwidth}
 					\begin{tikzpicture}[
 						xscale=\picscale, yscale=\picscale,
 						declare function = {
 							surfaceFunction(\x)	= -tanh(\x);
 							meanCurvature(\x)	= (
 											-2*tanh(\x)
 										)/(
 											pow(pow(cosh(\x), 4/3) + (1/pow(cosh(\x), 8/3)), 3/2)
 										);
 							unitNormalX1(\x)	= -1 / (
 											pow(1 + pow(cosh(\x), 4), 1/2)
 										);
 							unitNormalX2(\x)	= -1 / (
 											pow(1 + (1/pow(cosh(\x), 4)), 1/2)
 										);
 						}
 					]
 						\begin{axis}[x=1cm, y=1cm, hide axis]
 							\foreach \xi in {-4, -3.8, ..., 4} {
 								\addplot[domain=0:1, samples = 2]
 									plot(	{
 											\xi + \x*\la*\be*meanCurvature(\xi)*unitNormalX1(\xi)
 										},
 										{
 											surfaceFunction(\xi) + \x*\la*\be*meanCurvature(\xi)*unitNormalX2(\xi)
 										}
 									);
 							}
 							\addplot[domain=-4:4, samples = \bigsample, violet, thick]
 								plot(\x, {surfaceFunction(\x)});
 							\addplot[domain=-4:4, samples = \bigsample, orange, thick]
 								plot(	{
 										\x + \la*\be*meanCurvature(\x)*unitNormalX1(\x)
 									},
 									{
 										surfaceFunction(\x) + \la*\be*meanCurvature(\x)*unitNormalX2(\x)
 									}
 								);
 						\end{axis}
 					\end{tikzpicture}
 					\caption{Effect of surface tension on a sigmoidal wave}
 				\end{subfigure}
				\pgfmathsetmacro{\picscale}{1.2}
				\pgfmathsetmacro{\la}{0.015}
				\pgfmathsetmacro{\al}{0.35}
				\pgfmathsetmacro{\be}{7}
 				\begin{subfigure}{0.45\textwidth}
 					\begin{tikzpicture}[
 						xscale=\picscale, yscale=\picscale,
 						declare function = {
 							surfaceFunction(\x)	= exp(-pow(\x,2));
 							meanCurvature(\x)	= (
 											exp(-pow(\x,2)) * (2 - 4*pow(\x,2))
 										)/(
 											pow(
 												1 + 4*pow(\x,2)*exp(-2*pow(\x,2)),
 												3/2
 											)
 										);
 							unitNormalX1(\x)	= (
 											-2*\x*exp(-pow(\x,2))
 										)/(
 											pow(
 												1 + 4*pow(\x,2)*exp(-2*pow(\x,2)),
 												1/2
 											)
 										);
 							unitNormalX2(\x)	= -1 / (
 											pow(
 												1 + 4*pow(\x,2)*exp(-2*pow(\x,2)),
 												1/2
 											)
 										);
 						}
 					]
 						\begin{axis}[x=1cm, y=1cm, hide axis]
 							\foreach \xi in {-3, -2.9, ..., 3} {
 								\addplot[domain=0:1, samples = 2]
 									plot(	{
 											\xi + \x*\la*\be*meanCurvature(\xi)*unitNormalX1(\xi)
 										},
 										{
 											surfaceFunction(\xi) + \x*\la*\be*meanCurvature(\xi)*unitNormalX2(\xi)
 										}
 									);
 							}
 							\addplot[domain=-3:3, samples = \bigsample, violet, thick]
 								plot(\x, {surfaceFunction(\x)});
 							\addplot[domain=-3:3, samples = \bigsample, orange, thick]
 								plot(	{
 										\x + \la*\be*meanCurvature(\x)*unitNormalX1(\x)
 									},
 									{
 										surfaceFunction(\x) + \la*\be*meanCurvature(\x)*unitNormalX2(\x)
 									}
 								);
 						\end{axis}
 					\end{tikzpicture}
 					\caption{Effect of surface tension on a Gaussian wave}
 				\end{subfigure}
				\pgfmathsetmacro{\picscale}{0.8}
				\pgfmathsetmacro{\la}{0.36}
				\pgfmathsetmacro{\al}{0.58}
				\pgfmathsetmacro{\be}{3.7}
 				\begin{subfigure}{0.45\textwidth}
 					\begin{tikzpicture}[
 						xscale=\picscale, yscale=\picscale,
 						declare function = {
 							surfaceFunction(\x)	= -tanh(\x);
 							firstVarWillmore(\x)	=
 										(
 											-1030*sinh(   \x)
 											+1986*sinh( 3*\x)
 											- 269*sinh( 5*\x)
 											- 217*sinh( 7*\x)
 											-   3*sinh( 9*\x)
 											+     sinh(11*\x)
 										)/(
 											64*pow(
 												pow(cosh(\x), 26/9)
 												+(1/(pow(cosh(\x), 10/9))),
 												9/2
 											)
 										);
 							unitNormalX1(\x)	= -1 / (
 											pow(1 + pow(cosh(\x), 4), 1/2)
 										);
 							unitNormalX2(\x)	= -1 / (
 											pow(1 + (1/pow(cosh(\x), 4)), 1/2)
 										);
 						}
 					]
 						\begin{axis}[x=1cm, y=1cm, hide axis]
 							\foreach \xi in {-4, -3.8, ..., 4} {
 								\addplot[domain=0:1, samples = 2]
 									plot(	{
 											\xi + \x*\la*\al*firstVarWillmore(\xi)*unitNormalX1(\xi)
 										},
 										{
 											surfaceFunction(\xi) + \x*\la*\al*firstVarWillmore(\xi)*unitNormalX2(\xi)
 										}
 									);
 							}
 							\addplot[domain=-4:4, samples = \bigsample, violet, thick]
 								plot(\x, {surfaceFunction(\x)});
 							\addplot[domain=-4:4, samples = \vbigsample, pink, thick]
 								plot(	{
 										\x + \la*\al*firstVarWillmore(\x)*unitNormalX1(\x)
 									},
 									{
 										surfaceFunction(\x) + \la*\al*firstVarWillmore(\x)*unitNormalX2(\x)
 									}
 								);
 						\end{axis}
 					\end{tikzpicture}
 					\caption{Effect of flexural forces on a sigmoidal wave}
 				\end{subfigure}
				\pgfmathsetmacro{\picscale}{1.2}
				\pgfmathsetmacro{\la}{0.015}
				\pgfmathsetmacro{\al}{0.35}
				\pgfmathsetmacro{\be}{7}
 				\begin{subfigure}{0.45\textwidth}
 					\begin{tikzpicture}[
 						xscale=\picscale, yscale=\picscale,
 						declare function = {
 							surfaceFunction(\x)	= exp(-pow(\x,2));
 							firstVarWillmore(\x)	= (
 											   8*          exp(5*pow(\x,2))*(  5 - 126*pow(\x,2) + 284*pow(\x,4) - 168*pow(\x,6))
 											+ 64*pow(\x,2)*exp(3*pow(\x,2))*(-15 +  36*pow(\x,2) -  44*pow(\x,4) +  48*pow(\x,6))
 											+  8*          exp(7*pow(\x,2))*(  3 -  12*pow(\x,2) +   4*pow(\x,4))
 										)/(
 											(
 												pow(
 													exp(2*pow(\x,2))
 													+ 4*pow(\x,2),
 													4
 												)
 											)*(
 												pow(
 													1 + 4*pow(\x,2)*exp(-2*pow(\x,2)),
 													1/2
 												)
 											)
 										);
 							unitNormalX1(\x)	= (
 											-2*\x*exp(-pow(\x,2))
 										)/(
 											pow(
 												1 + 4*pow(\x,2)*exp(-2*pow(\x,2)),
 												1/2
 											)
 										);
 							unitNormalX2(\x)	= -1 / (
 											pow(
 												1 + 4*pow(\x,2)*exp(-2*pow(\x,2)),
 												1/2
 											)
 										);
 						}
 					]
 						\begin{axis}[x=1cm, y=1cm, hide axis]
 							\foreach \xi in {-3, -2.9, ..., 3} {
 								\addplot[domain=0:1, samples = 2]
 									plot(	{
 											\xi + \x*\la*\al*firstVarWillmore(\xi)*unitNormalX1(\xi)
 										},
 										{
 											surfaceFunction(\xi) + \x*\la*\al*firstVarWillmore(\xi)*unitNormalX2(\xi)
 										}
 									);
 							}
 							\addplot[domain=-3:3, samples = \bigsample, violet, thick]
 								plot(\x, {surfaceFunction(\x)});
 							\addplot[domain=-3:3, samples = \vbigsample, pink, thick]
 								plot(	{
 										\x + \la*\al*firstVarWillmore(\x)*unitNormalX1(\x)
 									},
 									{
 										surfaceFunction(\x) + \la*\al*firstVarWillmore(\x)*unitNormalX2(\x)
 									}
 								);
 						\end{axis}
 					\end{tikzpicture}
 					\caption{Effect of flexural forces on a Gaussian wave}
 				\end{subfigure}
				\pgfmathsetmacro{\picscale}{0.8}
				\pgfmathsetmacro{\la}{0.36}
				\pgfmathsetmacro{\al}{0.58}
				\pgfmathsetmacro{\be}{3.7}
 				\begin{subfigure}{0.45\textwidth}
 					\begin{tikzpicture}[
 						xscale=\picscale, yscale=\picscale,
 						declare function = {
 							surfaceFunction(\x)	= -tanh(\x);
 							meanCurvature(\x)	= (
 											-2*tanh(\x)
 										)/(
 											pow(pow(cosh(\x), 4/3) + (1/pow(cosh(\x), 8/3)), 3/2)
 										);
 							firstVarWillmore(\x)	=
 										(
 											-1030*sinh(   \x)
 											+1986*sinh( 3*\x)
 											- 269*sinh( 5*\x)
 											- 217*sinh( 7*\x)
 											-   3*sinh( 9*\x)
 											+     sinh(11*\x)
 										)/(
 											64*pow(
 												pow(cosh(\x), 26/9)
 												+(1/(pow(cosh(\x), 10/9))),
 												9/2
 											)
 										);
 							unitNormalX1(\x)	= -1 / (
 											pow(1 + pow(cosh(\x), 4), 1/2)
 										);
 							unitNormalX2(\x)	= -1 / (
 											pow(1 + (1/pow(cosh(\x), 4)), 1/2)
 										);
 						}
 					]
 						\begin{axis}[x=1cm, y=1cm, hide axis]
 							\foreach \xi in {-4, -3.8, ..., 4} {
 								\addplot[domain=0:1, samples = 2]
 									plot(	{
 											\xi
 											+ \x*\la*\al*firstVarWillmore(\xi)*unitNormalX1(\xi)
 											+ \x*\la*\be*meanCurvature(\xi)*unitNormalX1(\xi)
 										},
 										{
 											surfaceFunction(\xi)
 											+ \x*\la*\al*firstVarWillmore(\xi)*unitNormalX2(\xi)
 											+ \x*\la*\be*meanCurvature(\xi)*unitNormalX2(\xi)
 										}
 									);
 							}
 							\addplot[domain=-4:4, samples = \bigsample, violet, thick]
 								plot(\x, {surfaceFunction(\x)});
 							\addplot[domain=-4:4, samples = \vbigsample, gray, thick]
 								plot(	{
 										\x
 										+ \la*\al*firstVarWillmore(\x)*unitNormalX1(\x)
 										+ \la*\be*meanCurvature(\x)*unitNormalX1(\x)
 									},
 									{
 										surfaceFunction(\x)
 										+ \la*\al*firstVarWillmore(\x)*unitNormalX2(\x)
 										+ \la*\be*meanCurvature(\x)*unitNormalX2(\x)
 									}
 								);
 						\end{axis}
 					\end{tikzpicture}
 					\caption{Combined effects on a sigmoidal wave}
 				\end{subfigure}
				\pgfmathsetmacro{\picscale}{1.2}
				\pgfmathsetmacro{\la}{0.015}
				\pgfmathsetmacro{\al}{0.35}
				\pgfmathsetmacro{\be}{7}
 				\begin{subfigure}{0.45\textwidth}
 					\begin{tikzpicture}[
 						xscale=\picscale, yscale=\picscale,
 						declare function = {
 							surfaceFunction(\x)	= exp(-pow(\x,2));
 							meanCurvature(\x)	= (
 											exp(-pow(\x,2)) * (2 - 4*pow(\x,2))
 										)/(
 											pow(
 												1 + 4*pow(\x,2)*exp(-2*pow(\x,2)),
 												3/2
 											)
 										);
 							firstVarWillmore(\x)	= (
 											   8*          exp(5*pow(\x,2))*(  5 - 126*pow(\x,2) + 284*pow(\x,4) - 168*pow(\x,6))
 											+ 64*pow(\x,2)*exp(3*pow(\x,2))*(-15 +  36*pow(\x,2) -  44*pow(\x,4) +  48*pow(\x,6))
 											+  8*          exp(7*pow(\x,2))*(  3 -  12*pow(\x,2) +   4*pow(\x,4))
 										)/(
 											(
 												pow(
 													exp(2*pow(\x,2))
 													+ 4*pow(\x,2),
 													4
 												)
 											)*(
 												pow(
 													1 + 4*pow(\x,2)*exp(-2*pow(\x,2)),
 													1/2
 												)
 											)
 										);
 							unitNormalX1(\x)	= (
 											-2*\x*exp(-pow(\x,2))
 										)/(
 											pow(
 												1 + 4*pow(\x,2)*exp(-2*pow(\x,2)),
 												1/2
 											)
 										);
 							unitNormalX2(\x)	= -1 / (
 											pow(
 												1 + 4*pow(\x,2)*exp(-2*pow(\x,2)),
 												1/2
 											)
 										);
 						}
 					]
 						\begin{axis}[x=1cm, y=1cm, hide axis]
 							\foreach \xi in {-3, -2.9, ..., 3} {
 								\addplot[domain=0:1, samples = 2]
 									plot(	{
 											\xi
 											+ \x*\la*\al*firstVarWillmore(\xi)*unitNormalX1(\xi)
 											+ \x*\la*\be*meanCurvature(\xi)*unitNormalX1(\xi)
 										},
 										{
 											surfaceFunction(\xi)
 											+ \x*\la*\al*firstVarWillmore(\xi)*unitNormalX2(\xi)
 											+ \x*\la*\be*meanCurvature(\xi)*unitNormalX2(\xi)
 										}
 									);
 							}
 							\addplot[domain=-3:3, samples = \bigsample, violet, thick]
 								plot(\x, {surfaceFunction(\x)});
 							\addplot[domain=-3:3, samples = \vbigsample, gray, thick]
 								plot(	{
 										\x
 										+ \la*\al*firstVarWillmore(\x)*unitNormalX1(\x)
 										+ \la*\be*meanCurvature(\x)*unitNormalX1(\x)
 									},
 									{
 										surfaceFunction(\x)
 										+ \la*\al*firstVarWillmore(\x)*unitNormalX2(\x)
 										+ \la*\be*meanCurvature(\x)*unitNormalX2(\x)
 									}
 								);
 						\end{axis}
 					\end{tikzpicture}
 					\caption{Combined effects on a Gaussian wave}
 				\end{subfigure}
 				\caption{
					The purple curve is the profile of a free surface $\Sigma$ given as the graph of $\eta = \tanh$ on the left and of $\eta\brac{x} = e^{-x^2 / 2}$ on the right.
 					The black segments show the force $\delta \will\brac{\eta} \nu_\Sigma$ exercised on the free surface
 					corresponding to a surface energy $\will$.
					In (A) and (B), $\will = \int_\Sigma 1$; in (C) and (D), $\will = \int_\Sigma H^2$; and in (E) and (F), $\will = \int_\Sigma \alpha + \beta H^2$
 					for some $\alpha, \beta > 0$.
					The other curve (orange in (A) and (B), pink in (C) and (D), and grey in (E) and (F)) illustrates the new profile of the free surface
 					after application of the force $\delta \will\brac{\eta}\nu_\Sigma$.
 				}
 				\label{fig:sigmoidalGaussian}
 			\end{figure}

\subsection{Fluid equations}

We now consider a slab of periodic fluid occupying the moving domain 
\begin{equation*}
	\Omega\brac{t} \defeq \setdef{x = \brac{\bar{x},x_3} \in \T^2 \times \R}{-b < x_3 <\eta\brac{t,\bar{x}}}
\end{equation*}
for an unknown height function $\eta :   [0,\infty) \times \T^2 \to (-b,\infty)$.  The lower boundary of $\Omega(t)$ is the rigid unmoving interface 
\begin{equation*}
	\Sigma_b \defeq \setdef{x \in \T^2 \times \R}{x_3 = -b},
\end{equation*}
while the upper boundary is the moving interface 
\begin{equation*}
	\Sigma\brac{t} := \setdef{x \in \T^2 \times \R }{ x_3 = \eta\brac{t,\bar{x}}}.
\end{equation*}

We assume that the fluid is subject to a uniform gravitational field of strength $g \in \R$ acting perpendicularly to $\Sigma_b$.  Note in particular that we do not require $g \ge 0$: more will be said about this below in the latter part of Section \ref{sec:assumSurfEnerDensity}.
We assume that the free interface is subject to surface stresses generated by the energy 
\begin{equation}\label{intro_energy_form}
 \will\brac{\eta} = \int_{\mathbb{T}^2} f\brac{\nabla \eta, \nabla^2 \eta} 
\end{equation}
for a function  $f : \R^2 \times \R^{2 \times 2} \to\R$ satisfying the assumptions enumerated below in Section \ref{sec:assumSurfEnerDensity}.  If $\nu$ denotes the unit normal pointing \emph{out} of $\Omega(t)$, then the surface stress is
\begin{equation}\label{intro_surface_stress}
 -\fvr\brac{\eta} (-\nu) = \fvr\brac{\eta} \nu,
\end{equation}
i.e. the magnitude of the stress is $-\fvr\brac{\eta}$ but the direction is $-\nu$, which indicates that the surface stress acts on the fluid.  This form of $\will$ allows us to consider a generalized mixture of bending and surface tension stresses.  Due to this general form, we will not attribute the source of the energy (and hence the stress) to any particular model, but as elaborated on above in Section \ref{sec:intro_se_discussion}, such an energy would arise if we viewed the surface as a thin biological membrane or as a thin layer of ice.  Our assumptions on  $f$ will always require that $\fvr\brac{\eta}$ is a fourth-order differential operator, typically of quasilinear form.

We will assume that the fluid is incompressible and viscous, which means that we can describe its state by specifying its velocity $u\brac{t,\cdot} : \Omega(t) \to \R^3$ and pressure $p\brac{t,\cdot} : \Omega(t) \to \R$.  For simplicity we will assume that the fluid density and viscosity are normalized to unity.  The equations of motion are then the free boundary Navier-Stokes equations coupled to surface stresses of the form \eqref{intro_surface_stress} generated by the free energy \eqref{intro_energy_form}.  These read
\begin{subnumcases}{}
    \pdt u + \brac{u\cdot\nabla}u = - \nabla p + \Delta u						&in $\Omega\brac{t}$, \label{NS_euler_s}\\
    \nabla\cdot u = 0											&in $\Omega\brac{t}$,\\
    \brac{ p I - \symgrad u} \nu = \Bigbrac{\fvr\brac{\eta} + g \eta }  \nu				&on $\Sigma\brac{t}$,\\
    \pdt\eta = \brac{u \cdot \nu} \sqrt{1 + \abs{\nabla \eta}^2}					&on $\Sigma\brac{t}$, and\\
    u = 0												&on $\Sigma_b$, \label{NS_euler_e}
\end{subnumcases}
where  
\begin{equation}
(\symgrad u)_{ij} = \partial_i u_j + \partial_j u_i 
\end{equation}
is the symmetrized gradient and $I$ is the $3 \times 3$ identity matrix.  The first two equations are the usual incompressible Navier-Stokes system, the third is the balance of stresses on the free interface, the fourth is the kinematic transport equation, and the fifth is the no-slip condition at the rigid interface.  Note that what we call the pressure $p$ is really the difference between the standard pressure $\bar{p}$ and hydrostatic pressure $-g x_3$, i.e. $p = \bar{p} + g x_3$.  Making this substitution in the first and third equations reveals that the gravitational term is originally a bulk force acting in $\Omega(t)$.

Sufficiently regular solutions to \eqref{NS_euler_s}--\eqref{NS_euler_e} obey the following equations: the energy-dissipation identity
\begin{equation}\label{intro_ed}
 \frac{d}{dt} \left( \int_{\Omega(t)} \frac{1}{2} \abs{u}^2 + \int_{\mathbb{T}^2} \frac{g}{2} \abs{\eta}^2  + \will(\eta) \right) + \int_{\Omega(t)} \frac{1}{2} \abs{\symgrad u}^2 =0,
\end{equation}
and the mass conservation identity 
\begin{equation}\label{intro_mass}
 \frac{d}{dt} \int_{\mathbb{T}^2} \eta = 0.
\end{equation}
The first term in parentheses in \eqref{intro_ed} is the kinetic energy of the fluid, the second is the total gravitational potential energy stored in the fluid, and the third is the surface energy \eqref{intro_energy_form}.  The term outside parentheses is the usual viscous dissipation, which in particular forces the total energy (the sum of the three terms) to be non-increasing in time.  The equation \eqref{intro_mass} is understood as the integral form of mass conservation since $b + \int_{\T^2} \eta(\cdot,t)$ is the mass of the fluid body at time $t \ge 0$.  We will assume that the parameter $b$ is chosen such that the initial mass of fluid is $b$, which means that 
\begin{equation}\label{intro_zero_avg}
	\int_{\mathbb{T}^2} \eta_0 = 0 \text{ and hence }  \int_{\mathbb{T}^2} \eta\brac{t,\cdot} = 0 \text{ for } t \ge 0.
\end{equation}

From the no-slip condition, Korn's inequality (see Proposition \ref{eqEstKorn}), and  \eqref{intro_ed} we conclude that any equilibrium (time-independent) solutions must satisfy $u=0$.  In turn, this, \eqref{NS_euler_s}--\eqref{NS_euler_e}, and  \eqref{intro_zero_avg} imply that $p=0$, which reduces to $\eta$ solving 
\begin{equation}\label{intro_eta_equation}
	\fvr\brac{\eta} + g \eta = 0.
\end{equation}
It's clear that $\eta =0$ is a solution to this, but it does not follow from our assumptions on the energy density $f$ (enumerated below in Section \ref{sec:assumSurfEnerDensity}) that $0$ is the only solution to this equation.  However, our assumptions do require that $0$ is a local minimum of the total surface energy  (the sum of $\will$ and the gravitational potential $\mathcal{P}$)
\begin{equation}\label{intro_T_energy}
\will\brac{\eta} + \mathcal{P}\brac{\eta} := \will\brac{\eta} +  \int_{\T^2} \frac{g}{2} \abs{\eta}^2 
\end{equation}
and that the second variation of $\will + \mathcal{P}$ is positive definite at $0$, when restricted to functions of zero average.  It is a simple matter to check that \eqref{intro_eta_equation} corresponds to the Euler-Lagrange equation $\delta(\will + \mathcal{P})\brac{\eta}=0$, which means that $0$ is an isolated critical point of $\will + \mathcal{P}$.  Then $\eta =0$ is the only solution to \eqref{intro_eta_equation} within an open set containing $0$.  Thus, there is a locally unique equilibrium corresponding to a flat slab of quiescent fluid.  Our main goal in this paper is to show that this equilibrium solution is asymptotically stable and to characterize the rate of decay to equilibrium.

Much is known about problems of the form \eqref{NS_euler_s}--\eqref{NS_euler_e} when $\will$ is a multiple $\sigma \ge 0$ of the area function, $g >0$, and the cross-section is either periodic ($\T^2$) or infinite ($\R^2$).  The case $\sigma >0$ corresponds to surface tension, and $\sigma =0$ corresponds to no surface tension.  Beale \cite{beale_1981} proved the first local well-posedness results for the infinite cross section without surface tension.  Beale \cite{beale_1984} also proved global existence of solutions near equilibrium for the infinite problem with surface tension.  Beale--Nishida\cite{beale_nishida} then proved that these global solutions decay at an algebraic rate.  The existence of global solutions with and without surface tension was also studied by Tani--Tanaka \cite{tani_tanaka}, but no decay information was obtained.  Guo--Tice proved that for the infinite problem without surface tension, small data leads to global solutions that decay algebraically.   For the periodic problem without surface tension, Hataya \cite{hataya} constructed global solutions decaying at a fixed algebraic rate, and Guo--Tice \cite{guo_tice_periodic} proved that solutions decay almost exponentially, with the decay rate determined by the data.  Nishida--Teramoto--Yoshihara \cite{nishida_teramoto_yoshihara} proved that the periodic problem with surface tension leads to global solutions near equilibrium that decay exponentially. Tan--Wang \cite{tan_wang} established a sort of continuity result, proving that the global solutions with surface tension converge to the global solutions without surface tension as $\sigma \to 0$.  

As mentioned in Section \ref{sec:intro_se_discussion}, there are results on the local existence of solutions to models coupling incompressible Navier-Stokes to free boundaries with elastic and bending stresses: \cite{cheng-coutand-shkoller, cheng_shkoller_2010,muha_canic_2015,muha-canic-hemodynamics}.  However, to the best of our knowledge, there are no global existence or asymptotic stability results on the problem \eqref{NS_euler_s}--\eqref{NS_euler_e} with $\will$ combining bending and surface tension stresses.

\section{Main result}
\subsection{Reformulation in a fixed domain}\label{sec:reformulation}
In order to solve the problem \eqref{NS_euler_s}--\eqref{NS_euler_e} we flatten the domain, which has the benefit of allowing us to work with a domain that is no longer time-dependent.
More precisely, we move from the Eulerian domain $\Omega\brac{t}$ to the fixed equilibrium domain $\Omega \defeq \T^2 \times \brac{-b,0}$
via a map $\Phi : \cobrac{0,T} \times \T^2 \times \R \to \T^2 \times \R$
such that for every $0\leqslant t < T$, $\Phi(t,\cdot) : \Omega \to \Omega\brac{t}$ is a diffeomorphism that maps the lower/upper boundary of $\Omega$ to the upper/lower boundary of $\Omega(t)$.  

To precisely define this map we need two tools.  The first is any smooth cutoff function $\rchi:\Omega\to\R$ such that
$\rchi = 1$ on $\Sigma$ and $\rchi = 0$ on $\Sigma_b$.  For instance, we can define $\rchi\brac{x_3} = 1 + \frac{x_3}{b}$.  The second tool is the harmonic extension map $\ext$, the precise definition of which can be found in Section \ref{sec:harmExt}.  For $0 \le t <T$, the extension allows us to extend $\eta(t,\cdot) : \T^2 \to \R$ to the function $\ext\eta(t,\cdot) : \Omega\to\R$, defined in the bulk.  The extension is done to help with regularity issues when taking the trace of $\Phi$ onto $\Sigma$.

With these tools in hand, we define 
\begin{equation}\label{intro_Phi_def}
\Phi(t,\cdot) = \id + \ext\eta(t,\cdot) \rchi e_3 
\end{equation}
for the choice of cutoff $\rchi$ as above.  An important observation is that if $\eta$ is sufficiently small (which is made precise in item (2) of Remark \ref{rmk:smallEnergyRegime}), then $\Phi(t,\cdot)$ is a diffeomorphism onto $\Omega\brac{t}$.  In particular, if we denote by $\Sigma = \T^2 \times \cbrac{0}$ the upper boundary of the fixed domain $\Omega$, then $\Phi(t, \brac{\Sigma} ) = \Sigma\brac{t}$ and $\Phi(t,\cdot) = \id$ on $\Sigma_b$: see Figure \ref{fig:cartoon_flattening}.
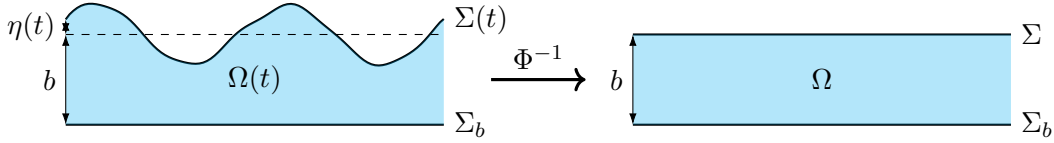
\begin{figure}[h!]
\centering
\begin{tikzpicture}[xscale=0.8,yscale=0.8]
	\draw [name path=St, thick, domain=0:2*pi, samples=\bigsample]
		plot (\x,{
			.5*sin((2*\x + pi/6) r)
			+ (
			0.02*sin((3*\x + pi/6) r)
				+ 0.01*sin((5*\x + pi/6) r)
				+ 0.02*sin((7*\x + pi/6) r)
			)*pow( \x*(2*pi-\x),1/5)
		});
	\draw [name path=Sb, thick] (0,-1.5) to (2*pi,-1.5);
	\tikzfillbetween[of=St and Sb]{cyan, opacity = 0.3}
	\draw [dashed] (0,0) to (2*pi,0);
	\draw [{Latex[length=1.5mm,width=1mm]}-{Latex[length=1.5mm,width=1mm]}]
		(0,0) to (0,0.25);
	\draw [{Latex[length=1.5mm,width=1mm]}-{Latex[length=1.5mm,width=1mm]}]
		(0,-1.5) to (0,0);
	\node [left] at (0,0.125) {$\eta(t)$};
	\node [left] at (0,-0.75) {$b$};
	\node [right] at (2*pi,0.25) {$\Sigma(t)$};
	\node [right] at (2*pi,-1.5) {$\Sigma_b$};
	\node at (pi,-0.75) {$\Omega(t)$};
		\draw [very thick, ->] (2.25*pi,-0.75) to (2.75*pi,-0.75);
		\node [above] at (2.5*pi,-0.75) {${\mathrm{\Phi}}^{-1}$};
	\draw [name path=Sflat, thick, domain=3*pi:5*pi]
	plot (\x, {0});
	\draw [name path=Sbflat, thick] (0+3*pi,-1.5) to (2*pi+3*pi,-1.5);
	\tikzfillbetween[of=Sflat and Sbflat]{cyan, opacity = 0.3}
	\draw [{Latex[length=1.5mm,width=1mm]}-{Latex[length=1.5mm,width=1mm]}]
		(3*pi,-1.5) to (3*pi,0);
	\node [left] at (3*pi,-0.75) {$b$};
	\node [right] at (2*pi+3*pi,0) {$\Sigma$};
	\node [right] at (2*pi+3*pi,-1.5) {$\Sigma_b$};
	\node at (pi+3*pi,-0.75) {$\Omega$};
\end{tikzpicture}
\caption{A cartoon of the diffeomorphism fixing the domain}
\label{fig:cartoon_flattening}
\end{figure}
Any function $f$ defined on the Eulerian domain $\Omega\brac{t}$ thus gives rise to a function $F \defeq f\circ\Phi$ defined on the fixed domain $\Omega$.
In particular, the manifestations on $\Omega$ of the temporal and spatial derivatives of $f$ are given by
\begin{equation*}
\left\{
\begin{aligned}
	&\nabla^\geo F \defeq \nabla\brac{F\circ\Phi\inv}\circ\Phi\text{ and}
	\\
	&\pdt^\geo F \defeq \pdt\brac{F\circ\Phi\inv}\circ\Phi
\end{aligned}
\right.
\end{equation*}
i.e. $f = F\circ\Phi\inv$ and $\nabla f = \brac{\nabla^\geo F}\circ\Phi\inv$ (and similarly for temporal derivatives).
The differential operators $\nabla^\geo$ and $\pdt^\geo$ are called $\geo$-differential operators.
In more concrete terms, the $\geo$-differential operators may be written as
\begin{equation*}
\left\{
\begin{aligned}
	&\nabla^\geo = \geo \cdot \nabla \text{ and}
	\\
	&\pdt^\geo = \pdt - \brac{\pdt\Phi}\cdot\nabla^\geo
\end{aligned}
\right.
\end{equation*}
for $\geo \defeq {\brac{\nabla\Phi}}^{-T}$.
Similarly, we define the $\geo$-versions of the symmetrized gradient and of the Laplacian via
$\symgrad^\geo F \defeq \nabla^\geo F + {\brac{\nabla^\geo F}}^T$ and $\Delta^\geo F \defeq \nabla^\geo \cdot \brac{\nabla^\geo F}$.  We may now reformulate \eqref{NS_euler_s}--\eqref{NS_euler_e} as a system of PDEs on the fixed domain $\Omega$.
Indeed, solutions $\x^* = \brac{v,q,\eta}$ on $\Omega\brac{t}$ of \eqref{NS_euler_s}--\eqref{NS_euler_e} correspond to solutions
$\x = \brac{v\circ\Phi, q\circ\Phi, \eta} \eqdef \brac{u, p, \eta}$ on $\Omega$ of
\begin{subnumcases}{}
	\pdt^\geo u + \brac{u\cdot\nabla^\geo}u = - \nabla^\geo p + \Delta^\geo u					&in $\Omega$,\label{NS_fixed_s}\\
	\nabla^\geo\cdot u = 0												&in $\Omega$, \label{NS_fixed_div} \\
	\brac{p I - \symgrad^\geo u} \nu_{\partial\Omega} = \Bigbrac{\fvr\brac{\eta} + g\eta} \nu_{\partial\Omega}	&on $\Sigma$, \label{NS_fixed_dyn} \\
	\pdt\eta = u \cdot \nu_{\partial\Omega} \sqrt{1 + \abs{\nabla \eta}^2}						&on $\Sigma$, and \label{NS_fixed_kin}\\
	u = 0														&on $\Sigma_b$ \label{NS_fixed_e}.
\end{subnumcases}
The rest of this paper is therefore concerned with the study of this system.

\subsection{Assumptions on the surface energy density}\label{sec:assumSurfEnerDensity}
	We now make precise the assumptions that we impose on the surface energy density $f: \R^2 \times \R^{2\times 2} \to \R$  throughout the paper.  We assume the following.
	\begin{enumerate}
		\item	$f$ is smooth, i.e. infinitely differentiable. If we keep track of the regularity needed on $f$ at the lowest level of regularity to close the estimates in this paper, then
			we only need $f\in C^{7,1}$. However, no effort has been made to make this regularity optimal in light of the fact that if we sought smooth solutions, then $f$ would have to be smooth as well.
		\item	$f\brac{0,0} = 0$ and $\nabla f \brac{0,0} = 0$.
			This is an assumption that can be made without loss of generality
			because we may reduce the general case to this one by adding null a Lagrangian and a constant to the surface energy.
			Indeed, for an arbitrary $\tilde{f}$,
			we may define 
			\begin{equation*}
			f\brac{w,M} \defeq \tilde{f}\brac{w,M} - \tilde{f}\brac{0,0} + \nabla_w \tilde{f} \brac{0,0} \cdot w + \nabla_M \tilde{f}\brac{0,0} : M  
			\end{equation*}
			such that indeed, $f\brac{0}=0$, $\nabla f\brac{0} = 0$ and for $\eta : \T^2 \to \R$ sufficiently regular  we have that
			\begin{equation*}
			\begin{split}
				\int_{\mathbb{T}^2} f\brac{\nabla \eta, \nabla^2 \eta}
				&= \int_{\mathbb{T}^2} \tilde{f}\brac{\nabla \eta, \nabla^2 \eta}
						- \int_{\mathbb{T}^2} \tilde{f}\brac{0}
				- \int_{\mathbb{T}^2}  \nabla_w \tilde{f} \brac{0,0} \cdot \nabla \eta + \nabla_M \tilde{f}\brac{0,0} : \nabla^2 \eta 	\\
				& =\int_{\mathbb{T}^2} \tilde{f}\brac{\nabla \eta, \nabla^2 \eta} - \int_{\mathbb{T}^2} \tilde{f}\brac{0},
			\end{split}
			\end{equation*}
			i.e. the surface energies defined by $f$ and $\tilde{f}$ only differ by an irrelevant constant.  Note that the third integral on the right side of the first equality vanishes by integrating by parts.
		\item	The Hessian of $f$ satisfies
			\begin{equation}\label{f_assume_hessian}
				\nabla^2_{M,M} f \brac{0} \bullet \brac{k^{\otimes 4}}
				- \nabla^2_{\p,\p} f \brac{0} \bullet \brac{k^{\otimes 2}}
				+ g
				\gtrsim
				\abs{k}^4
			\end{equation}
			for all $k\in\Z^2\setminus\cbrac{0}$, i.e. $\svr{0} + g$ is strictly elliptic over functions of average zero.  See Section \ref{sec:dynBC} for a more detailed discussion of the ellipticity of $\svr{0}+g$.
	\end{enumerate}
Note in particular that our assumptions on $f$ do not necessarily imply that $\will$ is positive definite.  However, the third assumption requires that the total surface energy $\will + \mathcal{P}$ defined in \eqref{intro_T_energy} is positive definite for sufficiently small perturbations of $0$.
	
The third assumption can also be understood as saying that flexural effects dominate. 	For example, if we consider
	\[
		\will = \int_\Sigma \alpha + \beta H^2,
	\]
	then $\svr{0} = -\alpha\Delta + \beta\Delta^2$.   If $\alpha, g < 0$, then upon applying the Fourier transform we see that
	\[
		{\brac{\svr{0} + g}}^{\wedge} \brac{k}
		= 16 \pi^4 \beta \abs{k}^4
		+ 4 \pi^2 \alpha \abs{k}^2
		+ g
		\geqslant \brac{
			16 \pi^4 \beta + 4 \pi^2 \alpha + g
		} \abs{k}^4
	\]
	since $\abs{k} \geqslant 1$ for all $k\in\Z^2 \setminus\cbrac{0}$.
	In particular, even if $\alpha, g < 0$, as long as $16\pi^4 \beta > - \brac{4\pi^2 \alpha + g}$,
	then $\svr{0} + g$ is strictly elliptic over functions of average zero.
	In physically meaningful terms (c.f. Figure \ref{fig:negGravity} for a sketch), this means that sufficiently strong flexural effects dominate over adverse surface tension and gravity effects.  In particular, we can allow for $g <0$ in general.
	\begin{figure}[h!]
	\centering
	\begin{tikzpicture}
		\draw [name path=top, thick] (-1.5,0) to (1.5,0);
		\draw [name path=bottom, thick, domain=-1:1, samples=\bigsample]
			plot(\x,{
					1/2 * (\x+1) * (\x-1) * (1 - pow(\x,2))
				});
		\tikzfillbetween[of=bottom and top]{cyan, opacity=0.3}
		\draw [->] (1,-0.25) -> (1,-0.75);
		\node [right] at (1,-0.5) {$g$};
	\end{tikzpicture}
	\caption{Sufficiently strong flexural effects dominate adverse gravitational effects.}
	\label{fig:negGravity}
	\end{figure}
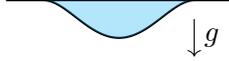

\subsection{Statement of the main result}
In order state the main result, it is convenient to introduce the notion of an admissible initial condition and to introduce the energy and dissipation functionals.

An admissible initial condition is, loosely speaking, a pair $\brac{u_0, \eta_0}$ such that $u_0$ is incompressible,
the boundary conditions are satisfied, an appropriate compatibility condition holds, and $\eta_0$ has average zero.
The precise definition of an admissible initial condition may be found in Definition \ref{def:admissibleIC}, and a more detailed discussion of the compatibility condition is included in Remark \ref{rmk:defAdmIC}.

Now let us introduce the energy and dissipation functionals. Given a triple $\x = \brac{u, p, \eta}$, the associated energy and dissipation functionals are
\begin{equation*}
	\enimp\brac{\x} \defeq
		\normtypdom{u}{H}{2}{\Omega}^2
		+ \normtypdom{\pdt u}{L}{2}{\Omega}^2
		+ \normtypdom{p}{H}{1}{\Omega}^2
		+ \normtypdom{\eta}{H}{{9/2}}{\T^2}^2
		+ \normtypdom{\pdt\eta}{H}{2}{\T^2}^2
\end{equation*}
and
\begin{equation*}
	\dsimp\brac{\x} \defeq
		\normtypdom{u}{H}{3}{\Omega}^2
		+ \normtypdom{\pdt u}{H}{1}{\Omega}^2
		+ \normtypdom{p}{H}{2}{\Omega}^2
		+ \normtypdom{\eta}{H}{{11/2}}{\T^2}^2
		+ \normtypdom{\pdt\eta}{H}{{5/2}}{\T^2}^2
		+ \normtypdom{\pdt^2 \eta}{H}{{1/2}}{\T^2}^2,
\end{equation*}
respectively.
We will sometimes abuse notation slightly and write $\enimp\brac{t} \defeq \enimp\brac{\x\brac{t}}$ and $\dsimp\brac{t} \defeq \dsimp\brac{\x\brac{t}}$
when it is clear from context which triple $\x$ is being used.
We may now state the main result of this paper.
\begin{theorem}
	\label{thm:main-intro}
	Assume that $f$ satisfies the conditions enumerated in Section \ref{sec:assumSurfEnerDensity}. Then there exist universal constants $C, \lambda, \epsilon > 0$ such that for every admissible initial condition $\brac{u_0, \eta_0}$ satisfying
	\begin{equation}\label{eq:main-intro_small}
		\normtypdom{\eta_0}{H}{9/2}{\T^2}^2
		+ \normtypdom{u_0}{H}{2}{\Omega}^2
		+ \normtypdom{
				u_0 \cdot \brac{-\nabla\eta_0,1}
			}{H}{2}{\Sigma}^2
		\leqslant \epsilon
	\end{equation}
	there exists a unique solution $\x = \brac{u,p,\eta}$ of \eqref{NS_fixed_s}--\eqref{NS_fixed_e} on $\cobrac{0,\infty}$ such that
	\[
		\sup_{t \geqslant 0} \enimp\brac{t} e^{\lambda t}
		+ \int_0^\infty \dsimp\brac{t} e^{\lambda t} dt
		\leqslant C \enimp\brac{0}.
	\]
\end{theorem}
Note that requiring the smallness of the third term in \eqref{eq:main-intro_small} comes from the compatibility condition (c.f. Section \ref{sec:gwpDecay} for a more detailed discussion).  Theorem \ref{thm:main-intro} is proved in Section \ref{sec:gwpDecay} in the somewhat more precise form of Theorem \ref{thm:gwpDecay}.  Theorem \ref{thm:main-intro} guarantees that $\eta$ is regular and small enough to transform the solution back to the Eulerian system, which then gives rise to a global decaying solution to \eqref{NS_euler_s}--\eqref{NS_euler_e}, obeying similar estimates.

\section{Discussion}
\label{sec:discuss}
In order to prove global well-posedness and decay, we employ a \emph{nonlinear energy method}.
We outline this method in Section \ref{sec:DiscussNonlinEnMeth}, discuss the difficulties that arise in Section \ref{sec:DiscussDiff},
and provide a strategy of the proof in Section \ref{sec:DiscussStrategyPf}.  We also discuss how the work presented in this paper fits with respect to previous work considering other types of surface forces, highlighting that the present work may be viewed, in some sense, as `supercritical.'

\subsection{Nonlinear energy method}
\label{sec:DiscussNonlinEnMeth}
	In this section we provide a high-level overview of the nonlinear energy method employed to prove global well-posedness and decay.
	The \emph{nonlinear energy method} informs the scheme of a priori estimates that we employ, and begin as follows:
	we multiply the PDE  by the unknown $u$ and integrate by parts with respect to the \emph{nonlinear} differential operators $\nabla^\geo$. 	This yields the energy-dissipation relation
	\begin{align*}
		\Dt\brac{
			\int_\Omega \frac{1}{2} \abs{u}^2 J
			+ \will\brac{\eta}
			+ \int_{\T^2} \frac{g}{2} \abs{\eta}^2
		}
		+ \brac{
			\int_\Omega \frac{1}{2} \abs{\symgrad^\geo u}^2 J
		}
		= 0
	\end{align*}
	where $J\defeq \det\nabla\Phi$ accounts for the local deformation in volume due to the change of coordinates $\Phi$.
	Close to the equilibrium solution $\brac{u,p,\eta} = 0$,
	the energy-dissipation relation becomes the same as that which is obtained by a standard energy estimate for the linearization of the PDE about the equilibrium, namely
	\begin{align*}
		\Dt\underbrace{
			\brac{
				\int_\Omega \frac{1}{2} \abs{u}^2
				+ \quadw{0} \brac{\eta}
				+ \int_{\T^2} \frac{g}{2} \abs{\eta}^2
			}
		}_{E}
		+ \underbrace{
			\brac{
				\int_\Omega \frac{1}{2} \abs{\symgrad u}^2
			}
		}_{D}
		= 0,
	\end{align*}
	where $\quadw{0}$ denotes the quadratic approximation of $\will$ about the equilibrium, and is defined precisely in Section \ref{sec:notSurfEner}.
	The good news is that if we restrict our attention to terms involving $u$, i.e. consider only $E_u: = \half\int_\Omega \abs{u}^2$,
	then it follows from the no-slip boundary condition $u=0$ on $\Sigma_b$ and Korn's inequality, Proposition \ref{eqEstKorn}, that the dissipation is coercive over the energy, i.e. 
	\[
		E_u = \int_\Omega \half\abs{u}^2 \lesssim \int_\Omega \half{\vbrac{\symgrad u}}^2.
	\]
	If for the moment we ignore the terms in the energy depending on $\eta$, then a Gronwall-type argument shows that we should expect exponential decay of $E_u$:
	\begin{align*}
		\begin{cases}
			\Dt E_u + D = 0\\
			E_u \leqslant CD
		\end{cases}
		\qquad\Rightarrow\qquad
		\Dt\brac{E_u \brac{t} e^{Ct}} \leqslant 0
		\qquad\Rightarrow\qquad
		E_u\brac{t} \leqslant e^{-Ct} E_u\brac{0}.
	\end{align*}
	Of course, we are not actually able to ignore the $\eta$ terms in the energy, so we must find a mechanism for controlling these terms with the dissipation functional.
	
	Such a mechanism is found by appealing to the equations \eqref{NS_fixed_dyn} and \eqref{NS_fixed_kin}, which allow us to estimate $\eta$ and $\partial_t \eta$.  Indeed, in order to obtain this coercivity we may use the elliptic nature of the dynamic boundary condition \eqref{NS_fixed_dyn} to transfer control of $u$ (and $p$) onto additional control of $\eta$.
	However at this stage we can only conclude that $u \in H^1$, which is insufficient to make sense of the trace of the stress tensor in the dynamic boundary condition,
	and so this mechanism for regularity transfer is not available to us.

	To resolve this issue we take derivatives of the problem that are compatible with the no-slip boundary condition (temporal and horizontal spatial derivatives) and apply a version of the energy-dissipation estimate.  The extra control that this provides then allows us to use a host of auxiliary estimates that permit the transfer of regularity between $u$ and $\eta$.
	For example, the dynamic boundary condition allows us to gain control of higher-order derivatives of $\eta$.
	Proceeding in this fashion, we can close the estimates by showing that the dissipation is coercive over the \emph{full} energy.

\subsection{Difficulties}
\label{sec:DiscussDiff}
	We now turn to a discussion of the difficulties encountered when employing the nonlinear energy method described above. 	The central difficulty is that there is a nontrivial interdependence between two essential features of the problem,
	namely the regularity gain and transfer mechanisms on one hand and the energy-dissipation structure on the other hand. This difficulty is exacerbated by two components of the problem in particular:
	\begin{itemize}
		\item	$\will$ is of order two, which is supercritical (in a sense made precise below), 
		\item	$\fvr$ is generally a quasilinear differential operator of order four.
	\end{itemize}

	In order to describe the difficulties we encounter, it is helpful to write the problem in a more compact form as $N\brac{\x} = 0$
	for $\x = \brac{u, p, \eta}$ the unknown and $N$ the nonlinear differential operator given by
	\begin{align*}
		N \brac{\x}
		= N \brac{u,p,\eta}
		= \begin{pmatrix}
			\pdt^\geo u + \brac{u\cdot\nabla^\geo}u + \nabla^\geo p - \Delta^\geo u
			\\
			\nabla^\geo\cdot u
			\\
			\tr_\Sigma \brac{pI - \symgrad^\geo u} \nu_\Sigma - \Bigbrac{\fvr\brac{\eta} + g\eta} \nu_\Sigma
			\\
			\pdt\eta - \tr_\Sigma u \cdot \nu_\Sigma \sqrt{1 + \abs{\nabla\eta}^2}
			\\
			\tr_{\Sigma_b} u
		\end{pmatrix}.
	\end{align*}

	\subsubsection{Structured estimates}
	
	Most terms in $N$ may be viewed as linear operators with multilinear dependence on geometric coefficients under control (such as $\geo$ and $J$).
	When computing the commutators between $N$ and partial derivatives, the contribution from these kind of terms is relatively benign.
	A more detailed description of these operators and the corresponding commutators may be found in Section \ref{sec:enerDissEst}.
	However the term $\fvr\brac{\eta}$, which comes from the fully nonlinear surface energy, cannot be written in this form
	and as a consequence it gives rise to commutators that are too singular to be controlled in a \emph{structured} manner.

	More precisely: the first attempt would be to write the equation $\partial^\alpha \brac{N\brac{\x}} = 0$ as a perturbation of $L\x = 0$, where $L$ denotes the linearization of $N$ about the equilibrium.
	In other words, we would seek to write $\partial^\alpha \brac{N\brac{\x}} = L\partial^\alpha\x + C\brac{\partial^\alpha\x}$ for some commutators $C$.
	Then upon integrating by parts and deriving the corresponding energy-dissipation relation, we would obtain commutators that are too singular to be controlled in a structured manner.

	To elucidate what we mean by this, let us consider the following cartoon: consider the following energy-dissipation relation
	\begin{equation*}
		\Dt E + D = C
	\end{equation*}
	where $C$ denotes some commutators.
	If we can show that
	\begin{equation}
	\label{eq:struc_est}
		\abs{C} \leqslant \sqrt{E} D,
	\end{equation}
	then for $E \leqslant \frac{1}{4}$ (i.e. in the cartoon version of what we will later call the small energy regime)
	we have that
	\begin{equation*}
		\Dt E + \half D \leqslant 0.
	\end{equation*}
	Moreover, if the dissipation $D$ is coercive over the energy $E$ (i.e. $E \leqslant D$), then we can conclude that the energy decays exponentially fast. 	However, if instead of \eqref{eq:struc_est} we can only show that
	\begin{equation}
	\label{eq:no_struc_est}
		\abs{C} \leqslant D^{3/2},
	\end{equation}
	then we cannot conclude anything about the boundedness or decay of $E$. In other words: whilst both \eqref{eq:struc_est} and \eqref{eq:no_struc_est} show that the commutators $C$ can be controlled,
	only \eqref{eq:struc_est} shows that the commutators can be controlled in a manner \emph{respectful of the energy-dissipation structure}.
	In particular, note that unstructured estimates like \eqref{eq:no_struc_est} are typically easier to obtain than structured estimates like \eqref{eq:struc_est} due to the fact that the dissipation is coercive over the energy, and hence $\sqrt{E} D \leqslant D^{3/2}$.

	A more specific discussion of why our scheme of a priori estimates would fail due to the term coming from the nonlinear surface energy may be found in Remark \ref{rmk:fakeCommutators}.

	\subsubsection{Parabolic criticality}
	
	As hinted at earlier, a particular source of difficulty when attempting to estimate these commutators comes from the fact that energies of order two, like the energies of Willmore-type considered here, are `supercritical.'	This critical phenomenon comes from the fact that the Stokes system embedded into our problem imposes parabolic scaling on $u$, but when we use the equations of motion to gain dissipative control of spatial and temporal derivatives of $\eta$ this generally induces non-parabolic scaling for $\eta$ estimates.  This mismatch between the $u$ scaling and the $\eta$ scaling is precisely the source of the critical threshold.	In particular, as will be detailed below, previous work dealing with capillary forces due to surface tension may be viewed as `subcritical'	whilst this work dealing with flexural forces due to bending may be viewed as `supercritical.'

	To better understand this difficulty it is helpful to consider a toy example in which 
	\begin{equation*}
	 \will\brac{\eta} = \int_{\T^2} {\vbrac{\abs{\nabla}^\alpha \eta}}^2
	\end{equation*}
	for some $\alpha > 0$.  We then observe that if $u\in H^s\brac{\Omega}$ (and so $p \in H^{s-1}\brac{\Omega}$), then we may use the kinematic and dynamic boundary conditions,
	\[
	\begin{cases}
		\brac{\svr{0}}\eta = \tr \brac{pI - \symgrad u} : \brac{e_3 \otimes e_3} \in H^{s-\frac{3}{2}}\brac{\T^2}\text{ and}\\
		\pdt\eta = \tr u \cdot e_3 \in H^{s-\half}\brac{\T^2},
	\end{cases}
	\]
	to obtain the following control over $\eta$ and $\pdt\eta$:
	\[
	\begin{cases}
		\normtypdom{\eta}{H}{s+2\alpha-3/2}{\T^2}
		\lesssim
		\normtypdom{u}{H}{s}{\Omega} + \normtypdom{p}{H}{s-1}{\Omega} \text{ and}\\
		\normtypdom{\pdt\eta}{H}{s-1/2}{\T^2}
		\lesssim
		\normtypdom{u}{H}{s}{\Omega}.
	\end{cases}
	\]
	Therefore the difference in regularity between $\eta$ and $\pdt\eta$ is $\brac{s+2\alpha-\frac{3}{2}} - \brac{s-\half} = 2\alpha -1$.
	To summarize schematically, the induced dissipative $\eta$ scaling is:
	\begin{equation*}
		\pdt\eta \sim {\vbrac{\nabla}}^{2\alpha-1}\eta,
	\end{equation*}
	where this should be understood in the sense that if we control $\pdt\eta$ in $H^s$, then we expect to control $\eta$ in $H^{s+\brac{2\alpha-1}}$,and vice-versa (i.e. control of $\eta$ in $H^s$ is expected to correspond to control of $\pdt\eta$ in $H^{s-\brac{2\alpha-1}}$).  
	
	This scaling mismatch complicates the design of a scheme of a priori estimates in which control of time derivatives is leveraged to gain control of spatial derivatives, but temporal differentiation of the equations leads to high-order commutators.  In particular:
	\begin{itemize}
		\item	For $\alpha < \frac{3}{2}$, temporal derivatives of $\eta$ are cheap relative to spatial derivatives (by contrast with parabolic scaling).  This is what we refer to as the subcritical case.  The case of surface tension, which corresponds to $\alpha = 1$, falls into this category.
		\item	For $\alpha = \frac{3}{2}$, $\eta$ follows parabolic scaling.
		\item	For $\alpha > \frac{3}{2}$, temporal derivatives of $\eta$ are expensive relative to spatial derivatives (by contrast with parabolic scaling).  This is what we refer to as the supercritical case.
			The case of flexural forces, which corresponds to $\alpha = 2$ and which is considered in this paper, falls into this category.
	\end{itemize}
	Since the Willmore-type energies we consider here are supercritical, we must therefore be very wary of commutators involving time derivatives of $\eta$.	Again, the precise manner in which this can be an issue for the scheme of a priori estimates presented here is discussed in Remark \ref{rmk:fakeCommutators}.
	
	\subsubsection{Appropriate linearization}
	To summarize the difficulties discussed so far: we seek to estimate the commutators in a structured manner, and we have to be particularly careful regarding terms involving time derivatives of $\eta$ due to the supercriticality of the Willmore-type energies discussed here.
	To address both of these issues we proceed as follows: instead of linearizing the PDE system directly (whether about the equilibrium or about any $\x$),
	we find a quadratic approximation of the energy and dissipation, and then derive the associated PDE
	- which is also linear but \emph{not} the same as the linearization of the nonlinear operator $N$.
	In some sense, it is beneficial to perform the linearization in this manner since it is more respectful of the structure of the fully nonlinear surface energy.
	In a more precise sense, we will see below that performing the linearization in this manner leads to commutators that can be controlled.

	We thus view $N$ as a perturbation of some linear operator $L_\x$ (i.e. a linear operator whose coefficients depend on $\x$) different from its linearization
	but such that the energy-dissipation relation associated with $L_\x$ has `good commutators'.
	Note that we write $L_\x$ to emphasize that the coefficients of this linear operators depend on $\x$.
	We will thus consider the commutators (called this by a slight abuse of notation) $\partial^\alpha \circ N - L_\x \circ \partial^\alpha$, where $L_\x$ is given by 
	\begin{align*}
		L_\x \brac{\y} := L_\x \brac{v,q,\zeta}
		= \begin{pmatrix}
			\pdt^\geo v + \brac{u\cdot\nabla^\geo}v + \nabla^\geo p - \Delta^\geo v
			\\
			\nabla^\geo\cdot v
			\\
			\tr_\Sigma \brac{qI - \symgrad^\geo v} \nu_\Sigma - \Bigbrac{\svr{\eta}\brac{\zeta} + g\zeta} \nu_\Sigma
			\\
			\pdt\zeta - \tr_\Sigma v \cdot \nu_\Sigma \sqrt{1 + \abs{\nabla\eta}^2}
			\\
			\tr_{\Sigma_b} u.
		\end{pmatrix}.
	\end{align*}
	Note here that $\geo = \geo\brac{\eta}$ and $\nu_\Sigma = \nu_\Sigma \brac{\eta}$, i.e. these geometric coefficients depend on $\eta$ (i.e. on $\x$) and not $\zeta$ (i.e. not on $\y$).

	This is where the subtle interdependence between the energy-dissipation structure and the regularity gain and transfer structure is most apparent.
	On one hand the linearization of $N$ about the equilibrium, denoted by $L$, tells us how much regularity can be gained and therefore tell us which commutators can be controlled,
	and on the other hand the energy-dissipation structure associated with $N$ tells us which form of control of these commutators is allowed in order to close the estimates.
	The precise form of $L_\x$ is then chosen such that it yields `good' commutators respectful of both of these features, i.e. commutators upon which we have \emph{structured control},
	and which are also tame enough despite the supercriticality of the surface energy.
	In particular, note that when $\x$ is the equilibrium solution, i.e. $\x = 0$, then $L_0 = L$.
	
	\subsubsection{Failure of coercivity}
	We discussed above that surface energies of order $\alpha = 3/2$ are critical, in some sense.
	Nonetheless, close to that exponent, i.e. whether in the case of surface tension where $\alpha = 1$ or in the case of bending energies where $\alpha = 2$,
	exponential decay of the energy can be obtained.
	Whilst this is not addressed directly in this paper, it is worth pointing out that this is no longer true when $\alpha < 1/2$ or $\alpha > 5/2$.
	\begin{itemize}
		\item	When $\alpha < 1/2$ one does not obtain exponential decay of the energy for the linearized problem about equilibrium, but only algebraic decay.  We refer to Tice--Zbarsky \cite{tice_zbarsky} for details.
		\item	When $\alpha > 5/2$, the scheme of a priori estimates is not sufficient to obtain coercivity of the dissipation over the energy.
			Recall that in order to show that the dissipation is coercive over the energy, we must differentiate the PDE. Indeed, upon differentiating we obtain enough control on $u$ to make sense of the trace of the stress tensor $pI - \symgrad u$,
			which in turns allows us to leverage the dynamic boundary condition to turn control of $u$ into higher-order control of $\eta$,
			thus obtaining coercivity. Taking derivatives up to parabolic order two (i.e. taking one temporal and two spatial derivatives)
			we see that the only appearance of $\pdt\eta$ in the energy is via the term
			\[
				\quadw{0}\brac{\pdt\eta} \asymp \normtypdom{\pdt\eta}{H}{\alpha}{\T^2},
			\]
			whilst the kinematic boundary tells us that 
			\[
				\dsimp
				\gtrsim \normtypdom{u}{H}{3}{\Omega}^2
				\gtrsim \normtypdom{\tr u}{H}{5/2}{\T^2}^2
				\gtrsim \normtypdom{\pdt\eta}{H}{5/2}{\T^2}^2.
			\]
			So indeed, for $\alpha > \frac{5}{2}$, $\dsimp\not\gtrsim\enimp$.  Note that this problem cannot be circumvented by applying more time derivatives, as it will always occur for the highest order term.
	\end{itemize}

\subsection{Strategy of the proof}
\label{sec:DiscussStrategyPf}
	In this section we sketch the strategy of the proof.
	We describe the key moving pieces in Section \ref{sec:StrategyPfMovingPieces}, then discuss how they interact in Section \ref{sec:StrategyPfInteractions}.
	This allows us to identify in Section \ref{sec:StrategyPfHardAnalysis} the `hard analysis' estimates that have to be made to close the estimates and thus conclude the proof.
	Throughout this section we also outline the plan of the paper,
	pointing to the location of each step of the proof.
	\subsubsection{The moving pieces}
	\label{sec:StrategyPfMovingPieces}
		The key moving pieces are: 1.\! $L$,\, 2.\! $L_\x$, and 3.\! the various versions of the energy and the dissipation.
		\begin{enumerate}
			\item	We denote by $L$ the linearization of $N$ about the equilibrium, which is responsible for the regularity gain and transfer mechanisms.
			\item	We denote by $L_\x$ a linear approximation of $N$ about $\x$, which is responsible for the energy-dissipation structure of the problem.
				In particular $L_\x$ dictates the precise form of the energy-dissipation relation
				and of the commutators $\partial^\alpha \circ N - L_\x \circ \partial^\alpha$.
			\item	The various versions of the energy and the dissipation ( precisely defined in Section \ref{sec:notationVersionsEnergyDissipation}): 
				\begin{itemize}
					\item	The \emph{equilibrium} versions, denoted by $\eneq$ and $\dseq$,
						which come from the energy-dissipation relation corresponding to the linearized problem about the equilibrium
						and consist of functional norms of the unknowns.
					\item	The \emph{improved} versions, denoted by $\enimp$ and $\dsimp$,
						which are obtained by bootstrapping from the equilibrium versions,
						using the regularity gain and transfer mechanisms embedded in $L$.
						In other words, if $L\x = 0$ then $\eneq$ controls $\enimp$ and $\dseq$ controls $\dsimp$.
					\item	The \emph{geometric} versions, denoted by $\engeo$ and $\dsgeo$,
						which come from the energy-dissipation relation corresponding to $N$ and $L_\x$
						and consist of functional norms of the unknowns involving the $\geo$-differential operators
						and weighted by the geometric coefficient (such as $J$).
				\end{itemize}
				In particular, note that since $L_\x$ depends on $\x$, so do $\engeo$ and $\dsgeo$, and so we also write them as
				$\engeo\brac{\cdot\,;\x}$ and $\dsgeo\brac{\cdot\,;\x}$, respectively.
				Moreover, note that the notation we use is consistent since on one hand, when $\x=0$ we have that $L_0 = L$,
				and on the other hand $\engeo\brac{\cdot\,;0} = \eneq$ and $\dsgeo\brac{\cdot\,;0} = \dseq$.
				This is summarized in the diagram below, where `IBP' denotes integration by parts.
				\begin{center}
				\begin{tikzcd}
					{L_\x \y = 0}
						\arrow[r, "IBP"]
						\arrow[d, "\x = 0"]
					& {\Dt \engeo\brac{\y\,;\x} + \dsgeo\brac{\y\,;\x} = 0}
						\arrow[d, "\x = 0"]
					\\
					{L\y = 0}
						\arrow[r, "IBP"]
					& {\Dt \eneq\brac{\y} + \dseq\brac{\y} = 0}
				\end{tikzcd}
				\end{center}
				The precise derivation of the energy-dissipation relations can be found at the start of Section \ref{sec:enerDissEst}.
		\end{enumerate}
	\subsubsection{How the moving pieces interact}
	\label{sec:StrategyPfInteractions}
		As discussed earlier, there are two key features of the problem that our proof relies on: 
		\begin{enumerate}
			\item	Given the equilibrium versions of the energy and the dissipation,
				the regularity gain and transfer mechanisms embedded in the linearization $L$ dictate the form of the improved versions.
				The general form of the auxiliary estimates obtained from those regularity gain and transfer mechanisms can be found at the start of Section \ref{sec:nonlinearCorr}.
			\item	The form of $L_\x$ dictates the energy-dissipation structure,
				which thus determines the form of the geometric versions of the energy and dissipation, as well as the form of the commutators
				$\partial^\alpha \circ N - L_\x \circ \partial^\alpha$.
				The derivation of the energy-dissipation relation and the computation of the commutators can be found at the start of Section \ref{sec:enerDissEst}.
		\end{enumerate}
		The interaction of the moving pieces is also summarized more tersely in Figure \ref{fig:sketchArg}.
		\begin{figure}[h!]
		\begin{tikzcd}[column sep=large]
			{L\x = R}
				\arrow[d]
			& {N\brac{\x} = 0}
				\arrow[l, <->]
				\arrow[r, "\geo-IBP"]
				\arrow[d, "\partial^\alpha"]
			& {\Dt \engeofunc^0\brac{\x} + \dsgeofunc^0\brac{\x} = 0}
			\\
			{\left\{
				\begin{aligned}
					\enimp \lesssim \eneq + \rem_E\\
					\dsimp \lesssim \dseq + \rem_D
				\end{aligned}
			\right.}
			& {L_\x \brac{\partial^\alpha \x} = C^\alpha}
				\arrow[r, "\geo-IBP"]
			& {\Dt \engeofunc\brac{\partial^\alpha \x;\x} + \dsgeofunc\brac{\partial^\alpha \x;\x} = {\abrac{C^\alpha, \partial^\alpha \x}}_{\x}}
		\end{tikzcd}
		\caption{
			Schematic overview of the strategy of the proof,
			where $\geo-IBP$ refers to integration by part with respect to the $\geo$-differential operators
			(c.f. Section \ref{sec:geoCoeffAndDiffOp} for the relevant integration theorems).
		}
		\label{fig:sketchArg}
		\end{figure}
	\subsubsection{The `hard analysis' estimates}
	\label{sec:StrategyPfHardAnalysis}
		In order to close the estimates, we need to show that, in the small energy regime,
		\begin{itemize}
			\item	the commutators are small, which is done in the latter part of Section \ref{sec:enerDissEst}, and
			\item	all versions of the energy are comparable (and similarly for the dissipation),
				which is done in the latter part of Section \ref{sec:nonlinearCorr} (where we essentially show that the equilibrium and improved versions are comparable)
				and in Section \ref{sec:geoCorr} (where we essentially show that the equilibrium and geometric versions are comparable).
		\end{itemize}

\section{Notation}\label{sec:notation}
The purpose of this section is to collect in a single place all of the notational conventions we will use throughout the rest of the paper.

\subsection{Basics}

Here we collect notation for variables, derivatives, and tensor manipulations.

\subsubsection{Variables and derivatives}  We use the following notation for space-time variables.
	\begin{itemize}
		\item $T \in \ocbrac{0, \infty}$ denotes a time.
		\item For any $x = \brac{x_1, x_2, x_3}\in\R^3$, we write $\bar{x} \defeq \brac{x_1, x_2}\in\R^2$ and $\tilde{x} \defeq \brac{\bar{x},0} = \brac{x_1, x_2, 0}\in\R^3$.
		\item Similarly, we employ the following notation for derivatives: $\nabla = \brac{\partial_1, \partial_2, \partial_3}$, $\nablatwo \defeq \brac{\partial_1, \partial_2}$, and $\nablatwoemb \defeq \brac{\nablatwo,0} = \brac{\partial_1, \partial_2, 0}$.
	\end{itemize}

\subsubsection{Parabolic order of multi-indices} 
	For any $\alpha = \brac{\alpha_0, \bar\alpha} \in \N^{1+n}$ such that $\partial^\alpha = \partial_t^{\alpha_0} \partial_{\bar{x}}^{\bar\alpha}$, we define $\parabolicOrder{\alpha} \defeq 2 \alpha_0 + \bar\alpha$,
	and call it the \emph{parabolic order} of $\alpha$.

\subsubsection{Inequalities}  We say a constant $C$ is universal if it only depends on the various parameters of the problem, the dimension, etc., but not on the solution or the data.  The notation $\alpha \lesssim \beta$ will be used to mean that there exists a universal constant $C>0$ such that $\alpha \le C \beta$.

\subsubsection{Contractions, inner products, and derivatives of tensors}
	Throughout the paper we will use the Einstein summation convention of summing over repeated indices.  We will also need the following scalar products:
	\begin{itemize}
		\item	$a\cdot b = a_i b_i$ for any $a,b\in\R^n$,
		\item	$A:B = A_{ij} B_{ij}$ for any $A,B\in\R^{n\times n}$,
		\item	$T \bullet S = T_{i_1 \dots i_k} S_{i_1 \dots i_k}$ for any $T,S\in\R^{\overbrace{n \times \dots \times n}^{k\text{ times}}} = {\brac{\R^n}}^{\otimes k}$.
	\end{itemize}
	When contracting tensors of different ranks we will write 
	\begin{itemize}
		\item	${\brac{T \bullet S}}_{j_1 \dots j_p k_1 \dots k_r} = T_{j_1 \dots j_p i_1 \dots i_q} S_{i_1 \dots i_q k_1 \dots k_r}$
			for any $T\in{\brac{\R^n}}^{\otimes\brac{p+q}}$ and $S\in{\brac{\R^n}}^{\otimes\brac{q+r}}$,
			such that $T\bullet S \in{\brac{\R^n}}^{\otimes\brac{p+r}}$.
	\end{itemize}
	For derivatives of tensors we write:
	\begin{itemize}
		\item	${\brac{\nabla^l S}}_{i_1 \dots i_k a_1 \dots a_l} = \partial_{a_1} \dots \partial_{a_l} S_{i_1 \dots i_k}$
			for any $S:\R^n\to{\brac{\R^n}}^{\otimes k}$,
		\item	${\brac{\brac{{\nabla^l}}^T S}}_{a_1 \dots a_l i_1 \dots i_k} = \partial_{a_1} \dots \partial_{a_l} S_{i_1 \dots i_k}$
			for any $S:\R^n\to{\brac{\R^n}}^{\otimes k}$.
	\end{itemize}

\subsection{Sobolev spaces}  Here we record our notation for Sobolev spaces.
	\begin{itemize}
		\item For sets of the form $D = \T^2$ or $\Omega$ we write $H^s(D)$ to denote the usual $L^2-$based Sobolev space of order $s \ge 0$,
			and write $\dot{H}^s\brac{D}$ to denote their homogeneous counterparts.
			When $D = \T^2$ we extend this to include $s <0$ using the standard Fourier characterization.
		\item	For sets of the form $D = \T^2$ or $\Omega$, the notation $H^{s+} \brac{D}$ will be employed to mean the following:
			\[
			\begin{cases}
				\alpha \lesssim \normtypdom{f}{H}{s+}{D}
				&\text{means that }\forall\,\epsilon > 0, \exists\,C>0 \text{ s.t. } \alpha \leqslant C \normtypdom{f}{H}{s+\epsilon}{D}\\
				\normtypdom{f}{H}{s+}{D} \lesssim \beta
				&\text{means that }\exists\,\epsilon > 0, \exists\,C>0 \text{ s.t. } \normtypdom{f}{H}{s+\epsilon}{D} \leqslant C\beta.\\
			\end{cases}
			\]
	\end{itemize}

\subsection{Domains and coefficients}\label{sec:dom_coeff}	
Here we record notation related to the Eulerian and fixed domains and the coefficients associated to them.	
	
\subsubsection{Eulerian and flattened domains}
	We recall that the Eulerian and fixed or equilibrium domains satisfy the following. 
	\begin{multicols}{2}
		\textbf{The Eulerian domain}
		\begin{itemize}
			\item $\Omega\brac{t} \defeq \setdef{x \in \T^2 \times \R}{-b < x_3 <\eta\brac{t,\bar{x}}}$
			\item $\Sigma\brac{t} \defeq \setdef{x \in \T^2 \times \R }{ x_3 = \eta\brac{t,\bar{x}}}$
			\item $\Sigma_b \defeq \setdef{x \in \T^2 \times \R}{x_3 = -b}$
			\item $\partial\Omega\brac{t} = \Sigma\brac{t} \sqcup \Sigma_b$
		\end{itemize}
		
		\textbf{The fixed domain}
		\begin{itemize}
			\item $\Omega \defeq \T^2 \times \brac{-b, 0}$
			\item $\Sigma \defeq \T^2 \times \cbrac{0}$
			\item $\Sigma_b$ as before
			\item $\partial\Omega = \Sigma \sqcup \Sigma_b$
		\end{itemize}
	\end{multicols}

\subsubsection{Geometric coefficients}\label{sec:geoCoeff}
	 Recall that the flattening map $\Phi$ defined by \eqref{intro_Phi_def} allows us to map $\Omega$ to $\Omega(t)$.  Associated to the flattening map are the following essential geometric coefficients.
	\begin{itemize}
		\item	$J \defeq \det\nabla\Phi = 1 + \partial_3 \brac{\rchi\ext\eta}$
		\item	$\geo \defeq \brac{\nabla\Phi}^{-T}
			= {\brac{I + e_3 \otimes \nabla\brac{\rchi\ext\eta}}}^{-T}
			= I - \frac{\nabla\brac{\rchi\ext\eta}\otimes e_3}{1 + \partial_3 \brac{\rchi\ext\eta}}$
	\end{itemize}
	See Lemma \ref{lemma:estGeoCoeff} for the computations of $J$ and $\geo$.

\subsubsection{Differential operators with variable coefficients}
	\label{sec:notGeoDiffOp}
	Given any matrix field $M : \Omega \to \R^{3 \times 3}$ and any vector field $v : \Omega \to \R^3$, we define
	\begin{itemize}
		\item $\nabla^M \defeq M \cdot \nabla$, i.e. $\partial^M_i = M_{ij} \partial_j$
		\item $\symgrad^M v \defeq 2\,\sym\brac{\nabla^M v} = \nabla^M v + \brac{\nabla^M v}^T$
	\end{itemize}
	When $M = \geo$, these operators arise naturally as ``$\Phi$-conjugates'' of the usual differential operators $\nabla$ and $\symgrad$.
	More precisely, upon changing variables via $\Phi$ we have that $\nabla^\geo f = \nabla\brac{f \circ \Phi\inv} \circ \Phi$ (and similarly for the symmetrized gradient).  Note that, as illustrated in Figure \ref{fig:vec_fields}, horizontal slices in the fixed domain correspond to curved hypersurfaces in the Eulerian domain.  In particular, horizontal derivatives in the fixed domain correspond to derivatives tangential to these hypersurfaces in the Eulerian domain.  
		
	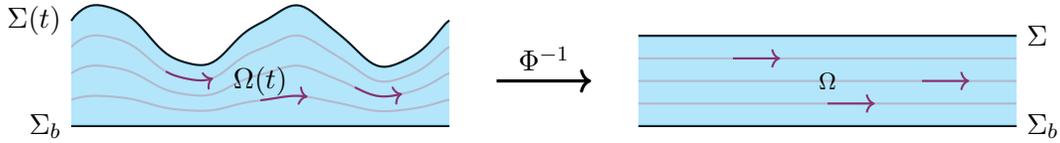
\begin{figure}[h!]
	\centering
	\begin{tikzpicture}[xscale=0.8,yscale=0.8]
			\draw [name path=St, thick, domain=0:2*pi, samples=\bigsample]
			plot (\x,{.5*sin((2*\x + pi/6) r) + (0.02*sin((3*\x + pi/6) r) + 0.01*sin((5*\x + pi/6) r) + 0.02*sin((7*\x + pi/6) r))*pow( \x*(2*pi-\x),1/5)});
			\draw [thick, domain=0:2*pi, samples=\bigsample, color=pink]
			plot (\x,{-3/8 + (9/12)*(.5*sin((2*\x + pi/6) r) + (0.02*sin((3*\x + pi/6) r) + 0.01*sin((5*\x + pi/6) r) + 0.02*sin((7*\x + pi/6) r))*pow( \x*(2*pi-\x),1/5))});
			\draw [thick, domain=0:2*pi, samples=\bigsample, color=pink]
			plot (\x,{-6/8 + (6/12)*(.5*sin((2*\x + pi/6) r) + (0.02*sin((3*\x + pi/6) r) + 0.01*sin((5*\x + pi/6) r) + 0.02*sin((7*\x + pi/6) r))*pow( \x*(2*pi-\x),1/5))});
			\draw [thick, domain=0:2*pi, samples=\bigsample, color=pink]
			plot (\x,{-9/8 + (3/12)*(.5*sin((2*\x + pi/6) r) + (0.02*sin((3*\x + pi/6) r) + 0.01*sin((5*\x + pi/6) r) + 0.02*sin((7*\x + pi/6) r))*pow( \x*(2*pi-\x),1/5))});
			\draw [thick, domain=4/8*pi:6/8*pi, samples=\smallsample, color=purple, ->]
			plot (\x,{-3/8 + (9/12)*(.5*sin((2*\x + pi/6) r) + (0.02*sin((3*\x + pi/6) r) + 0.01*sin((5*\x + pi/6) r) + 0.02*sin((7*\x + pi/6) r))*pow( \x*(2*pi-\x),1/5))});
			\draw [thick, domain=12/8*pi:14/8*pi, samples=\smallsample, color=purple, ->]
			plot (\x,{-6/8 + (6/12)*(.5*sin((2*\x + pi/6) r) + (0.02*sin((3*\x + pi/6) r) + 0.01*sin((5*\x + pi/6) r) + 0.02*sin((7*\x + pi/6) r))*pow( \x*(2*pi-\x),1/5))});
			\draw [thick, domain=8/8*pi:10/8*pi, samples=\smallsample, color=purple, ->]
			plot (\x,{-9/8 + (3/12)*(.5*sin((2*\x + pi/6) r) + (0.02*sin((3*\x + pi/6) r) + 0.01*sin((5*\x + pi/6) r) + 0.02*sin((7*\x + pi/6) r))*pow( \x*(2*pi-\x),1/5))});
			\draw [name path=Sb, thick] (0,-1.5) to (2*pi,-1.5);
			\tikzfillbetween[of=St and Sb]{cyan, opacity = 0.3}
			\node [left] at (0*pi,0.25) {$\Sigma(t)$};
			\node [left] at (0*pi,-1.5) {$\Sigma_b$};
			\node at (pi,-0.75) {$\Omega(t)$};
			\draw [very thick, ->] (2.25*pi,-0.75) to (2.75*pi,-0.75);
			\node [above] at (2.5*pi,-0.75) {${\mathrm{\Phi}}^{-1}$};
			\draw [name path=Sflat, thick, domain=3*pi:5*pi]
			plot (\x, {0});
			\draw [name path=Sbflat, thick] (0+3*pi,-1.5) to (2*pi+3*pi,-1.5);
			\draw [thick, domain=0+3*pi:2*pi+3*pi, samples=2, color=pink]
			plot (\x,{-3/8});
			\draw [thick, domain=0+3*pi:2*pi+3*pi, samples=2, color=pink]
			plot (\x,{-6/8});
			\draw [thick, domain=0+3*pi:2*pi+3*pi, samples=2, color=pink]
			plot (\x,{-9/8});
			\draw [thick, domain=4/8*pi+3*pi:6/8*pi+3*pi, samples=2, color=purple, ->]
			plot (\x,{-3/8});
			\draw [thick, domain=12/8*pi+3*pi:14/8*pi+3*pi, samples=2, color=purple, ->]
			plot (\x,{-6/8});
			\draw [thick, domain=8/8*pi+3*pi:10/8*pi+3*pi, samples=2, color=purple, ->]
			plot (\x,{-9/8});
			\tikzfillbetween[of=Sflat and Sbflat]{cyan, opacity = 0.3}
			\node [right] at (2*pi+3*pi,0) {$\Sigma$};
			\node [right] at (2*pi+3*pi,-1.5) {$\Sigma_b$};
			\node at (pi+3*pi,-0.75) {{\footnotesize $\Omega$}};
	\end{tikzpicture}
	\caption{Horizontal slices in the fixed domain are mapped to curved hypersurfaces in the Eulerian domain by the diffeomorphism flattening the domain.}\label{fig:vec_fields}
	\end{figure}

	Since $\Phi$ is time-dependent, we also define $\geo$-versions of time derivatives:
	\begin{itemize}
		\item $\pdt^\geo \defeq \pdt - \brac{\pdt\Phi}\cdot\nabla^\geo = \pdt - \frac{1}{J}\rchi\ext\brac{\pdt\eta}\partial_3$
		\item $\pdtm^{v,\geo} \defeq \pdt^\geo + v\cdot\nabla^\geo$
	\end{itemize}
	(c.f. Lemma \ref{lemma:estGeoCoeff} for the computation of $\pdt\Phi$).
	Once again, these differential operators arise naturally when changing variables since $\pdt^\geo f = \pdt\brac{f \circ \Phi\inv} \circ \Phi$. 	Similarly, $\pdtm^{v,\geo}$ arises naturally in the context of the $\geo$-Reynolds transport theorem (c.f. Proposition \ref{prop:geoTransportThm}). Finally, when integrating by parts, since $\nabla^\geo \neq \nabla$ we will pick up a normal $\nugeo \neq \nu_{\partial\Omega}$ defined as
	\begin{align*}
		\nugeo \defeq \underbrace{\brac{\geo J}}_{\text{cof}{\nabla\Phi}} \cdot\, \nu_{\partial\Omega}
		= \begin{cases}
			-\nablatwoemb\eta + e_3 = \sqrt{1 + \abs{\nablatwo \eta}^2} \;\nu_{\partial\Omega}
				&\text{on }\Sigma\\
			- e_3 = \nu_{\partial\Omega}
				&\text{on }\Sigma_b
		\end{cases}
	\end{align*}
	(see Proposition \ref{prop:geoDivThm} for the statement of the $\geo$-divergence theorem and Lemma \ref{lemma:estGeoCoeff} for the computation of $\nugeo$).

\subsection{Terms related to the surface energy}
Here we record notation related to the surface energy.

\subsubsection{Functionals and operators associated with the surface energy}\label{sec:notSurfEner}
	We consider some surface energy density $f : \R^2 \times \R^{2 \times 2} \to\R$, and define the following for any sufficiently regular $\eta, \phi, \psi, \phi_i : \T^2 \to \R$, where  $i=1, \dots, k$.
	\begin{itemize}
		\item Jet: $\jet\eta \defeq \brac{\nabla\eta, \nabla^2\eta}$, i.e. $\jet = \brac{\nabla,\nabla^2}$
			such that $\jet^* \brac{\p, M} = -\nabla\cdot \p + \nabla^2 : M$.
		\item Surface energy: $\will\brac{\eta} \defeq \int_{\T^2} f\brac{\jet\eta}$.
		\item Directional derivatives: $\delta_\phi \will\brac{\eta} \defeq \Dt \will\brac{\eta + t\phi} \vert_{t=0}$.
		\item Derivative: $D\will$ defined via $\abrac{D\will\brac{\eta}, \phi} \defeq \delta_\phi \will \brac{\eta}$.
		\item Second derivative: $D^2 \will$ defined via $\abrac{D^2\will\brac{\eta}, \brac{\phi, \psi}} \defeq \delta_\phi \delta_\psi \will \brac{\eta}
			= \delta_\psi \delta_\phi \will \brac{\eta}$.
		\item Higher-order derivatives: for $k\in\N$, $D^k \will$ defined via
			$$
				\abrac{D^k\will\brac{\eta}, \brac{\phi_1, \phi_2, \dots, \phi_k}} \defeq \delta_{\phi_1} \delta_{\phi_2} \dots \delta_{\phi_k} \will \brac{\eta}.
			$$
		\item First variation: $\fvr\brac{\eta} \defeq \jet^* \brac{\nabla f \brac{\jet\eta}}$ such that
			$$
				\abrac{D\will\brac{\eta}, \phi}
				= \int_{\T^2} \fvr\brac{\eta} \phi
				= \int_{\T^2} \nabla f \brac{\jet\eta} \cdot \jet\phi.
			$$
		\item Second variation: $\brac{\svr{\eta}} \phi \defeq \jet^* \brac{\nabla^2 f \brac{\jet\eta} \cdot \jet\phi}$ such that
			$$
				\abrac{D^2 \will\brac{\eta}, \brac{\phi, \psi}}
				= \int_{\T^2} \brac{\brac{\svr{\eta}}\phi} \psi
				= \int_{\T^2} \nabla^2 f \brac{\jet\eta} \bullet \brac{\jet\phi \otimes \jet\psi}.
			$$
		\item Higher-order variations: for $k\in\N$,
			$$
				\brac{\hovr{k}{\eta}}
				\brac{\phi_1, \phi_2, \dots, \phi_{k-1}}
				\defeq \jet^* \brac{
					\nabla^k f \brac{\jet\eta} \bullet
					\brac{\jet\phi_1 \otimes \jet\phi_2 \otimes \dots \otimes \jet\phi_{k-1}}
				}
			$$
			such that
			\begin{align*}
				\abrac{D^k \will\brac{\eta}, \brac{\phi_1, \phi_2, \dots, \phi_{k-1}, \phi_k}}
				&= \int_{\T^2} \Bigbrac{\brac{\hovr{k}{\eta}} \brac{\phi_1, \phi_2, \dots, \phi_{k-1}}} \phi_k\\
				&= \int_{\T^2} \nabla^k f \brac{\jet\eta} \bullet \brac{\jet\phi_1 \otimes \jet\phi_2 \otimes \dots \otimes \jet\phi_{k-1} \otimes \jet\phi_k}.
			\end{align*}
		\item Quadratic approximation:
			$$
				\quadw{\eta}\brac{\phi}
				\defeq \half \int_{\T^2} \nabla^2 f \brac{\jet\eta} \bullet \brac{\jet\phi \otimes \jet\phi}
				= \half \int_{\T^2} \brac{\brac{\svr{\eta} \phi}} \phi
				= \half \abrac{D^2 \will \brac{\eta}, \brac{\phi, \phi}}.
			$$
		\item Derivatives of the quadratic approximation: for any $\alpha \in \N^2$,
			$$
				\brac{\partial^\alpha \quadw{\eta}}\brac{\phi}
				\defeq \half \int_{\T^2} \partial^\alpha \Bigbrac{\nabla^2 f \brac{\jet\eta}} \bullet \brac{\jet\phi \otimes \jet\phi}
			$$
			and in particular $\quadw{\dot{\eta}} \defeq \pdt\quadw{\eta}$.
	\end{itemize}

\subsubsection{Constants associated to $f$}  At several points in our analysis we will need to refer to special constants related to the surface energy density $f$.  We define these now.
\begin{definition}[Universal constants]\label{def:universalConstants} 
    We define the following.
	\begin{itemize}
		\item	Define
			\[
				C_1 \defeq \norm{\jet}{
					\mathcal{L} \brac{
						H^{9/2} \brac{\T^2};
						L^\infty \brac{\T^2}
					}
				}
			\]
			and observe that $C_1$ is a finite universal constant since it only depends on the Sobolev embedding
			$H^s \brac{\T^2} \hookrightarrow L^\infty \brac{\T^2}$ for all $s > 1$.
		\item	Define, for all $k\in\N$,
			\[
				C_f^{\brac{k}} \defeq \normtypdom{\nabla^k f}{L}{\infty}{\overline{B\brac{0,C_1}}}.
			\]
			Crucially, note that if we are in the small energy regime (see Definition \ref{def:smallEnergyRegime}), where in particular $\enimp \leqslant 1$,
			then
			\[
				\normtypdom{\nabla^k f\brac{\jet\eta}}{L}{\infty}{\T^2}
				\leqslant 
				C_f^{\brac{k}}
				< \infty
			\]
			since for all $\eta : \T^2 \to \R$ sufficiently regular
			\[
				\normtypdom{\jet\eta}{L}{\infty}{\T^2}
				\leqslant
				C_1 \normtypdom{\eta}{H}{9/2}{\T^2}
				\leqslant
				C_1 \sqrt\enimp
				\leqslant
				C_1.
			\]
			This will be helpful to recall when we are performing the a~priori estimates since the term $\normtypdom{\nabla^k f\brac{\jet\eta}}{L}{\infty}{\T^2}$
			frequently appears (for various values of $k$).
	\end{itemize}
\end{definition}

\subsection{Quantities associated with the unknowns}
Here we collect notation associated with the unknowns.

\subsubsection{Unknown variables}\label{sec:notVariables}
We will use the following notation to refer to unknowns in the fluid equations.
	\begin{itemize}
		\item Velocities are $u, v : \cobrac{0,T}\times\Omega \to \R^3$.
		\item Pressures are $p, q : \cobrac{0,T}\times\Omega \to \R$.
		\item Stress tensors are $S^\geo, T^\geo : \cobrac{0,T}\times\Omega \to \sym\brac{R^{3\times 3}}$ defined by $S^\geo \defeq pI - \symgrad^\geo u$ and $T^\geo \defeq qI - \symgrad^\geo v$.
		\item Surface elevations are $\eta, \zeta : \cobrac{0,T}\times\T^2 \to \brac{-b, \infty}$.
	\end{itemize}

\subsubsection{The different versions of the energy and dissipation}\label{sec:notationVersionsEnergyDissipation}
We will need various forms of the energy and dissipation functionals.  We record the definitions of these now.

\textbf{Geometric versions:}
	For $\x_0 = \brac{u,p,\eta}$ and $\y = \brac{v,q,\zeta}$, we define
	\begin{subnumcases}{}
		\nonumber \engeofunc^0 \brac{\mathcal{X}_0} \defeq
			\frac{1}{2} \int_\Omega {\abs{u}}^2 J\brac{\eta}
			+ \will\brac{\eta}
			+ \frac{g}{2} \int_{\T^2} {\abs{\eta}}^2,\\
		\nonumber \engeofunc \brac{\y;\x_0} \defeq
			\frac{1}{2} \int_\Omega {\abs{v}}^2 J\brac{\eta}
			+ \quadw{\eta}\brac{\zeta}
			+ \frac{g}{2} \int_{\T^2} {\abs{\zeta}}^2,\\
		\nonumber \dsgeofunc^0 \brac{\x_0} \defeq
			\frac{1}{2} \int_\Omega {\abs{
				\symgrad^{\geo\brac{\eta}} u
			}}^2 J\brac{\eta}\text{, and}\\
		\nonumber \dsgeofunc \brac{\y;\x_0} \defeq
			\frac{1}{2} \int_\Omega {\abs{
				\symgrad^{\geo\brac{\eta}} v
			}}^2 J\brac{\eta}
	\end{subnumcases}
	where we have written $J\brac{\eta}$ and $\geo\brac{\eta}$ instead of writing, as we do elsewhere, $J$ and $\geo$ respectively in order to emphasize
	the dependence on $\eta$ of these geometric coefficients.
	We also define
	\begin{subnumcases}{}
		\engeo \brac{\y;\x_0} \defeq \engeofunc^0 \brac{\x_0}
			+ \engeofunc\brac{\pdt \y;\x_0}
			+ \engeofunc\brac{\nablatwo \y;\x_0}
			+ \engeofunc\brac{\nablatwo^2 \y;\x_0}
			\text{ and}
		\label{defeq:engeo}\\
		\dsgeo \brac{\y;\x_0} \defeq \dsgeofunc^0 \brac{\x_0}
			+ \dsgeofunc\brac{\pdt \y;\x_0}
			+ \dsgeofunc\brac{\nablatwo \y;\x_0}
			+ \dsgeofunc\brac{\nablatwo^2 \y;\x_0}
		\label{defeq:dsgeo}
	\end{subnumcases}
	i.e.\! sum up to derivatives of parabolic order two, where we write $F\brac{\nablatwo \y}$ to mean $\sum_i F\brac{\nablatwo_i \y}$
	and $F\brac{\nablatwo^2 \y}$ to mean $\sum_{i,j} F\brac{\nablatwo_{ij} \y}$.

	Note that $\engeofunc^0\brac{\x_0}$ and $\dsgeofunc^0\brac{\x_0}$ are functions whose domain is the space where $\x_0$ lives,
	but $\engeofunc\brac{\y;\x_0}$ and $\dsgeofunc\brac{\y;\x_0}$ are approximations of theses functions about $\x_0$,
	taking values $\y$ in the \emph{tangent space} to the space where $\x_0$ lives, hence why they are quadratic in $\y$.

\textbf{Equilibrium versions:}
	For $\x_{eq} = \brac{u_{eq},p_{eq},\eta_{eq}} = \brac{0,0,0}$, i.e. the equilibrium configuration, and $\y = \brac{v,q,\zeta}$, we define
	\begin{subnumcases}{}
		\eneqfunc \brac{\y}
			\defeq \engeofunc\brac{\y;\x_{eq}}
			= \half \int_\Omega {\abs{v}}^2
			+ \quadw{0}\brac{\zeta}
			+ \frac{g}{2} \int_{\T^2} {\abs{\zeta}}^2
		\nonumber
		\\
		\nonumber \quad
			= \half \int_\Omega {\abs{v}}^2
			+ \half \int_{\T^2} \brac{\brac{\svr{0}+g}\zeta}\zeta
		\text{ and}
		\\
		\nonumber \dseqfunc \brac{\y}
			\defeq \dsgeofunc\brac{\y;\x_{eq}}
			= \frac{1}{2} \int_\Omega {\abs{\symgrad v}}^2.
	\end{subnumcases}
	Note that, using the uniform ellipticity of $\svr{0}+g$ stated in Section \ref{sec:assumSurfEnerDensity}, we obtain that
	\[
	\begin{cases}
		&\eneqfunc \brac{\y}
			\asymp
			\normtypdom{v}{L}{2}{\Omega}^2
			+ \normtypdom{\zeta}{H}{2}{\T^2}^2, \\
		&\dseqfunc \brac{\y}
			\asymp \normtypdom{\symgrad v}{L}{2}{\Omega}^2.
	\end{cases}
	\]
	Then we define, once again summing up to derivatives of parabolic order two:
	\begin{subnumcases}{}
		\eneq \brac{\y} \defeq
			\eneqfunc\brac{\y} + \eneqfunc\brac{\pdt \y} + \eneqfunc\brac{\nablatwo \y} + \eneqfunc\brac{\nablatwo^2 \y}
			\label{defeq:eneq}\\
		\dseq \brac{\y} \defeq
			\dseqfunc\brac{\y} + \dseqfunc\brac{\pdt \y} + \dseqfunc\brac{\nablatwo \y} + \dseqfunc\brac{\nablatwo^2 \y}.
			\label{defeq:dseq}
	\end{subnumcases}
	
\textbf{Improved versions:}
	For $\y=\brac{v,q,\zeta}$, we define
	\begin{subnumcases}{}
		\enimp\brac{\y} \defeq
			\normtypdom{u}{H}{2}{\Omega}^2
			+ \normtypdom{\pdt u}{L}{2}{\Omega}^2
			+ \normtypdom{p}{H}{1}{\Omega}^2
			+ \normtypdom{\eta}{H}{{9/2}}{\T^2}^2
			+ \normtypdom{\pdt\eta}{H}{2}{\T^2}^2
			\text{ and}
		\label{defeq:enimp}\\
		\dsimp\brac{\y} \defeq
			\normtypdom{u}{H}{3}{\Omega}^2
			+ \normtypdom{\pdt u}{H}{1}{\Omega}^2
			+ \normtypdom{p}{H}{2}{\Omega}^2
		\nonumber
		\\
		\hspace{1.5 cm}
			+ \normtypdom{\eta}{H}{{11/2}}{\T^2}^2
			+ \normtypdom{\pdt\eta}{H}{{5/2}}{\T^2}^2
			+ \normtypdom{\pdt^2 \eta}{H}{{1/2}}{\T^2}^2.
		\label{defeq:dsimp}
	\end{subnumcases}
	Note that defined this way, coercivity is immediate, i.e. we have that $\enimp \lesssim \dsimp$.

\subsubsection{Small energy regime}	
We now define the `small energy regime' that is used throughout the paper.	
\begin{definition}[Small energy regime]\label{def:smallEnergyRegime}
	Let $C_0 > 0$ be defined by
	\[
		C_0 \defeq \norm{
			\ext
		}{
			\mathcal{L}\brac{
				H^{3/2} \brac{\T^2};\; L^\infty \brac{\Omega}
			}
		}\brac{
			\frac{1}{b}
			+ \norm{
				\sqrt{-\Delta}
			}
			{
				\mathcal{L}\brac{
					H^{5/2} \brac{\T^2};\; H^{3/2} \brac{\T^2}
				}
			}
		}
	\]
	and fix some $0 < \delta_0 < \min\brac{\frac{1}{C_0^2},1}$.
	We say that we are in the `small energy regime' if and only if there exists a solution $\x = \brac{u,p,\eta}$ on $\cobrac{0,T}$ such that
	\[
		\sup_{t\in\cobrac{0,T}} \enimp\brac{\x} \leqslant \delta_0
		\quad\text{and}\quad
		\sup_{t\in\cobrac{0,T}} \dsimp\brac{\x} < \infty.
	\]
\end{definition}
The following remarks will be important later.
\begin{remark}\label{rmk:smallEnergyRegime}$\text{}$

	\begin{enumerate}
		\item $C_0 < \infty$ since
			\[
			\begin{cases}
				\sqrt{-\Delta} \in \mathcal{L}\brac{
					H^{5/2}\brac{\T^2}; H^{3/2}\brac{T^2}
				},\\
				\ext \in \mathcal{L}\brac{
					H^{3/2}\brac{\T^2}; H^2 \brac{\Omega}
				}\text{, and}\\
				H^2\brac{\Omega} \hookrightarrow L^\infty \brac{\Omega}.
			\end{cases}
			\]
		\item	If $\x$ is a solution such that $\enimp\brac{\x} \leqslant \delta_0$, then in particular,
			by definition of $\rchi$ (c.f. Section \ref{sec:reformulation}), and by Lemma \ref{lemma:harmExtIdentities},
			\begin{align*}
				\normtypdom{\partial_3 \brac{\rchi\ext\eta}}{L}{\infty}{\Omega}
				&= \normtypdom{
					\frac{\ext\eta}{b} + \rchi\ext\sqrt{-\Delta}\eta
					}{L}{\infty}{\Omega}\\
				&\leqslant
					\norm{
						\ext
					}
					{
						\mathcal{L}\brac{
							H^{3/2} \brac{\T^2}; L^\infty\brac{\Omega}
						}
					}
					\Bigg(
					\frac{1}{b}
					\normtypdom{\eta}{H}{3/2}{\T^2}
					+ \norm{
						\sqrt{-\Delta}
					}{
						\mathcal{L}\brac{
							H^{5/2} \brac{\T^2}; H^{3/2} \brac{\T^2}
						}
					}
					\normtypdom{\eta}{H}{5/2}{\T^2}
					\Bigg)\\
				&\leqslant C_0 \normtypdom{\eta}{H}{5/2}{\T^2}
				\quad\leqslant C_0 \sqrt\enimp
				\quad\leqslant C_0 \sqrt{\delta_0}
				\quad < 1
			\end{align*}
			and therefore $\inf\ext\eta \geqslant -bC_0\sqrt{\delta_0} > -b$ such that $\Phi$ is well-defined,
			and
			\[
				\inf J = 1 + \inf\brac{ \frac{\ext\eta}{b} + \rchi\ext\sqrt{\brac{-\Delta\eta}}} \geqslant 1 - C_0\sqrt{\delta_0} > 0
			\]
			such that $\Phi$ is diffeomorphism.
		\item	We require $\delta < 1$ in order to simplify the a~priori estimates by not having to track powers of the energy.
			Indeed, for $\enimp \leqslant \delta_0 < 1$, $\enimp^{\alpha_1} + \dots + \enimp^{\alpha_n} \lesssim \enimp^{\min \alpha_i}$.
	\end{enumerate}
\end{remark}

\section{A priori estimates}
\subsection{Energy-dissipation estimates}
\label{sec:enerDissEst}
In this section we record the energy-dissipation relations arising from the original problem
(known as the zeroth-order energy-dissipation relation) in Proposition \ref{prop:ZerothOrderEnergyDissipationEstimates}
and from the differentiated problem (known as the higher-order energy-dissipation relation) in Proposition \ref{prop:HigherOrderEnergyDissipationEstimates}.
We then sketch the computation of the commutators, relegating the full details to the appendix,
and we estimate these commutators in Lemma \ref{lemma:smallnessEstimateCommutators}.

We start by recording, immediately below, the energy-dissipation relation arising from the original problem.
Note that in the notation of Section \ref{sec:DiscussDiff} this is the energy-dissipation relation corresponding to the system of PDEs $N\brac{\x} = 0$.
\begin{prop}[Zeroth-order energy-dissipation relation]\label{prop:ZerothOrderEnergyDissipationEstimates}
	If $\brac{u,p,\eta}$ solves \eqref{NS_fixed_s}--\eqref{NS_fixed_e}, then
	\begin{align*}
		\Dt\brac{
			\int_\Omega \frac{1}{2} \abs{u}^2 J
			+ \will\brac{\eta}
			+ \int_{\T^2} \frac{g}{2} \abs{\eta}^2
		}
		+ \brac{
			\int_\Omega \frac{1}{2} \abs{\symgrad^\geo u}^2 J.
		}
		= 0
	\end{align*}
	In other words, for $\engeofunc^0$ and $\dsgeofunc^0$ as defined in Section \ref{sec:notationVersionsEnergyDissipation} and $\x_0=\brac{u,p,\eta}$, we have that
	\begin{equation*}
		\Dt \engeofunc^0 \brac{\x_0} + \dsgeofunc^0 \brac{\x_0} = 0.
	\end{equation*}
\end{prop}
\begin{proof}
	We take the dot product of \eqref{NS_fixed_s} with $u$, multiply by $J$ to account for the geometry, and integrate over $\Omega$.  This results in:
	\begin{align*}
		0
		&=	\int_\Omega \brac{\pdtm^{u,\geo} u} \cdot u J
			+ \int_\Omega \brac{\divgeo S^\geo} \cdot u J\\
		&\stackrel{(1)}{=}
			\int_\Omega \pdtm^{u,\geo} \brac{\frac{1}{2} {\abs{u}}^2} J
			- \int_\Omega S^\geo : \brac{\nabla^\geo u} J
			+ \int_\Omega \brac{\brac{S^\geo}^T \cdot u} \cdot \nugeo\\
		&\stackrel{(2)}{=}
			\Dt \brac{\int_\Omega \frac{1}{2} {\abs{u}}^2 J}
			+ \int_\Omega \frac{1}{2} {\abs{\symgrad^\geo u}}^2 J
			+ \int_{\partial\Omega} \brac{S^\geo \cdot \nugeo} \cdot u\\
		&=	\Dt \brac{\int_\Omega \frac{1}{2} {\abs{u}}^2 J}
			+ \int_\Omega \frac{1}{2} {\abs{\symgrad^\geo u}}^2 J
			+ \int_{\T^2} \brac{\fvr\brac{\eta} + g\eta} \brac{u \cdot \nugeo}\\
		&=	\Dt \brac{\int_\Omega \frac{1}{2} {\abs{u}}^2 J}
			+ \int_\Omega \frac{1}{2} {\abs{\symgrad^\geo u}}^2 J
			+ \int_{\T^2} \brac{\fvr\brac{\eta} + g\eta} \pdt\eta\\
		&=	\Dt \brac{\int_\Omega \frac{1}{2} {\abs{u}}^2 J}
			+ \int_\Omega \frac{1}{2} {\abs{\symgrad^\geo u}}^2 J
			+ \Dt \brac{\will\brac{\eta} + \int_{\T^2} \frac{g}{2} {\abs{\eta}}^2}\\
		&=	\Dt \brac{
				\int_\Omega \frac{1}{2} {\abs{u}}^2 J
				+ \will\brac{\eta}
				+ \int_{\T^2} \frac{g}{2} {\abs{\eta}}^2
			}
			+ \int_\Omega \frac{1}{2} {\abs{\symgrad^\geo u}}^2 J.
	\end{align*}
	Here in (1) we have used the $\geo$-divergence theorem (Proposition \ref{prop:geoDivThm}) and the fact that
	$\nabla\cdot\brac{M^T\cdot v} = M:\nabla v + \brac{\nabla\cdot M}\cdot v$.  In (2) we have used the $\geo$-Reynolds transport theorem (Proposition \ref{prop:geoTransportThm}) and the fact that $\divgeo u = 0$.
\end{proof}
Having recorded the energy-dissipation relation associated with the original problem above in Proposition \ref{prop:ZerothOrderEnergyDissipationEstimates},
we now record the energy-dissipation relation associated with the differentiated problem below in Proposition \ref{prop:HigherOrderEnergyDissipationEstimates}.
Note that in the notation of Section \ref{sec:DiscussDiff} and for $C=\brac{C^1,C^2,C^3,C^4}$ this is the energy-dissipation relation corresponding to the system of PDEs $L_{\x_0}\brac{\y} = C$.

\begin{prop}[Higher-order energy-dissipation relation]\label{prop:HigherOrderEnergyDissipationEstimates}
	If $\x_0 = \brac{u,p,\eta}$ and $\y = \brac{v,q,\zeta}$ solve
	\begin{subnumcases}{}
		\pdtm^{u,\geo} v + \divgeo T^\geo = C^1		\label{eq:highOrderEDRelPDELinMom}				&in $\Omega$,\\
		\divgeo v = C^2													&in $\Omega$,\\
		\Bigbrac{\brac{\svr{\eta}}\zeta + g\zeta} \nugeo - T^\geo \cdot \nugeo = C^3					&on $\Sigma$,\\
		\pdt\zeta - v \cdot \nugeo = C^4										&on $\Sigma$\text{, and}\\
		v = 0														&on $\Sigma_b$,
	\end{subnumcases}
	where recall that $T^\geo \defeq qI - \symgrad^\geo v$ (c.f. Section \ref{sec:notVariables})
	and where $\geo = \geo\brac{\eta}$,
	then
	\begin{align*}
		&\Dt\brac{
			\int_\Omega \frac{1}{2} \abs{v}^2 J\brac{\eta}
			+ \quadw{\eta}\brac{\zeta}
			+ \int_{\T^2} \frac{g}{2} \abs{\zeta}^2
		}
		+ \brac{
			\int_\Omega \frac{1}{2} \abs{\symgrad^{\geo\brac{\eta}} v}^2 J\brac{\eta}
		}
		= \\
		&= \quadw{\dot{\eta}}\brac{\zeta}
		+ \int_\Omega \brac{C^1 \cdot v} J\brac{\eta}
		+ \int_\Omega C^2 q J\brac{\eta}
		+ \int_{\T^2} C^3 \cdot v
		+ \int_{\T^2} C^4\brac{\svr{\eta}+g} \zeta
		\eqdef {\abrac{C,\y}}_{\x_0}
	\end{align*}
	for $C=\brac{C^1,C^2,C^3,C^4}$.  In other words, for  $\engeofunc$ and $\dsgeofunc$ as defined in Section \ref{sec:notationVersionsEnergyDissipation}, 	
	\begin{equation*}
		\Dt \engeofunc\brac{\y;\x_0}
		+ \dsgeofunc\brac{\y;\x_0}
		= \abrac{C, \y}_{\x_0},
	\end{equation*}
	where we have written $J\brac{\eta}$ and $\geo\brac{\eta}$ instead of writing, as we do elsewhere, $J$ and $\geo$ respectively in order to emphasize the dependence on $\eta$ of these geometric coefficients.
\end{prop}
\begin{proof}
	Taking the dot product of \eqref{eq:highOrderEDRelPDELinMom} with $uJ$ and integrating over $\Omega$ yields
	\begin{align*}
		\overbrace{\int_\Omega \brac{C^1 \cdot v} J}^{\rnum{1}}
		&= \int_\Omega \brac{\pdtm^\geo v}\cdot v J
			+ \int_\Omega \brac{\divgeo T^\geo} \cdot v J\\
		&= \int_\Omega \pdtm^\geo \brac{\frac{1}{2}\abs{v}^2} J
			- \int_\Omega \brac{T^\geo : \nabla^\geo v} J
			+ \int_{\partial\Omega} \brac{T^\geo \cdot v} \cdot \nugeo\\
		&= \underbrace{\Dt \brac{\int_\Omega \frac{1}{2} \abs{v}^2 J}}_{\rnum{2}}
			- \underbrace{\int_\Omega q C^2 J}_{\rnum{3}}
			+ \underbrace{\int_\Omega \frac{1}{2} \abs{\symgrad^\geo v}^2 J}_{\rnum{4}}
			+ \underbrace{\int_{\T^2} \brac{T^\geo \cdot \nugeo} \cdot v}_{\brac{\star}}
	\end{align*}
	where
	\begin{align*}
		\brac{\star}
		&= \int_{\T^2} \Bigbrac{\brac{\svr{\eta}+g}\zeta}\brac{v\cdot\nugeo}
			- \int_{\T^2} C^3 \cdot v\\
		&= \int_{\T^2} \Bigbrac{\brac{\svr{\eta}+g}\zeta}\pdt\zeta
			- \int_{\T^2} \Bigbrac{\brac{\svr{\eta}+g}\zeta}C^4
			- \int_{\T^2} C^3 \cdot v\\
		&= \brac{
			\underbrace{\Dt \Bigbrac{\quadw{\eta}\brac{\zeta}}}_{\rnum{5}}
			- \underbrace{\quadw{\dot{\eta}}\brac{\zeta}}_{\rnum{6}}
		}
			+ \underbrace{\Dt\brac{\int_{\T^2} \frac{g}{2} \zeta^2}}_{\rnum{7}}
			- \underbrace{\int_{\T^2} \Bigbrac{\brac{\svr{\eta}+g}\zeta}C^4}_{\rnum{8}}
			- \underbrace{\int_{\T^2} C^3 \cdot v}_{\rnum{9}}.
	\end{align*}
	So finally
	\begin{align*}
		&\rnum{1} = \rnum{2} - \rnum{3} + \rnum{4} + \rnum{5} - \rnum{6} + \rnum{7} - \rnum{8} - \rnum{9}\\
		\iff\quad &\brac{\rnum{2} + \rnum{5} + \rnum{7}} + \rnum{4}
			= \rnum{6} + \rnum{1} + \rnum{3} + \rnum{9} + \rnum{8}\\
		\iff\quad &\Dt\brac{
				\int_\Omega \frac{1}{2} \abs{v}^2 J
				+ \quadw{\eta}\brac{\zeta}
				+ \int_{\T^2} \frac{g}{2} \abs{\zeta}^2
			}
			+ \brac{
				\int_\Omega \frac{1}{2} \abs{\symgrad^\geo v}^2 J
			}
			= \\
			&= \quadw{\dot{\eta}}\brac{\zeta}
			+ \int_\Omega \brac{C^1 \cdot v} J
			+ \int_\Omega C^2 q J
			+ \int_{\T^2} C^3 \cdot v
			+ \int_{\T^2} C^4\brac{\svr{\eta}+g} \zeta.
	\end{align*}
\end{proof}
Using the notation from the sketch in Section \ref{sec:DiscussDiff}, we can rephrase Proposition \ref{prop:HigherOrderEnergyDissipationEstimates} as follows:
if $\x_0$ and $\y$ solve $L_{\x_0} \brac{\y} = C$, then $\Dt \engeo \brac{\y;\x_0} + D \brac{\y;\x_0} = \abrac{C,\y}_{\x_0}$.  We thus seek to compute $C^\alpha = L_{\x} \brac{\partial^\alpha\x} - \partial^\alpha\brac{N\brac{\x}}$.

As discussed in Section \ref{sec:DiscussDiff}, the `commutator' $C^\alpha$ is not quite equal to $\sbrac{N,\partial^\alpha}$ because of the `fully nonlinear' term coming from the surface energy.
In particular, the terms in $N$ are of two types: almost all terms can be written as non-constant coefficient linear operators which have a \emph{multilinear dependence on their coefficients},
and one term (coming from the surface energy) is `fully nonlinear' and cannot be written in that form.
For terms of the first type, we have genuine commutators, and these are easy to compute: if $L = \hat{L}\brac{\pi_a,\dots,\pi_k}$, then
	\begin{equation*}
		\sbrac{\partial^\alpha,L}
		= \sum_{\substack{
			\beta + \sum_{i=1}^{k} \gamma_i = \alpha\\
			\beta < \alpha
		}}
		\hat{L}\brac{\partial^{\gamma_1}\pi_1,\dots,\partial^{\gamma_k}\pi_k} \circ \partial^\beta.
	\end{equation*}
See Proposition \ref{prop:commLinOpMultilinDepParam} for the full computations.
For the term of the second type, we do not compute
	\begin{equation*}
		\sbrac{\nugeo \fvr, \partial^\alpha}
		=	\brac{\nugeo \fvr} \circ \partial^\alpha
			- \partial^\alpha \circ \brac{\nugeo \fvr}
	\end{equation*}
but instead compute
	\begin{equation*}
		\mathcal{C}^{\will, \alpha} \brac{\eta}
		\defeq	\Bigbrac{
			\brac{\nugeo \svr{\eta}} \circ \partial^\alpha
			- \partial^\alpha \circ \brac{\nugeo \fvr}
		} \brac{\eta}.
	\end{equation*}
\begin{remark}
	\label{rmk:fakeCommutators}
	Using $\svr{\eta}$, as opposed to $\svr{0}$ in the differentiated version of the PDE is natural since
	it is precisely this operator which appears when differentiating $\fvr$, i.e. since
	\[
		\partial^\alpha \brac{\fvr\brac{\eta}}
		= \brac{\svr{\eta}} \brac{\partial^\alpha \eta}.
	\]
	Using $\svr{0}$ instead of $\svr{\eta}$ would also make it difficult to close the estimates since it would yield
	(due to commutators arising when differentiating the PDE in time) interactions of the form
	\[
		\int_{\T^2} \underbrace{
			\brac{\svr{\eta} - \svr{0}}
		}_{\brac{\star}}
		\brac{\pdt\eta} \brac{\tr\pdt u}
	\]
	where typically, i.e. unless the surface energy density $f$ has a special structure, $\brac{\star}$ involves fourth-order derivatives.
	For example, in the case of the `scalar' Willmore energy, i.e.
	\[
		\will\brac{\eta} \defeq \int_{\T^2} m\brac{\nabla\eta} {\vbrac{\Delta\eta}}^2
	\]
	for some smooth $m : \R^2 \to \brac{0,\infty}$ with $m\brac{0} > 0$, we have that
	\begin{align*}
		\brac{\svr{\eta}-\svr{0}}\phi
		= \brac{m\brac{\nabla\eta}-m\brac{0}} \Delta^2 \phi
		+ 2 \nabla\brac{m\brac{\nabla\eta}} \cdot \nabla\Delta\phi
		+ \Delta\brac{m\brac{\nabla\eta}} \Delta\phi.
	\end{align*}
	In general
	\begin{align*}
		\brac{\svr{\eta}-\svr{0}}\phi
		= \jet^* \brac{
			\brac{ \nabla^2 f \brac{\jet\eta} - \nabla^2 f \brac{0} }
			\bullet \jet\phi
		}
	\end{align*}
	which (again, unless $f$ has some special structure) typically involves fourth-order derivatives of $\phi$.
	Such interactions would be troublesome because they would thus take the form
	\[
		\int_{\T^2} \brac{\nabla^4 \pdt\eta} \brac{\tr\pdt u} \brac{\text{l.o.t.}}
	\]
	for some lower order terms that could be controlled via the energy.
	Terms like this cannot be controlled in our scheme of a priori estimates because we have insufficient control of $\pdt\eta$ and $\pdt u$,
	since we only know that
	\[
		\dsimp \gtrsim \normtypdom{\nabla^4 \pdt\eta}{H}{-3/2}{\T^2}^2  + \normtypdom{\tr\pdt u}{H}{1/2}{\T^2}^2.
	\]
\end{remark}
The detailed computations of $\mathcal{C}^{\will, \alpha} \brac{\eta}$ are in Lemma \ref{lemma:computingCommSurfEner}.
Putting it all together, we obtain that:
\begin{equation}\label{ed_alpha_commutators}
\begin{split}
	{\abrac{C^\alpha, \partial^\alpha \x}}_\x
	&= \quadw{\dot{\eta}} \brac{\partial^\alpha \eta}
		- \int_\Omega
			\brac{
				\sbrac{\pdt\Phi\cdot\nabla^\geo, \partial^\alpha} u
			}
			\cdot \brac{\partial^\alpha u} J
		+ \int_\Omega
			\brac{
				\sbrac{u\cdot\nabla^\geo, \partial^\alpha} u
			}
			\cdot \brac{\partial^\alpha u} J
		\\&\quad
		- \int_\Omega
			\brac{
				\sbrac{\brac{\nabla^\geo \cdot \geo^T}\cdot\nabla, \partial^\alpha} u
			}
			\cdot \brac{\partial^\alpha u} J
		- \int_\Omega
			\brac{
				\sbrac{\brac{\geo^T\cdot\geo}:\nabla^2, \partial^\alpha} u
			}
			\cdot \brac{\partial^\alpha u} J
		\\&\quad
		+ \int_\Omega
			\brac{
				\sbrac{\nabla^\geo, \partial^\alpha} p
			}
			\cdot \brac{\partial^\alpha u} J
		+ \int_\Omega
			\brac{
				\sbrac{\divgeo, \partial^\alpha} u
			}
			\brac{\partial^\alpha p} J
		\\&\quad
		+ \int_{\T^2}
			\brac{
				\sbrac{\nugeo\cdot\symgrad^\geo, \partial^\alpha} u
			}
			\cdot \partial^\alpha u
		- \int_{\T^2}
			\brac{
				\sbrac{\nugeo, \partial^\alpha} p
			}
			\cdot \partial^\alpha u
		\\&\quad
		+ g \int_{\T^2}
			\brac{
				\sbrac{\nugeo, \partial^\alpha} \eta
			}
			\cdot \partial^\alpha u
		+ \int_{\T^2}
			\mathcal{C}^{\will,\alpha} \brac{\eta}
			\cdot \partial^\alpha u
		\\&\quad
		- \int_{\T^2}
			\brac{
				\sbrac{\nugeo\cdot, \partial^\alpha} u
			}
			\Bigbrac{
				\brac{\svr{\eta} + g}\brac{\partial^\alpha \eta}
			}\\
	&\eqdef
		  \rnum{1}
		+ \rnum{2}
		+ \rnum{3}
		+ \rnum{4}
		+ \rnum{5}
		+ \rnum{6}
		+ \rnum{7}
		+ \rnum{8}
		+ \rnum{9}
		+ \rnum{10}
		+ \rnum{11}
		+ \rnum{12}.
\end{split}
\end{equation}
The following lemma shows how these terms may be estimated.

\begin{lemma}
\label{lemma:smallnessEstimateCommutators}
If the small energy assumptions hold (see Definition \ref{def:smallEnergyRegime}), then there are functionals $\mathcal{C}^1, \mathcal{C}^2$ such that
	\begin{equation*}
		\sum_{\parabolicOrder{\alpha} \leq 2} \abrac{C^\alpha, \partial^\alpha \x}_\x
		= \mathcal{C}^1 + \Dt \mathcal{C}^2
	\end{equation*}
	with
	\begin{equation*}
		\abs{\mathcal{C}^1} \lesssim \sqrt\enimp \dsimp
		\text{ and }
		\abs{\mathcal{C}^2} \lesssim \sqrt\enimp \enimp.
	\end{equation*}
\end{lemma}
\begin{proof}
	We begin with a sketch of the general argument.  Most of the commutators appearing in $\rnum{1}-\rnum{12}$ in \eqref{ed_alpha_commutators} are multilinear in terms of quantities that we control (such as the unknowns $u$, $p$, $\eta$ and geometric coefficients $J$, $\geo$, $\Phi$, $\nugeo$).  To handle such commutators, we use the H\"{o}lder and Sobolev inequalities. See Proposition \ref{prop:controlInteractSobNorm} for how we control terms of the form
	\[
		\vbrac{\int f_1 \dots f_k}
	\]
	when we control the $f_i$'s in some $H^{s_i}$ spaces.

	In some cases, we may need to use a couple of other tools to be able to place functions in Sobolev spaces of sufficiently high regularity.  We may need to `borrow' regularity, i.e. use that $H^{s+\alpha} \brac{\R^n} \cdot H^{s+\beta} \brac{\R^n} \hookrightarrow H^s \brac{\R^n}$: see Propositions \ref{prop:ProdEstSobSpaces}, \ref{prop:prodEstSobCtsMult}, and \ref{prop:prodEstSobBothFactors}.  We also need to use post-composition results, i.e. use that $C^{k,\alpha} \brac{ H^s \brac{\R^n} } \hookrightarrow H^s \brac{\R^n}$:	see Proposition \ref{prop:postCompEstSob}.  

	For a few commutators, namely $\rnum{11}$ and $\rnum{12}$, we will need to use the smallness and boundedness of variations of the surface energy, i.e. Lemmas \ref{lemma:smallFirstVar}, \ref{lemma:boundedSecondVar}, and \ref{lemma:boundedThirdVar}.

	Estimates of these forms ultimately contribute to $\mathcal{C}^1$.  We now turn to the question of how $\mathcal{C}^2$ arises.  The term $\rnum{7}$ involves an appearance of $\partial_t p$, which is not controlled in either the energy or dissipation, though it is defined through the local existence theory in a manner that allows us to integrate by parts in time:
	\[
		\int_\Omega \brac{\pdt p} w
		= \Dt \brac{\int_\Omega p w}
		- \int_\Omega p \brac{\pdt w}.
	\]
	Note that the non-time-differentiated term can be controlled by $\sqrt{\enimp} \dsimp$ like any of the other commutators contributing to $\mathcal{C}^1$, but the time-differentiated term must be controlled at a lower regularity level by $\enimp^{3/2}$.  In particular, the term of the form $\int_\Omega pw$ arising from commutator $\rnum{7}$ is the only contribution to $\mathcal{C}^2$.   
	
	We now provide detailed proofs for the estimates of four terms that are particularly delicate.  For example, three of them are `critical' in the sense that they lead to a full factor of $\dsimp$ appearing, suggesting that they are precisely at the limit of what the improved energy and dissipation allow us to control. Moreover, these four terms are representative of various difficulties encountered. We thus detail how to control:
	\begin{enumerate}
		\item	the commutator $\rnum{1}$ when $\partial^\alpha = \pdt$
			since it highlights how to handle terms of the form $\abs{\int f_1 \dots f_k}$,
		\item	the commutator $\rnum{7}$ when $\partial^\alpha = \pdt$ since this is precisely the term that requires integration by parts in time in order to be brought under control,
		\item	the commutator $\rnum{11}$ when $\partial^\alpha = \nablatwo^2$
			since it requires intermediate results about the smallness of $\fvr$, $\svr{\eta}$, and $\hovr{3}{\eta}$,
			and since it highlights how post-composition and product estimates in Sobolev spaces are used, and
		\item	the commutator $\rnum{12}$ when $\partial^\alpha = \pdt$, for the same reasons.
	\end{enumerate}
	Estimating the remaining commutators follows a similar procedure and thus we omit those estimates.
	\begin{enumerate}
		\item	\emph{A typical estimate on the surface.} We detail how to control the commutator $\rnum{1}$ when $\partial^\alpha = \pdt$.
			The commutator is
			\[
				\quadw{\dot{\eta}} \brac{\pdt\eta}
				= \frac{1}{2} \int_{\T^2}
				\nabla^3 f\brac{\jet\eta} \cdot
				\brac{
					\brac{\jet\pdt\eta}^{\otimes 3},
				}
			\]
			and it can be controlled as follows:
			\begin{align*}
				\abs{ \quadw{\dot\eta} \brac{\pdt\eta} }
				&\lesssim
					\normtypdom{\nabla^3 f\brac{\jet\eta}}{L}{\infty}{\T^2}
					\normtypdom{\jet\pdt\eta}{L}{2}{\T^2}
					\normtypdom{\jet\pdt\eta}{L}{4}{\T^2}^2\\
				&\lesssim
					C^{\brac{3}}_f
					\normtypdom{\pdt\eta}{H}{2}{\T^2}
					\normtypdom{\jet\pdt\eta}{H}{1/2}{\T^2}^2\\
				&\lesssim
					\sqrt\enimp
					\normtypdom{\pdt\eta}{H}{5/2}{\T^2}^2
				\quad
				\lesssim
					\sqrt\enimp \dsimp.
			\end{align*}
			Recall that $C^{\brac{3}}_f$ is defined in Definition \ref{def:universalConstants}.
		\item	\emph{Integration by parts in time.} We detail how to control the commutator $\rnum{7}$ when $\partial^\alpha = \pdt$. The commutator is
			\[
				\int_\Omega \brac{\pdt \geo} : \brac{\nabla u} \brac{\pdt p} J.
			\]
			Schematically, we have:
			\[
				\int_\Omega \brac{\pdt p} w
				= \Dt \brac{\int_\Omega p w} - \int_\Omega p \brac{\pdt w}
			\]
			where we may \emph{only} use the energy (and not the dissipation) to control $\int_\Omega p w$ since it is time-differentiated,
			and where we may proceed as usual, i.e. using both the energy and the dissipation, but not using the dissipation more than twice, to control $\int_\Omega p\brac{\pdt w}$.
			The first term is
			\[
				\int_\Omega \brac{\pdt\geo : \nabla u} p J,
			\]
			and it can be estimated in the following way:
			\begin{align*}
				\abs{\dots}
				&\lesssim
					\normtypdom{\pdt\geo}{L}{3}{\Omega}
					\normtypdom{\nabla u}{L}{3}{\Omega}
					\normtypdom{p}{L}{3}{\Omega}
					\normtypdom{J}{L}{\infty}{\Omega}\\
				&\lesssim
					\normtypdom{\pdt\geo}{H}{1/2}{\Omega}
					\normtypdom{\nabla u}{H}{1/2}{\Omega}
					\normtypdom{p}{H}{1/2}{\Omega}
					\normtypdom{J}{H}{3/2+}{\Omega}\\
				&\lesssim
					\brac{1 + \sqrt\enimp} \enimp^{3/2}
				\quad\lesssim
					\enimp^{3/2}.
			\end{align*}
			The second term is
			\[
				\int_\Omega \brac{\pdt^2 \geo : \nabla u} J p
				+ \int_\Omega \brac{\pdt\geo : \nabla\pdt u } J p
				+ \int_\Omega \brac{\pdt\geo : \nabla u} \brac{\pdt J} p,
			\]
			and can be estimated in the following way:
			\begin{align*}
				\abs{\dots}
				&\lesssim
					\normtypdom{\pdt^2 \geo}{L}{2}{\Omega}
					\normtypdom{\nabla u}{L}{6}{\Omega}
					\normtypdom{J}{L}{\infty}{\Omega}
					\normtypdom{p}{L}{6}{\Omega}\\
				&\quad
					+ \normtypdom{\pdt\geo}{L}{3}{\Omega}
					\normtypdom{\nabla \pdt u}{L}{3}{\Omega}
					\normtypdom{J}{L}{\infty}{\Omega}
					\normtypdom{p}{L}{3}{\Omega}\\
				&\quad
					+ \normtypdom{\pdt\geo}{L}{3}{\Omega}
					\normtypdom{\nabla u}{L}{3}{\Omega}
					\normtypdom{\pdt J}{L}{\infty}{\Omega}
					\normtypdom{p}{L}{3}{\Omega}
					\\
				&\lesssim
					\normtypdom{\pdt^2 \geo}{L}{2}{\Omega}
					\normtypdom{\nabla u}{H}{1}{\Omega}
					\normtypdom{J}{H}{3/2+}{\Omega}
					\normtypdom{p}{H}{1}{\Omega}\\
				&\quad
					+ \normtypdom{\pdt\geo}{H}{1/2}{\Omega}
					\normtypdom{\nabla \pdt u}{H}{1/2}{\Omega}
					\normtypdom{J}{H}{3/2+}{\Omega}
					\normtypdom{p}{H}{1/2}{\Omega}\\
				&\quad
					+ \normtypdom{\pdt\geo}{H}{1/2}{\Omega}
					\normtypdom{\nabla u}{H}{1/2}{\Omega}
					\normtypdom{\pdt J}{H}{3/2+}{\Omega}
					\normtypdom{p}{H}{1/2}{\Omega}
					\\
				&\lesssim
					\sqrt\dsimp \brac{1+\sqrt\enimp} \enimp
					+ \brac{1 + \sqrt\enimp} \enimp^{3/2}
					+ \enimp^2
				\quad\lesssim
					\brac{1 + \enimp} \enimp \sqrt\dsimp.
			\end{align*}
		\item	\emph{Another typical estimate on the surface.} We detail how to control the commutator $\rnum{11}$ when $\partial^\alpha = \nablatwo^2$.  The commutator is
			\begin{align*}
				\int_{\T^2} \brac{\nabla^3 \eta} \fvr\brac{\eta} \brac{\tr\nabla^2 u}
				+ \int_{\T^2} \brac{\nabla^2 \eta} \brac{\brac{\svr{\eta}}\brac{\nabla\eta}} \brac{\tr\nabla^2 u}
				\\
				+ \int_{\T^2} \nugeo \brac{\brac{\hovr{3}{\eta}}\brac{\nabla\eta,\nabla\eta}} \brac{\tr\nabla^2 u}
				\\
				\eqdef \rnum{11}_1 + \rnum{11}_2 + \rnum{11}_3.
			\end{align*}
			The first two terms can be estimated in the following way:
			\begin{align*}
				\abs{\rnum{11}_1 + \rnum{11}_2}
				&\lesssim
					\normtypdom{\nabla^3 \eta}{L}{\infty}{\T^2}
					\normtypdom{\fvr\brac{\eta}}{L}{2}{\T^2}
					\normtypdom{\tr \nabla^2 u}{L}{2}{\T^2}
				\\&\quad
					+ \normtypdom{\nabla^2 \eta}{L}{\infty}{\T^2}
					\normtypdom{\brac{\svr{\eta}}\brac{\nabla\eta}}{L}{2}{\T^2}
					\normtypdom{\nabla^2 u}{L}{2}{\T^2}\\
				&\lesssim
					\normtypdom{\nabla^3 \eta}{H}{1+}{\T^2}
					\normtypdom{\fvr\brac{\eta}}{L}{2}{\T^2}
					\normtypdom{\nabla^2 u}{H}{1/2}{\Omega}
				\\&\quad
					+ \normtypdom{\nabla^2 \eta}{H}{1+}{\T^2}
					\normtypdom{\brac{\svr{\eta}}\brac{\nabla\eta}}{L}{2}{\T^2}
					\normtypdom{\nabla^2 u}{H}{1/2}{\Omega}\\
				&\lesssim
					\sqrt\enimp \sqrt\enimp \sqrt\dsimp
					+ \sqrt\enimp \sqrt\dsimp \sqrt\dsimp
				\quad\lesssim
					\enimp \sqrt\dsimp 
					+ \sqrt\enimp \dsimp,
			\end{align*}
			where we have used that 
			$\normtypdom{\fvr\brac{\eta}}{H}{1/2}{\T^2} \lesssim \sqrt\enimp$,
			and that
			\begin{align*}
				\normtypdom{\brac{\svr{\eta}}\brac{\nabla\eta}}{L}{2}{\T^2}
				&= \normtypdom{\jet^* \brac{
					\nabla^2 f\brac{\jet\eta} \bullet \jet\nabla\eta
				}}{L}{2}{\T^2}\\
				&\hspace{-2cm}\lesssim
				\normtypdom{\nabla^2 f \brac{\jet\eta}\bullet\jet\nabla\eta}{H}{2}{\T^2}\\
				&\hspace{-2cm}\lesssim
				\normtypdom{\nabla^2 f \brac{\jet\eta}}{H}{2}{\T^2}
				\normtypdom{\jet\nabla\eta}{H}{2}{\T^2}\\
				&\hspace{-2cm}\lesssim
				\brac{
					C^{\brac{2}}_f
					+ C^{\brac{5}}_f
					\brac{
						\normtypdom{\jet\eta}{H}{2}{\T^2}
						+ \normtypdom{\jet\eta}{H}{2}{\T^2}^2
					}
				}
				\normtypdom{\eta}{H}{5}{\T^2}\\
				&\hspace{-2cm}\lesssim
					\brac{1
						+ \normtypdom{\eta}{H}{4}{\T^2}
						+ \normtypdom{\eta}{H}{4}{\T^2}^2
					}
					\sqrt\dsimp
				\quad\lesssim \sqrt\dsimp,
			\end{align*}
			recalling that $C^{\brac{3}}_f$ is defined in Definition \ref{def:universalConstants}.
			The last term requires a bit more precaution:
			\begin{align*}
				\abs{\rnum{11}_3}
				&\lesssim
					\normtypdom{
						\nugeo \brac{\tr\nabla^2 u}
					}{H}{1/2}{\T^2}
					\normtypdom{
						\brac{\hovr{3}{\eta}} \brac{\nabla\eta,\nabla\eta}
					}{H}{-1/2}{\T^2}\\
				&\lesssim
					\normtypdom{\nugeo}{H}{\frac{3}{2}+}{\T^2}
					\normtypdom{\tr\nabla^2 u}{H}{1/2}{\T^2}
					\normtypdom{
						\nabla^3 f \brac{\jet\eta} \bullet \brac{\jet\eta \otimes \jet\eta}
					}{H}{3/2}{\T^2}\\
				&\lesssim
					\brac{1+\sqrt\enimp}
					\sqrt\dsimp
					\normtypdom{
						\nabla^3 f\brac{\jet\eta}
					}{H}{3/2}{\T^2}
					\normtypdom{\jet\nabla\eta}{H}{3/2}{\T^2}^2\\
				&\lesssim
					\brac{1+\sqrt\enimp} \sqrt\dsimp \sqrt\enimp \enimp
				\quad\lesssim
					\brac{1+\sqrt\enimp} \enimp^{3/2} \sqrt\dsimp
				\quad\lesssim
					\enimp \sqrt\dsimp,
			\end{align*}
			where we have used that
			\begin{align*}
				\normtypdom{\nabla^3 f\brac{\jet\eta}}{H}{3/2}{\T^2}
				&\lesssim
					C^{\brac{3}}_f
					+ C^{\brac{5}}_f
					\brac{
						\normtypdom{\jet\eta}{H}{3/2}{\T^2}
						+ \normtypdom{\jet\eta}{H}{3/2}{\T^2}^2
					}\\
				&\lesssim 1 + \sqrt\enimp + \enimp^{3/2}
				\quad\lesssim 1.
			\end{align*}
		\item	\emph{One last typical estimate on the surface.}  We detail how to control the commutator $\rnum{12}$ when $\partial^\alpha = \pdt$.
			The commutator is
			\[
				\int_{\T^2} \brac{\nabla\pdt\eta} \brac{\tr u} \brac{\brac{\svr{\eta} + g}\brac{\pdt\eta}},
			\]
			and it can be estimated in the following way
			\begin{align*}
				\abs{\dots}
				&\lesssim
					\normtypdom{\brac{\nabla\pdt\eta}\brac{\tr u}}{H}{3/2}{\T^2}
					\normtypdom{\brac{\svr{\eta} + g}\brac{\pdt\eta}}{H}{-3/2}{\T^2}\\
				&\lesssim
					\normtypdom{\nabla\pdt\eta}{H}{3/2}{\T^2}
					\normtypdom{\tr u}{H}{3/2}{\T^2}
					\brac{
						\norm{\svr{\eta}}{\mathcal{L}\brac{H^{5/2};\,H^{-3/2}}}
						+ 1
					}
					\normtypdom{\pdt\eta}{H}{5/2}{\T^2}\\
				&\lesssim
					\sqrt\dsimp \sqrt\enimp \sqrt\dsimp
				\quad\lesssim
					\sqrt\enimp \dsimp.
			\end{align*}
	\end{enumerate}
\end{proof}

\subsection{Regularity gain}
\label{sec:nonlinearCorr}
In this section we record the auxiliary estimates arising from the linearized problem (about the equilibrium) in Proposition \ref{prop:GenericFormAuxiliaryEstimates}, we compute the nonlinear remainders obtained when writing the full nonlinear problem as a perturbation of its linearization, and finally we estimate these nonlinear remainders in Lemma \ref{lemma:smallnessEstimateNonlinearRemainder}.

We begin by recording our auxiliary estimates in a general form.

\begin{prop}[Generic form of the auxiliary estimates]\label{prop:GenericFormAuxiliaryEstimates}
	Let $R = \brac{R^1, R^2, R^3, R^4}$ be given and suppose that $\brac{u,p,\eta}$ solves
	\[
	\begin{cases}
		\pdt u - \Delta u + \nabla p = R^1					&\text{in }\Omega,\\
		\nabla\cdot u = R^2							&\text{in }\Omega,\\
		\brac{\svr{0} + g}\eta e_3 + \symgrad u \cdot e_3 - p e_3 = R^3		&\text{on }\Sigma,\\
		\pdt\eta - u\cdot e_3 = R^4						&\text{on }\Sigma\text{, and}\\
		u = 0									&\text{on }\Sigma_b.
	\end{cases}
	\]
	Then
	\footnote{Note that the terms $\norm{\nablatwo u}{L^2}$ and $\norm{\nablatwo^2 u}{L^2}$ are present in $\eneq$ but are absent from the right-hand side of the estimate.}
	\begin{equation}\label{eq:GenericFormAuxiliaryEstimates_1}
	\begin{split}
		&
			\normtypdom{u}{H}{2}{\Omega}
			+ \normtypdom{\pdt u}{L}{2}{\Omega}
			+ \normtypdom{p}{H}{1}{\Omega}
			+ \normtypdom{\eta}{H}{9/2}{\T^2}
			+ \normtypdom{\pdt \eta}{H}{2}{\T^2}
		\\&
		\hspace{1cm}
		\lesssim
			\normtypdom{\pdt u}{L}{2}{\Omega}
			+ \normtypdom{\eta}{H}{4}{\T^2}
			+ \normtypdom{\pdt\eta}{H}{2}{\T^2}
			+ \normtypdom{u}{L}{2}{\Omega}
		\\&
		\hspace{1cm}
			+ \normtypdom{R^1}{L}{2}{\Omega}
			+ \normtypdom{R^2}{H}{1}{\Omega}
			+ \normtypdom{R^3}{H}{1/2}{\T^2}
			+ \normtypdom{R^4}{H}{3/2}{\T^2}
	\end{split}		
	\end{equation}
	and
	\footnote{Note that the term $\norm{\symgrad \nablatwo u}{L^2}$ is present in $\dseq$ but are absent from the right-hand side of the estimate.}
	\begin{equation}\label{eq:GenericFormAuxiliaryEstimates_2}
	\begin{split}
		&
			\normtypdom{u}{H}{3}{\Omega}
			+ \normtypdom{\pdt u}{H}{1}{\Omega}
			+ \normtypdom{p}{H}{2}{\Omega}
			+ \normtypdom{\eta}{H}{11/2}{\T^2}
			+ \normtypdom{\pdt\eta}{H}{5/2}{\T^2}
			+ \normtypdom{\pdt^2 \eta}{H}{1/2}{\T^2}
		\\&
		\hspace{1cm}
		\lesssim
			\normtypdom{\symgrad u}{L}{2}{\Omega}
			+ \normtypdom{\symgrad\pdt u}{L}{2}{\Omega}
			+ \normtypdom{\symgrad\nablatwo^2 u}{L}{2}{\Omega}
		\\&
		\hspace{1cm}
			+ \normtypdom{R^1}{H}{1}{\Omega}
			+ \normtypdom{R^2}{H}{2}{\Omega}
			+ \normtypdom{R^3}{H}{3/2}{\T^2}
			+ \normtypdom{R^4}{H}{5/2}{\T^2}
			+ \normtypdom{\pdt R^4}{H}{1/2}{\T^2}
	\end{split}
	\end{equation}
	i.e.
	\begin{equation*}
		\begin{cases}
			\enimp \lesssim \eneq + \rem_E\text{ and}\\
			\dsimp \lesssim \dseq + \rem_D\\
		\end{cases}
	\end{equation*}
	for
	\[
	\begin{cases}
		\rem_E \defeq 
			\normtypdom{R^1}{L}{2}{\Omega}^2
			+ \normtypdom{R^2}{H}{1}{\Omega}^2
			+ \normtypdom{R^3}{H}{1/2}{\T^2}^2
			+ \normtypdom{R^4}{H}{3/2}{\T^2}^2 \\
		\rem_D \defeq 
			\normtypdom{R^1}{H}{1}{\Omega}^2
			+ \normtypdom{R^2}{H}{2}{\Omega}^2
			+ \normtypdom{R^3}{H}{3/2}{\T^2}^2
			+ \normtypdom{R^4}{H}{5/2}{\T^2}^2
			+ \normtypdom{\pdt R^4}{H}{1/2}{\T^2}^2.
	\end{cases}
	\]
\end{prop}
\begin{proof}
We begin with the estimates related to the energy.  We divide the argument into several steps.
\begin{enumerate}
		\item	We initiate our scheme of estimates in the usual way for parabolic problems: treat temporal derivatives as forcing terms in the stationary equations
			in order to recover control of the spatial derivatives from control of the temporal derivatives.
			In particular, note that $\brac{u,p,\eta}$ solves a Stokes problem with mixed boundary conditions where $\pdt u$ and $\pdt \eta$ are treated as forcing terms, i.e.
			\[
			\begin{cases}
				-\Delta u + \nabla p = - \pdt u + R^1			&\text{in }\Omega,\\
				\nabla\cdot u = R^2					&\text{in }\Omega,\\
				u \cdot e_3 = \pdt\eta - R^4				&\text{on }\Sigma,\\
				{\brac{\symgrad u  e_3}}_{tan} = {\brac{R^3}}_{tan}	&\text{on }\Sigma\text{, and}\\
				u = 0							&\text{on }\Sigma_b,
			\end{cases}
			\]
			where for any vector field $w:\Sigma\to\R^3$ we denote by $w_{tan}$ its tangential part, i.e. $w_{tan} = \brac{I - e_3\otimes e_3} w.$
			Therefore, by using elliptic regularity estimates for the Stokes problem (i.e. the auxiliary estimate \ref{eqEstStokesProbMixBC}) we obtain that
			\begin{align*}
					\normtypdom{u}{H}{2}{\Omega}
					+ \normtypdom{\nabla p}{L}{2}{\Omega}
				\lesssim
					\normtypdom{-\pdt u + R^1}{L}{2}{\Omega}
					+ \normtypdom{R^2}{H}{1}{\Omega}
					+ \normtypdom{\pdt\eta - R^4}{H}{3/2}{\T^2}
					+ \normtypdom{{\brac{R^3}}_{tan}}{H}{1/2}{\T^2}
				\\
				\quad\leqslant
					\normtypdom{\pdt u}{L}{2}{\Omega}
					+ \normtypdom{\pdt\eta}{H}{3/2}{\T^2}
					+ \normtypdom{R^1}{L}{2}{\Omega}
					+ \normtypdom{R^2}{H}{1}{\Omega}
					+ \normtypdom{{\brac{R^3}}_{tan}}{H}{1/2}{\T^2}
					+ \normtypdom{R^4}{H}{3/2}{\T^2}.
			\end{align*}
		\item	Ultimately, we wish to control the full $H^1$ norm of $p$ via the improved energy, but so far we only control the gradient of $p$.
			In order to proceed further we therefore use the normal component of the dynamic boundary condition to obtain control of the trace of $p$ on the top boundary.
			Indeed, since
			\[
				p = \symgrad u : \brac{e_3 \otimes e_3} + \brac{\svr{0} + g}\eta - R^3 \cdot e_3
				\qquad\text{on }\T^2\brac{\sim\Sigma}
			\]
			it follows that
			\begin{align*}
					\normtypdom{\tr_\Sigma p}{L}{2}{\T^2}
				&\lesssim
					\normtypdom{\tr_\Sigma \symgrad u}{L}{2}{\T^2}
					+ \normtypdom{\eta}{H}{4}{\T^2}
					+ \normtypdom{R^3\cdot e_3}{L}{2}{\T^2}
				\\
				&\lesssim
					\normtypdom{u}{H}{3/2}{\Omega}
					+ \normtypdom{\eta}{H}{4}{\T^2}
					+ \normtypdom{R^3\cdot e_3}{L}{2}{\T^2}.
			\end{align*}
		\item	We can now, as intended, recover control of the full $H^1$ norm of $p$ by using a Poincar\'{e}-type inequality (i.e. auxiliary estimate \ref{eqEstPoincareTypeTrace}):
			\[
					\normtypdom{p}{H}{1}{\Omega}
				\lesssim
					\normtypdom{\tr_\Sigma p}{L}{2}{\T^2}
					+ \normtypdom{\nabla p}{L}{2}{\Omega}.
			\]
		\item	Now that we have enough control on the stress tensor to obtain estimates for its trace onto the boundary,
			we can use the normal component of the dynamic boundary condition to obtain control of higher-order spatial derivatives of $\eta$.
			Indeed, since
			\[
				\brac{\svr{0} + g}\eta = p - \symgrad u : \brac{e_3 \otimes e_3} + R^3 \cdot e_3 
			\]
			it follows from the elliptic regularity of $\svr{0} + g$ (i.e. the auxiliary estimate \ref{eqEstDynBC}) that
			\begin{align*}
					\normtypdom{\eta}{H}{9/2}{\T^2}
				&\lesssim
					\normtypdom{\tr_\Sigma p}{H}{1/2}{\T^2}
					+ \normtypdom{\tr_\Sigma \symgrad u}{H}{1/2}{\T^2}
					+ \normtypdom{R^3 \cdot e_3}{H}{1/2}{\T^2}
				\\
				&\lesssim
					\normtypdom{p}{H}{1}{\Omega}
					+ \normtypdom{u}{H}{2}{\Omega}
					+ \normtypdom{R^3 \cdot e_3}{H}{1/2}{\T^2}.
			\end{align*}
	\end{enumerate}
	Assembling the above estimates, we see that
	\begin{align*}
			\normtypdom{u}{H}{2}{\Omega}
			+ \normtypdom{p}{H}{1}{\Omega}
			+ \normtypdom{\eta}{H}{9/2}{\T^2}
		\lesssim
			\normtypdom{\pdt u}{L}{2}{\Omega}
			+ \normtypdom{\pdt\eta}{H}{3/2}{\T^2}
			+ \normtypdom{\eta}{H}{4}{\T^2}
		\\
			+ \normtypdom{R^1}{L}{2}{\Omega}
			+ \normtypdom{R^2}{H}{1}{\Omega}
			+ \normtypdom{R^3}{H}{1/2}{\T^2}
			+ \normtypdom{R^4}{H}{3/2}{\T^2}.
	\end{align*}
	Then \eqref{eq:GenericFormAuxiliaryEstimates_1} follows immediately from this.

	We now turn our attention to estimates related to the dissipation.  Again, we divide the argument into steps.
	\begin{enumerate}
		\item	We begin by trading control of the symmetrized gradient for control of full $H^1$ norms.  This is possible due to the no-slip boundary conditions and a Korn-type inequality (i.e. auxiliary estimate \ref{eqEstKorn}):
			\[
			\begin{cases}
				\normtypdom{u}{H}{1}{\Omega}
				&\lesssim
				\normtypdom{\symgrad u}{L}{2}{\Omega},
				\\
				\normtypdom{\pdt u}{H}{1}{\Omega}
				&\lesssim
				\normtypdom{\symgrad\pdt u}{L}{2}{\Omega}\text{, and}
				\\
				\normtypdom{\nablatwo^2 u}{H}{1}{\Omega}
				&\lesssim
				\normtypdom{\symgrad \nablatwo^2 u}{L}{2}{\Omega}.
			\end{cases}
			\]
		\item	Next we use the fact that the horizontal derivatives of the trace of $u$ are equal to the trace of the horizontal derivatives,	i.e. $\nablatwo \circ \tr_\Sigma = \tr_\Sigma \circ \nablatwo$.  From this and standard trace estimates we obtain:
			\begin{align*}
					\normtypdom{\tr_\Sigma u}{H}{5/2}{\T^2}
				&\lesssim
					\normtypdom{\tr_\Sigma u}{H}{1/2}{\T^2}
					+ \normtypdom{\nablatwo^2 \brac{\tr_\Sigma u}}{H}{1/2}{\T^2}
				\\
				&\lesssim
					\normtypdom{u}{H}{1}{\Omega}
					+ \normtypdom{\tr_\Sigma \nablatwo^2 u}{H}{1/2}{\Omega}
				\\
				&\lesssim
					\normtypdom{u}{H}{1}{\Omega}
					+ \normtypdom{\nablatwo^2 u}{H}{1}{\Omega}.
			\end{align*}
		\item	We can now recover control of all the derivatives of $u$ by using the trace of $u$ as datum in a Stokes problem with Dirichlet boundary conditions.
			Indeed, since
			\[
			\begin{cases}
				-\Delta u + \nabla p = -\pdt u + R^1		&\text{in }\Omega,\\
				\nabla\cdot u = R^2				&\text{in }\Omega,\\
				u = u						&\text{on }\Sigma,\\
				u = 0						&\text{on }\Sigma_b
			\end{cases}
			\]
			it follows from elliptic regularity estimates for the Stokes problem (i.e. the auxiliary estimate \ref{eqEstStokesProbDirBC}) that 
			\begin{align*}
					\normtypdom{u}{H}{3}{\Omega}
					+ \normtypdom{\nabla p}{H}{1}{\Omega} 
				&\lesssim
					\normtypdom{-\pdt u + R^1}{H}{1}{\Omega}
					+ \normtypdom{R^2}{H}{2}{\Omega} 
					+ \normtypdom{\tr_\Sigma u}{H}{5/2}{\T^2}
				\\
				&\leqslant
					\normtypdom{\pdt u}{H}{1}{\Omega}
					+ \normtypdom{\tr_\Sigma u}{H}{5/2}{\T^2}
					+ \normtypdom{R^1}{H}{1}{\Omega}
					+ \normtypdom{R^2}{H}{2}{\Omega}.
			\end{align*}
		\item	Next we observe that
			\[
				\brac{\svr{0} + g}\brac{\nablatwo \eta}
				= \nablatwo p - \symgrad{\nablatwo u}:\brac{e_3 \otimes e_3} + \nablatwo R^3
				\qquad\text{on }\T^2\brac{\sim\Sigma}
			\]
			and therefore elliptic estimates for the operator $\svr{0} + g$  (i.e. the auxiliary estimate \ref{eqEstDynBC}) provide the bounds
			\begin{align*}
					\normtypdom{\nablatwo\eta}{H}{9/2}{\T^2}
				&\lesssim
					\normtypdom{\tr_\Sigma \nablatwo p}{H}{1/2}{\T^2}
					+ \normtypdom{\tr_\Sigma \symgrad\nablatwo u}{H}{1/2}{\T^2}
					+ \normtypdom{\nablatwo R^3}{H}{1/2}{\T^2}
				\\
				&\lesssim
					\normtypdom{\nabla p}{H}{1}{\Omega}
					+ \normtypdom{u}{H}{3}{\Omega}
					+ \normtypdom{\nablatwo R^3}{H}{1/2}{\T^2}.
			\end{align*}
			Moreover, since $\int_{\T^2} \eta = 0$, we have that $\norm{\eta}{H^{11/2}} \lesssim \norm{\nablatwo\eta}{H^{9/2}}$ (by auxiliary estimate \ref{eqEstPoincareTypeHs}),
			and so, finally, we have 
			\[
					\normtypdom{\eta}{H}{11/2}{\T^2}
				\lesssim
					\normtypdom{\nabla p}{H}{1}{\Omega}
					+ \normtypdom{u}{H}{3}{\Omega}
					+ \normtypdom{\nablatwo R^3}{H}{1/2}{\T^2}.
			\]
		\item	We now parlay the $\eta$ estimates into full $H^2$ control of the pressure by arguing as we did for the energy, obtaining control of the trace of the pressure. Since
			\[
				p = \symgrad u : \brac{e_3 \otimes e_3} + \brac{\svr{0} + g}\eta - R^3 \cdot e_3
				\qquad\text{on }\T^2\brac{\sim\Sigma},
			\]
			it follows that
			\begin{align*}
					\normtypdom{\tr_\Sigma p}{L}{2}{\T^2}
				&\lesssim
					\normtypdom{\tr_\Sigma \symgrad u}{L}{2}{\T^2}
					+ \normtypdom{\eta}{H}{4}{\T^2}
					+ \normtypdom{R^3}{L}{2}{\T^2}
				\\
				&\lesssim
					\normtypdom{u}{H}{3/2}{\Omega}
					+ \normtypdom{\eta}{H}{4}{\T^2}
					+ \normtypdom{R^3}{L}{2}{\T^2}.
			\end{align*}
		We then use a Poincare-type inequality	(i.e. auxiliary estimate \ref{eqEstPoincareTypeTrace}) to bound
			\[
					\normtypdom{p}{H}{2}{\Omega}
				\lesssim
					\normtypdom{\tr_\Sigma p}{L}{2}{\T^2}
					+ \normtypdom{\nabla p}{H}{1}{\Omega}.
			\]
		\item	Finally, we use the kinematic boundary condition and its time-differentiated version to obtain control of $\pdt\eta$ and $\pdt^2 \eta$.	Indeed, the kinematic boundary condition tells us that
			\[
				\pdt\eta = u\cdot e_3 + R^4
				\qquad\text{on }\T^2\brac{\sim\Sigma},
			\]
			and therefore
			\begin{align*}
					\normtypdom{\pdt\eta}{H}{5/2}{\T^2}
				&\lesssim
					\normtypdom{\tr_\Sigma u}{H}{5/2}{\T^2}
					+ \normtypdom{R^4}{H}{5/2}{\T^2}
				\lesssim
					\normtypdom{u}{H}{3}{\Omega}
					+ \normtypdom{R^4}{H}{5/2}{\T^2}.
			\end{align*}
			The time-differentiated kinematic boundary condition tells us that
			\[
				\pdt^2 \eta = \brac{\pdt u}\cdot e_3 + \pdt R^4
				\qquad\text{on }\T^2\brac{\sim\Sigma}
			\]
			and therefore
			\[
					\normtypdom{\pdt^2 \eta}{H}{1/2}{\T^2}
				\lesssim
					\normtypdom{\pdt u}{H}{1/2}{\Omega}
					+ \normtypdom{\pdt R^4}{H}{1/2}{\T^2}.
			\]
	\end{enumerate}
	Combining these estimates then shows that
	\begin{align*}
		&
			\normtypdom{u}{H}{3}{\Omega}
			+ \normtypdom{\pdt u}{H}{1}{\Omega}
			+ \normtypdom{p}{H}{2}{\Omega}
			+ \normtypdom{\eta}{H}{11/2}{\T^2}
			+ \normtypdom{\pdt\eta}{H}{5/2}{\T^2}
			+ \normtypdom{\pdt^2 \eta}{H}{1/2}{\T^2}
		\\&
		\hspace{1cm}
		\lesssim
			\normtypdom{\symgrad u}{L}{2}{\Omega}
			+ \normtypdom{\symgrad\pdt u}{L}{2}{\Omega}
			+ \normtypdom{\symgrad\nablatwo^2 u}{L}{2}{\Omega}
		\\&
		\hspace{1cm}
			+ \normtypdom{R^1}{H}{1}{\Omega}
			+ \normtypdom{R^2}{H}{2}{\Omega}
			+ \underbrace{
				\normtypdom{R^3}{L}{2}{\T^2}
				+ \normtypdom{\nabla R^3}{H}{1/2}{T^2}
			}_{
				\lesssim
				\normtypdom{R^3}{H}{3/2}{\T^2}
			}
			+ \normtypdom{R^4}{H}{5/2}{\T^2}
			+ \normtypdom{\pdt R^4}{H}{1/2}{\T^2},
	\end{align*}
	and then \eqref{eq:GenericFormAuxiliaryEstimates_2} follows immediately.
	
\end{proof}

Proposition \ref{prop:GenericFormAuxiliaryEstimates} tells us in which norm we need to be able to control the nonlinear remainders.
In the notation used in the sketch in Section \ref{sec:DiscussDiff}, these remainders $R$ are given by $R = \brac{L-N}\brac{\x}$.
Here $N$ corresponds to the system \eqref{NS_fixed_s}--\eqref{NS_fixed_e}, while $L$ corresponds to the system
\begin{subnumcases}{}
	\nonumber \pdt u + \diver S = 0				&in $\Omega$,\\
	\nonumber \diver u = 0					&in $\Omega$,\\
	\nonumber \brac{\svr{0} + g}\eta e_3 - S \cdot e_3 = 0	&on $\Sigma$,\\
	\nonumber \pdt \eta - u \cdot e_3 = 0			&on $\Sigma$, and\\
	\nonumber u = 0						&on $\Sigma_b$.
\end{subnumcases}
It follows that the remainders are given by
\begin{subnumcases}{}
	 R^1 = \brac{\pdtm^{u,\geo} u - \pdt u} + \brac{\divgeo S^\geo - \nabla\cdot S},  \label{rem_def_1} \\
	 R^2 = \divgeo u - \diver u, \label{rem_def_2}\\
	 R^3 = \Bigbrac{\fvr\brac{\eta} \nugeo - \brac{\svr{0}}\eta e_3} 
		+ g\eta\brac{\nugeo - e_3}
		- \brac{S^\geo \cdot \nugeo - S \cdot e_3}\text{, and} \label{rem_def_3}\\
	R^4 = u \cdot \brac{\nugeo - e_3} \label{rem_def_4}.
\end{subnumcases}

Before recording our estimates for these terms we discuss how to Taylor expand the surface energy terms.

\begin{remark}
	\label{rmk:nonlinearCorrTaylorExpSurfEnerDensity}
	An important subtetly in performing the estimates in this section arises from the fact that the surface energy density may be fully nonlinear. This plays a role in two terms in particular: $\fvr\brac{\eta}$ and $\brac{\fvr-\svr{0}}\brac{\eta}$.
	We write these terms in a manner more amenable to estimates by performing a Taylor expansion of $\nabla f$, i.e.
	\begin{itemize}
		\item	For $\fvr$:
			\begin{align*}
				\fvr\brac{\eta}
				&
				= \jet^* \brac{\nabla f \brac{\jet\eta}}
				= \jet^* \brac{\nabla f \brac{\jet\eta} - \nabla f \brac{0}}
				\\&
				= \jet^* \brac{ \int_0^1 \nabla^2 f\brac{t\jet\eta} dt \bullet \jet\eta}
				= \jet^* \brac{\tayh \brac{\jet\eta} \bullet \jet\eta},
			\end{align*}
			where
			\[
				\tayh\brac{z} \defeq \int_0^1 \nabla^2 f \brac{tz} dt
			\]
			for $z = \brac{\p,M} \in \R^n \times \R^{n \times n}$.
			Note that we may also write
			\[
				\fvr\brac{\eta}
				= \jet^* \brac{ \mathcal{R}_0 \sbrac{\nabla f, 0} \brac{\jet\eta} },
			\]
			where $\mathcal{R}_0$ is defined in Proposition \ref{prop:TaylorThm}.
			This is a useful way of writing $\fvr\brac{\eta}$ since it provides us with a unified way of estimating a certain number of terms showing up in the remainders.
		\item	For $\brac{\fvr - \svr{0}}\brac{\eta}$:
			\begin{align*}
				\fvr\brac{\eta} - \svr{0}\brac{\eta}
				&= \jet^* \Bigbrac{
					\nabla f \brac{\jet\eta} - \nabla^2 f \brac{0} \bullet \jet\eta
				}\\
				&= \jet^* \Bigbrac{
					\nabla f \brac{\jet\eta} - \mathcal{P}_1 \sbrac{\nabla f, 0} \brac{\jet\eta}
				}\\
				&= \jet^* \Bigbrac{
					\mathcal{R}_1 \sbrac{\nabla f, 0} \brac{\jet\eta}
				}\\
				&= \jet^* \brac{
					\brac{
						\half \int_0^1 \brac{1-t} \nabla^3 f \brac{t\jet\eta} dt
					} \bullet \brac{\jet\eta \otimes \jet\eta}
				}\\
				&= \jet^* \Bigbrac{
					\tayg\brac{\jet\eta} \bullet \brac{\jet\eta \otimes \jet\eta},
				}
			\end{align*}
			where
			\[
				\tayg\brac{z} \defeq \half \int_0^1 \brac{1-t} \nabla^3 f \brac{tz} dt
			\]
			for $z = \brac{\p,M} \in \R^n \times \R^{n \times n}$
			and where $\mathcal{P}_1$ and $\mathcal{R}_1$ are defined in Proposition \ref{prop:TaylorThm}.
	\end{itemize}
	Summarizing, we have:
	\begin{equation}\label{W_taylor_exp}
	\begin{cases}
		\fvr\brac{\eta}
		= \jet^* \brac{ \mathcal{R}_0 \sbrac{\nabla f, 0} \brac{\jet\eta} }
		= \jet^* \brac{ \tayh\brac{\jet\eta} \bullet \jet\eta}
		\\
		\brac{\fvr - \svr{0}}\brac{\eta}
		= \jet^* \brac{ \mathcal{R}_1 \sbrac{\nabla f, 0} \brac{\jet\eta} }
		= \jet^* \brac{ \tayg\brac{\jet\eta} \bullet \brac{\jet\eta \otimes \jet\eta}}
	\end{cases}
	\end{equation}
	where $\mathcal{R}_0$ and $\mathcal{R}_1$ are defined in Proposition \ref{prop:TaylorThm} and where
	\begin{equation*}
		\tayh\brac{z} \defeq r_0 \sbrac{\nabla f, 0} \brac{\jet\eta}
		= \displaystyle\int_0^1 \nabla^2 f \brac{tz} dt
	\quad\text{and}\quad
		\tayg\brac{z} \defeq r_1 \sbrac{\nabla f, 0} \brac{\jet\eta}
		= \half \displaystyle\int_0^1 \brac{1-t} \nabla^3 f \brac{tz} dt
	\end{equation*}
	for $z = \brac{\p,M} \in \R^n \times \R^{n \times n}$ and for $r_0$ and  $r_1$ defined in Proposition \ref{prop:TaylorThm}.
\end{remark}

Our next result records estimates for the remainder terms.

\begin{lemma}
\label{lemma:smallnessEstimateNonlinearRemainder}
	Let $\rem_E$ and $\rem_D$ be as defined in Proposition \ref{prop:GenericFormAuxiliaryEstimates}, and $R^1,R^2,R^3,R^4$ be as defined by \eqref{rem_def_1}--\eqref{rem_def_4}.  If the small energy assumptions hold (see Definition \ref{def:smallEnergyRegime}), then 
	\begin{equation*}
			\rem_E \lesssim \enimp^2 \text{ and }
			\rem_D \lesssim \enimp \dsimp.
	\end{equation*}
\end{lemma}
\begin{proof} 
	First let us sketch the argument.
	As in the proof of Proposition \ref{prop:GenericFormAuxiliaryEstimates}, 
	most terms are easily handled via the standard combination of H\"{o}lder and Sobolev inequalities (c.f. Proposition \ref{prop:controlInteractSobNorm})
	since they are multilinear, but some terms arising from the fully nonlinear surface energy have to be handled differently.
	Essentially, to control those, we make use of the fact that we are in a small energy regime and use Taylor expansions (c.f. Proposition \ref{prop:TaylorThm} for the notation used)
	to bring it back to the multilinear (i.e. polynomial) case.
	More precisely, the troublesome terms are $\fvr - \svr{0}$ and $\fvr$, which we handle by employing \eqref{W_taylor_exp}. 

	Let us now estimate each remainder in detail.
	$R^2$ and $R^4$ are easy to deal with since
	\begin{equation*}
		\nonumber R^2 = \brac{\geo - I}: \nabla u\text{ and } 
		\nonumber R^4 = - \brac{\tr u} \cdot \nabla\eta
	\end{equation*}
	and therefore we can use standard product estimates in Sobolev spaces (c.f. Propositions \ref{prop:ProdEstSobSpaces}, \ref{prop:prodEstSobCtsMult}, and \ref{prop:prodEstSobBothFactors}).

	$R^1$ is similar and only requires expanding out further before being estimated in the same way as $R^2$ and $R^4$ above:
	\begin{align*}
		R^1 = 
		&
			- \brac{\pdt\Phi} \cdot \geo \cdot \brac{\nabla u}^T
			+ u \cdot \geo \cdot \brac{\nabla u}^T
			+ \brac{\geo - I}\cdot\nabla p
		\\
		&\quad
			- \nabla \brac{\sym\brac{\brac{\nabla u}\cdot\brac{\geo-I}^T}} : \brac{\geo - I}
		\\
		&\quad
			- \brac{\nabla\symgrad u}:\brac{\geo - I}
			- \nabla\cdot\brac{\sym\brac{\brac{\nabla u}\cdot\brac{\geo-I}^T}}.
	\end{align*}
	$R^3$ requires more care, since it can be expanded out to be
	\begin{align*}
		R^3 =
		&
			- \brac{\fvr\brac{\eta} + g\eta}\brac{\nabla\eta}
			+ \brac{\fvr\brac{\eta} - \svr{0}\brac{\eta}} e_3
			+ p \nabla\eta
		\\
		&\quad
			- \sym\brac{ \brac{\nabla u} \geo^T } \cdot \brac{\nabla\eta}
			+ \sym\brac{ \brac{\nabla u} \brac{\geo - I} } \cdot e_3,
	\end{align*}
	where we have used that $R^3$ is defined on $\Sigma$ and $\nugeo\vert_\Sigma = -\nablatwoemb\eta + e_3$.
	Most terms in the expansion of $R^3$ can be handled by standard product estimates,
	but as sketched above, two terms require particular care, namely the ones involving $\fvr\brac{\eta}$ and $\fvr\brac{\eta} - \svr{0}\brac{\eta}$.

	According to \eqref{W_taylor_exp}, the key estimates required to control $\fvr$ and $\fvr-\svr{0}$ in $H^s$ are therefore the control of $\tayh\brac{\jet\eta}$ and $\tayg\brac{\jet\eta}$ in $H^s$.
	The details of this estimate rely on post-composition estimates in Sobolev spaces,
	and are recorded in the appendix in Lemma \ref{lemma:auxFuncTaylorExp} and Corollary \ref{cor:auxFuncTaylorExp}.
	From this we obtain that for any $s \geqslant 2$, 
	\[
	\begin{cases}
		\normtypdom{\tayh\brac{\jet\eta}}{H}{s}{\T^2}
		\lesssim
		C^{\brac{2}}_f
		+ C^{\brac{\ceil{s}+2}}_f
		\brac{
			\normtypdom{\eta}{H}{s+2}{\T^2}
			+ \normtypdom{\eta}{H}{s+2}{\T^2}^{\ceil{s}}
		}\text{ and}\\
		\normtypdom{\tayg\brac{\jet\eta}}{H}{s}{\T^2}
		\lesssim
		C^{\brac{3}}_f
		+ C^{\brac{\ceil{s}+3}}_f
		\brac{
			\normtypdom{\eta}{H}{s+2}{\T^2}
			+ \normtypdom{\eta}{H}{s+2}{\T^2}^{\ceil{s}}
		},
	\end{cases}
	\]
	where here we recall that the constants $C^{\brac{k}}_f$ are defined in Definition \ref{def:universalConstants}.
	
	We may now proceed with the estimates.  Since, as detailed above, most terms in the remainder are easy to control,
	we only highlight those which are more delicate and representative of the difficulties encountered.
	More precisely, we estimate in detail:
	\begin{enumerate}
		\item	the term involving $\fvr\brac{\eta} - \svr{0}\brac{\eta}$ in $\rem_E$,
		\item	the term involving $\fvr\brac{\eta} + g\eta$ in $\rem_D$, and
		\item	the term involving $\fvr\brac{\eta} - \svr{0}\brac{\eta}$ in $\rem_D$.
	\end{enumerate}
	These estimates are obtained as follows.
	\begin{enumerate}
		\item	We seek to control $
				\normtypdom{
					\brac{\fvr\brac{\eta} - \svr{0}\brac{\eta}}e_3
				}{H}{1/2}{\T^2}
			$:
		\begin{align*}
			\normtypdom{
				\brac{\fvr\brac{\eta} - \svr{0}\brac{\eta}}e_3
			}{H}{1/2}{\T^2}
			&=
				\normtypdom{\jet^* \Bigbrac{
					\tayg\brac{\jet\eta}\bullet\brac{\jet\eta \otimes \jet\eta}
			}}{H}{1/2}{\T^2}
			\lesssim
				\normtypdom{
					\tayg\brac{\jet\eta}\bullet\brac{\jet\eta \otimes \jet\eta}
				}{H}{5/2}{\T^2}\\
			&\lesssim
				\normtypdom{\tayg\brac{\jet\eta}}{H}{5/2}{\T^2}
				\normtypdom{\jet\eta}{H}{5/2}{\T^2}^2
			\lesssim
				1 \cdot \sqrt\enimp \sqrt\enimp
			= \enimp.
		\end{align*}
		\item	We seek to control $\normtypdom{\brac{\fvr\brac{\eta} + g\eta}\nabla\eta}{H}{3/2}{\T^2}$
			and thus the key term to control is $\normtypdom{\fvr\brac{\eta}}{H}{3/2}{\T^2}$.
			Since $\fvr\brac{\eta}$ is a differential operator of order 4,
			and since $\dsimp \gtrsim \normtypdom{\eta}{H}{11/2}{\T^2} \gtrsim \normtypdom{\nabla^4 \eta}{H}{3/2}{\T^2}$,
			we cannot get away with writing $\fvr\brac{\eta} = \jet^* \brac{\nabla f\brac{\jet\eta}}$
			and estimating $\normtypdom{\nabla f \brac{\jet\eta}}{H}{7/2}{\T^2}$.
			Instead, we use Lemma \ref{lemma:CompVarSurfEner} to obtain
			\begin{align*}
				\normtypdom{\fvr\brac{\eta}}{H}{3/2}{\T^2}
				&\leqslant
				\normtypdom{\nabla^2_{M,M} f\brac{\jet\eta} \bullet \nabla^4 \eta}{H}{3/2}{\T^2}
				+ \normtypdom{\nabla^2_{\p,\p} f\brac{\jet\eta} \bullet \nabla^2 \eta}{H}{3/2}{\T^2}\\
				&\hspace{-2cm}\quad+ \normtypdom{\nabla^3_{M,M,M} f\brac{\jet\eta} \bullet \brac{\nabla^3 \eta \otimes \nabla^3 \eta}}{H}{3/2}{\T^2}
				+ \normtypdom{\nabla^3_{M,M,\p} f\brac{\jet\eta} \bullet \brac{\nabla^3 \eta \otimes \nabla^2 \eta}}{H}{3/2}{\T^2}\\
				&\hspace{-2cm}\quad+ \normtypdom{\nabla^3_{\p,M,\p} f\brac{\jet\eta} \bullet \brac{\nabla^2 \eta \otimes \nabla^2 \eta}}{H}{3/2}{\T^2}\\
				&\hspace{-2cm}\lesssim
				\normtypdom{\nabla^2 f\brac{\jet\eta}}{H}{3/2}{\T^2}
				\normtypdom{\nabla^4 \eta}{H}{3/2}{\T^2}
				+ \normtypdom{\nabla^2 f\brac{\jet\eta}}{H}{3/2}{\T^2}
				\normtypdom{\nabla^2 \eta}{H}{3/2}{\T^2}\\
				&\hspace{-2cm}\quad+ \normtypdom{\nabla^3 f\brac{\jet\eta}}{H}{3/2}{\T^2}
				\normtypdom{\nabla^3 \eta}{H}{3/2}{\T^2}^2
				+ \normtypdom{\nabla^3 f\brac{\jet\eta}}{H}{3/2}{\T^2}
				\normtypdom{\nabla^3 \eta}{H}{3/2}{\T^2}
				\normtypdom{\nabla^2 \eta}{H}{3/2}{\T^2}\\
				&\hspace{-2cm}\quad+ \normtypdom{\nabla^3 f\brac{\jet\eta}}{H}{3/2}{\T^2}
				\normtypdom{\nabla^2 \eta}{H}{3/2}{\T^2}^2\\
				&\hspace{-2cm}\lesssim
				\normtypdom{\eta}{H}{11/2}{\T^2}
				+\normtypdom{\eta}{H}{7/2}{\T^2}
				+\normtypdom{\eta}{H}{9/2}{\T^2}^2
				+\normtypdom{\eta}{H}{9/2}{\T^2}
				\normtypdom{\eta}{H}{7/2}{\T^2}
				+\normtypdom{\eta}{H}{7/2}{\T^2}^2\\
				&\hspace{-2cm}\lesssim \sqrt\dsimp + \sqrt\enimp + 3\enimp
				\quad\lesssim \sqrt\dsimp,
			\end{align*}
			where we have used that for $k=2,3$,
			\begin{align*}
				\normtypdom{\nabla^k f\brac{\jet\eta}}{H}{3/2}{\T^2}
				&\leqslant
				\normtypdom{\nabla^k f\brac{\jet\eta}}{H}{2}{\T^2}
				\hspace{-1.6cm}&&\lesssim
				C^{\brac{k}}_f
				+ C^{\brac{k+3}}_f
				\brac{
					\normtypdom{\eta}{H}{4}{\T^2}
					+ \normtypdom{\eta}{H}{4}{\T^2}^2
				}\\
				\hspace{-1.6cm}&&&\lesssim 1 + \sqrt\enimp + \enimp
				\quad\lesssim 1
			\end{align*}
			for the constants $C^{\brac{k}}_f$ as defined in Definition \ref{def:universalConstants}.
			So finally:
			\begin{align*}
				\normtypdom{\brac{\fvr\brac{\eta}+g\eta}\nabla\eta}{H}{3/2}{\T^2}
				&\lesssim
				\normtypdom{\fvr\brac{\eta}+g\eta}{H}{3/2}{\T^2}
				\normtypdom{\nabla\eta}{H}{3/2}{\T^2}
				\lesssim
				\brac{\sqrt\dsimp + \normtypdom{\eta}{H}{3/2}\T^2}
				\normtypdom{\eta}{H}{5/2}{\T^2}\\
				&\lesssim
				\brac{\sqrt\dsimp + \sqrt\enimp}\sqrt\enimp
				\quad\lesssim
				\sqrt\enimp \sqrt\dsimp.
			\end{align*}
		\item	We seek to control $\normtypdom{\brac{\fvr\brac{\eta} - \brac{\svr{0}}\eta}}{H}{3/2}{\T^2}$.
			Observe that (using Lemma \ref{lemma:CompVarSurfEner} again)
			\begin{align*}
				\fvr\brac{\eta} - \brac{\svr{0}}\eta
				&= \brac{
					\nabla^2_{M,M} f \brac{\jet\eta} - \nabla^2_{M,M} f \brac{0}
				} \bullet \nabla^4 \eta
				- \brac{
					\nabla^2_{\p,\p} f\brac{\jet\eta} - \nabla^2_{\p,\p} f \brac{0}
				}
				\bullet \nabla^2 \eta\\
				&+ \nabla^3_{M,M,M} f\brac{\jet\eta} \bullet \brac{\nabla^3 \eta \otimes \nabla^3 \eta}
				+ \nabla^3_{M,M,\p} f\brac{\jet\eta} \bullet \brac{\nabla^3 \eta \otimes \nabla^2 \eta}\\
				&+ \nabla^3_{\p,M,\p} f\brac{\jet\eta} \bullet \brac{\nabla^2 \eta \otimes \nabla^2 \eta}.
			\end{align*}
			In particular, for
			\[
				F\brac{z} \defeq \int_0^1 \nabla\nabla^2_{M,M} f \brac{tz}dt\text{ and }
				G\brac{z} \defeq \int_0^1 \nabla\nabla^2_{p,p} f \brac{tz}dt,
			\]
			where $z = \brac{\p,M} \in \R^n \times \R^{n \times n}$,
			we have (by the Fundamental Theorem of Calculus)
			\begin{align*}
				\fvr\brac{\eta} - \brac{\svr{0}}\eta
				&= \brac{F\brac{\jet\eta} \bullet \jet\eta} \bullet \nabla^4 \eta
				+ \brac{G\brac{\jet\eta} \bullet \jet\eta} \bullet \nabla^2 \eta\\
				&+ \nabla^3_{M,M,M} f\brac{\jet\eta} \bullet \brac{\nabla^3 \eta \otimes \nabla^3 \eta}
				+ \nabla^3_{M,M,\p} f\brac{\jet\eta} \bullet \brac{\nabla^3 \eta \otimes \nabla^2 \eta}\\
				&+ \nabla^3_{\p,M,\p} f\brac{\jet\eta} \bullet \brac{\nabla^2 \eta \otimes \nabla^2 \eta}.
			\end{align*}
			Crucially, all terms have a part which is \emph{quadratic} in $\eta$.
			By an argument similar to that of Lemma \ref{lemma:auxFuncTaylorExp} we have, in the small energy regime, the estimates
			\[
				\normtypdom{F\brac{\jet\eta}}{H}{s}{\T^2} \lesssim 1\text{ and }
				\normtypdom{G\brac{\jet\eta}}{H}{s}{\T^2} \lesssim 1
			\]
			for any $s \in \sbrac{2,\frac{5}{2}}$.
			So finally, we can perform the estimate:
			\begin{align*}
				\normtypdom{\brac{\fvr\brac{\eta} - \brac{\svr{0}}\eta}}{H}{3/2}{\T^2}
				&\lesssim
				\normtypdom{F\brac{\jet\eta}}{H}{3/2}{\T^2}
				\normtypdom{\jet\eta}{H}{3/2}{\T^2}
				\normtypdom{\nabla^4 \eta}{H}{3/2}{\T^2}
				\\&\hspace{-3cm}
				+ \normtypdom{G\brac{\jet\eta}}{H}{3/2}{\T^2}
				\normtypdom{\jet\eta}{H}{3/2}{\T^2}
				\normtypdom{\nabla^2 \eta}{H}{3/2}{\T^2}
				+ \text{l.o.t.}
				\\&\hspace{-3cm}\lesssim
				\normtypdom{F\brac{\jet\eta}}{H}{2}{\T^2}
				\normtypdom{\eta}{H}{7/2}{\T^2}
				\normtypdom{\eta}{H}{11/2}{\T^2}
				+ \normtypdom{G\brac{\jet\eta}}{H}{2}{\T^2}
				\normtypdom{\eta}{H}{7/2}{\T^2}^2
				+ \text{l.o.t.}
				\\&\hspace{-3cm}\lesssim
				\sqrt\enimp \sqrt\dsimp + \enimp
				+ \text{l.o.t.}
				\quad\lesssim \sqrt\enimp \sqrt\dsimp
			\end{align*}
			where we have omitted the details for the lower order terms involving $\nabla^3 f$ (denoted l.o.t. above) since they follow exacty as in the second item above.
	\end{enumerate}		
\end{proof}

\subsection{Geometric corrections}
\label{sec:geoCorr}
In this section we compute the geometric corrections to the energy and dissipation
(i.e. the difference between their geometric and equilibrium versions) in Remark \ref{rmk:formGeometricCommutators},
and we estimate these corrections in Lemma \ref{lemma:smallnessEstimateGeometricCorrections}.
\begin{remark}
\label{rmk:formGeometricCommutators}
	The geometric corrections are
	\begin{equation*}
		\geocor_E \brac{\x} = \engeo\brac{\x;\x} - \eneq\brac{\x}\text{ and }
		\geocor_D \brac{\x} = \dsgeo\brac{\x;\x} - \dseq\brac{\x}
	\end{equation*}
	(c.f. equations \ref{defeq:engeo}, \ref{defeq:dsgeo}, \ref{defeq:eneq}, and \ref{defeq:dseq} for the definitions of the geometric and equilibrium versions of the energy and dissipation).
	For $\x = \brac{u, p, \eta}$ we can compute the geometric corrections to be	
	\begin{align*}
		\geocor_E \brac{\x}
		&=	\sum\limits_{
				\parabolicOrder{\alpha} \leq 2
			} \brac{
				\frac{1}{2} \int_\Omega \abs{\partial^\alpha u}^2 \brac{J-1}
				+ \frac{1}{2} \int_{\T^2} \brac{
					\int_0^1 g_\alpha \brac{t} \nabla^3 f \brac{t\jet\eta} \diff t
				} \bullet \brac{
					\jet\eta \otimes \jet\partial^\alpha \eta \otimes \jet\partial^\alpha \eta
				}
			}\\
	\intertext{and}
		\geocor_D \brac{\x}
		&= 	\sum\limits_{
				\parabolicOrder{\alpha} \leq 2
			} \brac{
				\frac{1}{2} \int_\Omega \abs{
					\symgrad^{\geo - I} \partial^\alpha u
				}^2 J
				- \int_\Omega\brac{
					\symgrad^{\geo - I} \partial^\alpha u
					:
					\symgrad \partial^\alpha u
				} J
				+ \frac{1}{2} \int_\Omega \abs{
					\symgrad \partial^\alpha u
				}^2 \brac{J-1}
			}
	\end{align*}
	(see Section \ref{sec:formGeoCorr} for the details of the computation of $\geocor_E$ and $\geocor_D$ and the definition of $g_\alpha$).
	All we need to know about $g_\alpha$ in order to estimate the geometric corrections is that $\abs{g_\alpha} \leqslant 1$ on $\sbrac{0,1}$.

	Note that $\nabla^3 f$ appears in the geometric corrections to the energy.
	This is as expected since $\geocor_E \brac{\x} = \engeo\brac{\x;\x} - \eneq{\x} \sim \engeo\brac{\x;\x} - \engeo\brac{\x,0}$
	where $\engeo$ depends on $\nabla^2 f$. Therefore, upon Taylor expanding about the equilibrium solution $\x=0$ we pick up a term involving $\nabla^3 f$.
\end{remark}

We now estimate the geometric corrections.

\begin{lemma}
\label{lemma:smallnessEstimateGeometricCorrections}
	In the small energy regime (see Definition \ref{def:smallEnergyRegime}) we have the estimates
	\begin{equation*}
		\abs{\geocor_E} \lesssim \sqrt\enimp \enimp
		\text{ and }
		\abs{\geocor_D} \lesssim \sqrt\enimp \dsimp,
	\end{equation*}
	where $\geocor_E$ and $\geocor_D$ are defined in Remark \ref{rmk:formGeometricCommutators}.
\end{lemma}
\begin{proof}
	First we check that the term involving $g_\alpha$ is small.
	Observe that since $\abs{g_\alpha\brac{t}}\leq 1$ when $t\in\sbrac{0,1}$, it follows that
	\[
		\sup_{\abs{z}\leqslant R} \vbrac{
			\int_0^1 g_\alpha \brac{t} \nabla^3 f \brac{tz} \diff t
		}
		\leqslant
		\norm{
			\nabla^3 f
		}{
			L^\infty \brac{
			\overline{
				B\brac{0,R}
			}}
		}
	\]
	and hence
	\[
		\norm{
			\int_0^1 g_\alpha \brac{t} \nabla^3 f \brac{t \jet\eta} \diff t
		}{\infty}
		\leq
		\norm{
			\nabla^3 f
		}{
			L^\infty \brac{
			\overline{
				B\brac{0,\norm{\jet\eta}{\infty}}
			}}
		}
		\leqslant C^{\brac{3}}_f,
	\]
	where the constant $C^{\brac{3}}_f$ is defined in Definition \ref{def:universalConstants}.
	In particular, in a small energy regime, $\norm{\jet\eta}{\infty} \lesssim \sqrt\enimp$, and hence
	\[
		\norm{
			\int_0^1 g_\alpha \brac{t} \nabla^3 f \brac{t \jet\eta} \diff t
		}{\infty}
		\lesssim C^{\brac{3}}_f
		\lesssim 1.
	\]
	Note that due to the fashion in which we perform the estimates, it is sufficient to handle the case $\partial^\alpha = \pdt, \nablatwo^2$.
	Recall that the control we have over the geometric coefficients $\geo$ and $J$ is recorded in Lemma \ref{lemma:estGeoCoeff}.
	
	We now estimate the corrections to the energy.
	\begin{itemize}
		\item[\framebox{$\pdt$}]
			The geometric correction is
			\[
				\frac{1}{2} \int_\Omega \abs{\pdt u}^2 \brac{J-1}
				+ \frac{1}{2} \int_{\T^2}
					\brac{ \int_0^1 \nabla^3 f \brac{t \jet\eta} \diff t}
					\bullet \brac{\jet\eta \otimes \jet\pdt\eta \otimes \jet\pdt\eta},
			\]
			and it can be estimated in the following way:
			\begin{align*}
				\abs{\dots}
				&\lesssim
					\normtypdom{\pdt u}{L}{2}{\Omega}^2
					\normtypdom{J-1}{L}{\infty}{\Omega}
					+ \normtypdom{\int_0^1 g_\alpha \brac{t} \nabla^3 f \brac{t \jet\eta} \diff t}{L}{\infty}{\T^2}
					\normtypdom{\jet\eta}{L}{\infty}{\T^2}
					\normtypdom{\jet\pdt\eta}{L}{2}{\T^2}^2\\
				&\lesssim
					\normtypdom{\pdt u}{L}{2}{\Omega}^2
					\normtypdom{J-1}{H}{3/2+}{\Omega}
					+ \normtypdom{\jet\eta}{H}{1+}{\T^2}
					\normtypdom{\jet\pdt\eta}{L}{2}{\T^2}^2\\
				&\lesssim
					\enimp \sqrt\enimp
					+ \sqrt\enimp \enimp
				\quad \lesssim \enimp^{3/2}.
			\end{align*}
		\item[\framebox{$\nablatwo^2$}]
			Note that the control of $\eta$ in the energy is similar to parabolic scaling, but with a little bit more spatial regularity.
			Consequently we handle this term as we did the previous one involving $\pdt$ and obtain (omitting the details)
			\[
				\vbrac{
					\frac{1}{2} \int_\Omega \abs{\nabla^2 u}^2 \brac{J-1}
					+ \frac{1}{2} \int_{\T^2}
						\brac{ \int_0^1 \nabla^3 f \brac{t \jet\eta} \diff t}
						\bullet \brac{\jet\eta \otimes \jet\nabla^2\eta \otimes \jet\nabla^2\eta}
				}
				\lesssim \enimp^{3/2}.
			\]
	\end{itemize}
	
	Next we estimate the dissipative corrections. 	Note that $\abs{\symgrad^M v} = \abs{2 \sym \brac{\nabla^M v}} \lesssim \abs{M}\abs{\nabla v}$.
	\begin{itemize}
		\item[\framebox{$\pdt$}]
			The geometric correction is
			\[
				\frac{1}{2} \int_\Omega \abs{
					\symgrad^{\geo - I} \pdt u
				}^2 J
				- \int_\Omega\brac{
					\symgrad^{\geo - I} \pdt u
					:
					\symgrad \pdt u
				} J
				+ \frac{1}{2} \int_\Omega \abs{
					\symgrad \pdt u
				}^2 \brac{J-1},
			\]
			and it can be estimated in the following way:
			\begin{align*}
				\abs{\dots}
				&\lesssim
					\normtypdom{\geo-I}{L}{\infty}{\Omega}^2
					\normtypdom{\nabla\pdt u}{L}{2}{\Omega}^2
					\normtypdom{J}{L}{\infty}{\Omega}
					+ \normtypdom{\geo-I}{L}{\infty}{\Omega}
					\normtypdom{\nabla\pdt u}{L}{2}{\Omega}^2
					\normtypdom{J}{L}{\infty}{\Omega}\\
				&\quad
					+ \normtypdom{\nabla\pdt u}{L}{2}{\Omega}^2
					\normtypdom{J-1}{L}{\infty}{\Omega}\\
				&\lesssim
					\normtypdom{\geo-I}{H}{3/2+}{\Omega}^2
					\normtypdom{\nabla\pdt u}{L}{2}{\Omega}^2
					\normtypdom{J}{H}{3/2+}{\Omega}
					+ \normtypdom{\geo-I}{H}{3/2+}{\Omega}
					\normtypdom{\nabla\pdt u}{L}{2}{\Omega}^2
					\normtypdom{J}{H}{3/2+}{\Omega}\\
				&\quad
					+ \normtypdom{\nabla\pdt u}{L}{2}{\Omega}^2
					\normtypdom{J-1}{H}{3/2+}{\Omega}\\
				&\lesssim
					\enimp \dsimp \brac{1 + \sqrt\enimp}
					+ \sqrt\enimp \dsimp \brac{1 + \sqrt\enimp}
					+ \dsimp \sqrt\enimp
				\quad \lesssim \sqrt\enimp \dsimp.
			\end{align*}
		\item[\framebox{$\nablatwo^2$}]
			Since the control we have on $u$ follows parabolic scaling precisely, upon replacing $\pdt u$ by $\nabla^2 u$ we can proceed in exactly the same way we did above.
			We therefore obtain that
			\begin{equation*}
				\vbrac{
					\frac{1}{2} \int_\Omega \abs{
						\symgrad^{\geo - I} \nabla^2 u
					}^2 J
					- \int_\Omega\brac{
						\symgrad^{\geo - I} \nabla^2 u
						:
						\symgrad \nabla^2 u
					} J
					+ \frac{1}{2} \int_\Omega \abs{
						\symgrad \nabla^2 u
					}^2 \brac{J-1}
				}
				\lesssim
				\sqrt\enimp \dsimp.
			\end{equation*}
	\end{itemize}
\end{proof}

\subsection{Synthesis}
\label{sec:synthesis}
In this section we piece together the various elements of the a~priori estimates into our main `a priori' theorem.
\begin{theorem}[A priori estimates]\label{thm:aPrioriEstimates}
	There exist $\delta, \lambda, C_{ap} > 0$ such that if there exists a solution $\x = \brac{u,p,\eta}$ on $\cobrac{0,T}$
	with initial condition $\x_0 = \brac{u_0, p_0, \eta_0}$ satisfying
	\[
		\sup\limits_{t\in\cobrac{0,T}} \enimp\brac{\x} \leqslant \delta
		\quad\text{and}\quad
		\int_0^T \dsimp\brac{\x} < \infty
	\]
	(and so in particular, for $\delta \leqslant 1$, we are in the small energy regime as defined in \ref{def:smallEnergyRegime}), then
	\[
		\sup_{t\in\cobrac{0,T}} \enimp\brac{\x} e^{\lambda t}
		+ \int_0^T \dsimp\brac{\x} e^{\lambda s} \;ds
		\leqslant C_{ap} \eneq\brac{\x_0}.
	\]
\end{theorem}
\begin{proof}
	In order to define $\delta$, $\lambda$, and $C_{ap}$, we must keep track of the constants in Lemmas
	\ref{lemma:smallnessEstimateCommutators}, \ref{lemma:smallnessEstimateNonlinearRemainder}, and \ref{lemma:smallnessEstimateGeometricCorrections}.
	In particular, we take $C_{C,E}, C_{C,D}, C_{N,E}, C_{N,D}, C_{G,E}, C_{G,D} > 0$ such that
	\begin{align*}
		\begin{cases}
			\abs{\mathcal{C}^1} \leqslant C_{C,E} \sqrt\enimp \dsimp,\\
			\abs{\mathcal{C}^2} \leqslant C_{C,D} \sqrt\enimp \enimp,
		\end{cases}
		\quad\begin{cases}
			\rem_E \leqslant C_{N,E,} \enimp^2,\\
			\rem_D \leqslant C_{N,D} \enimp \dsimp,
		\quad\end{cases}
		\begin{cases}
			\abs{\geocor_E} \leqslant C_{G,E,} \sqrt\enimp \enimp\text{, and}\\
			\abs{\geocor_D} \leqslant C_{G,D} \sqrt\enimp \dsimp.
		\end{cases}
	\end{align*}
	Moreover we assume without loss of generality that $C_{C,E}, C_{C,D}, C_{N,E}, C_{N,D}, C_{G,E}, C_{G,D} \geqslant 1$.
	Now pick
	\begin{align*}
		\delta =
		&
		\min
			\left(
				\delta_0,
				\frac{1}{2 C_{A,E} C_{N,E}},
				{\brac{\frac{1}{2 C_{G,E} C_{A,E}}}}^2,
				{\brac{\frac{1}{8 C_{C,E} C_{A,E}}}}^2,
			\right.
		\\&
		\hspace{1cm}
			\left.
				\frac{1}{2 C_{A,D} C_{N,D}},
				{\brac{\frac{1}{2 C_{G,D} C_{A,D}}}}^2,
				{\brac{\frac{1}{8 C_{C,D} C_{A,D}}}}^2
			\right),
	\end{align*}
	\[
		\frac{1}{2 \lambda}
		= 8 C_{A,D}
		\brac{1 + C_{C,E} \sqrt{\delta_0}}
		\brac{1 + 2 C_{G,E} C_{A,E} \sqrt{\delta_0}}\text{, and}
	\]
	\[
		C \defeq \max\brac{8 C_{A,E}, 16 C_{A,D}} 4 C_{A,E} \brac{1 + 2 C_{A,E} C_{G,E} \sqrt{\delta_0}}
		\brac{1 + \brac{1 + 2 C_{A,E}} C_{G,E} \sqrt{\delta_0}} > 0.
	\]
	We divide the remainder of the proof into several steps.
	
	\emph{Step 1}: We show that in the $\delta$-small energy regime,
			i.e. when $\sup \enimp \leqslant \delta$ and $\int\dsimp<\infty$,
			all versions of the energy, and all versions of the dissipation, are equivalent.
			The key observation is that the difference between various versions of the energy and the dissipation can be controlled
			(by Lemmas \ref{lemma:smallnessEstimateCommutators}, \ref{lemma:smallnessEstimateNonlinearRemainder}, and \ref{lemma:smallnessEstimateGeometricCorrections})
			by quantities of the form $\enimp^\alpha \enimp$ and $\enimp^\alpha \dsimp$ respectively, for some $\alpha > 0$.
			In particular, by picking $\delta$ small and imposing that $\enimp\leqslant\delta$ we may ensure that $\enimp^\alpha$ be small enough to perform absorption arguments.
			More precisely, we show that
				\[
					\enimp \asymp \eneq \asymp \engeo\text{ and }
					\dsimp \asymp \dseq \asymp \dsgeo,
				\]
				and in particular we show that
					\begin{subnumcases}{}
						\eneq \leqslant \enimp				\label{eq:eneqLeqslantEnimp},\\
						\enimp \leqslant C^E_{imp,eq} \eneq		\label{eq:enimpLeqslantEneq},\\
						\engeo \leqslant C^E_{geo,eq} \eneq		\label{eq:engeoLeqslantEneq},\\
						\enimp \leqslant C^E_{imp,geo} \engeo		\label{eq:enimpLeqslantEngeo},
					\end{subnumcases}
				and
					\begin{subnumcases}{}
						\dseq \leqslant \dsimp				\label{eq:dseqLeqslantDsimp},\\
						\dsimp \leqslant C^D_{imp,eq} \dseq		\label{eq:dsimpLeqslantDseq},\\
						\dsgeo \leqslant C^D_{geo,eq} \dseq		\label{eq:dsgeoLeqslantDseq},\\
						\dsimp \leqslant C^D_{imp,geo} \dsgeo,		\label{eq:dsimpLeqslantDsgeo}
					\end{subnumcases}
				where
				\begin{align*}
					\begin{cases}
						C^E_{imp,eq} = 2 C_{A,E},\\
						C^E_{geo,eq} = 1 + 2 C_{G,E} C_{A,E} \sqrt{\delta_0},\\
						C^E_{imp,geo} = 4 C_{A,E},
					\end{cases}
					\quad\begin{cases}
						C^D_{imp,eq} = 2 C_{A,D},\\
						C^D_{geo,eq} = 1 + 2 C_{G,D} C_{A,D} \sqrt{\delta_0}\text{, and}\\
						C^D_{imp,geo} = 4 C_{A,D}.
					\end{cases}
				\end{align*}
				To start, note that \eqref{eq:eneqLeqslantEnimp} follows immediately from the definition of $\eneq$ and $\enimp$.
				To obtain \eqref{eq:enimpLeqslantEneq}, we apply Proposition \ref{prop:GenericFormAuxiliaryEstimates}, Lemma \ref{lemma:smallnessEstimateNonlinearRemainder},
				and note that since $\enimp \leqslant \delta \leqslant \frac{1}{2 C_{A,E} C_{N,E}}$
				it follows that $\frac{C_{A,E}}{1 - C_{A,E} C_{N,E} \enimp} \leqslant 2 C_{A,E} = C^E_{imp,eq}$.  Thus
				\begin{align*}
					&\enimp
					\leqslant C_{A,E} \brac{\eneq + \rem_E}
					\leqslant C_{A,E} \brac{\eneq + C_{N,E} \enimp^2}\\
					\Rightarrow\quad
					&\enimp \leqslant \frac{C_{A,E}}{1 - C_{A,E} C_{N,E} \enimp} \eneq
					\leqslant C^E_{imp,eq} \eneq.
				\end{align*}
				To obtain \eqref{eq:engeoLeqslantEneq}, we use Remark \ref{rmk:formGeometricCommutators},
				Lemma \ref{lemma:smallnessEstimateGeometricCorrections}, and \eqref{eq:enimpLeqslantEneq} to see that
				\[
					\engeo
					= \eneq + \geocor_E
					\leqslant \eneq + C_{G,E} \sqrt\enimp \enimp
					\leqslant \brac{1 + 2 C_{G,E} C_{A,E} \sqrt{\delta_0}} \eneq
					= C^E_{geo,eq} \eneq.
				\]
				To obtain \eqref{eq:enimpLeqslantEngeo} we apply \eqref{eq:enimpLeqslantEneq}, Remark \ref{rmk:formGeometricCommutators},
				and Lemma \ref{lemma:smallnessEstimateGeometricCorrections} to see that
				\begin{align*}
					\enimp
					&\leqslant C^E_{imp,eq} \eneq
					= C^E_{imp,eq} \brac{\engeo - \geocor_E}
					\leqslant C^E_{imp,eq} \brac{\engeo + C_{G,E} \sqrt\enimp \enimp}\\
					\Rightarrow\quad
					\enimp
					&\leqslant \frac{C^E_{imp,eq}}{1 - C_{G,E} C^E_{imp,eq} \sqrt\enimp} \engeo
					= \frac{2 C_{A,E}}{1 - 2 C_{G,E} C_{A,E} \sqrt\enimp} \engeo
					\stackrel{\brac{\star}}{\leqslant} 4 C_{A,E} \engeo
					= C^E_{imp,geo} \engeo,
				\end{align*}
				where $\brac{\star}$ holds since $\enimp \leqslant \delta \leqslant {\brac{\frac{1}{2 C_{G,E} C_{A,E}}}^2}$.
				The bound \eqref{eq:dseqLeqslantDsimp} follows immediately from the definition of $\dseq$ and $\dsimp$.
				To obtain \eqref{eq:dsimpLeqslantDseq}, we apply Proposition \ref{prop:GenericFormAuxiliaryEstimates} and Lemma \ref{lemma:smallnessEstimateNonlinearRemainder}
				to see that
				\begin{align*}
					\dsimp
					&\leqslant C_{A,D} \brac{\dseq + \rem_D}
					\leqslant C_{A,D} \brac{\dseq + C_{N,D} \enimp\dsimp}\\
					\Rightarrow\quad
					\dsimp
					&\leqslant \frac{C_{A,D}}{1 - C_{A,D} C_{N,D} \enimp} \dseq
					\stackrel{\brac{\star}}{\leqslant} 2 C_{A,D} \dseq
					= C^D_{imp,eq} \dseq,
				\end{align*}
				where $\brac{\star}$ holds since $\enimp \leqslant \delta \leqslant \frac{1}{2 C_{A,D} C_{N,D}}$.
				To obtain \eqref{eq:dsgeoLeqslantDseq}, we use Remark \ref{rmk:formGeometricCommutators},
				Lemma \ref{lemma:smallnessEstimateGeometricCorrections}, and \eqref{eq:dsimpLeqslantDseq} to see that
				\[
					\dsgeo
					= \dseq + \geocor_D
					\leqslant \dseq + C_{G,D} \sqrt\enimp \dsimp
					\leqslant \brac{1 + 2 C_{G,D} C_{A,D} \sqrt{\delta_0}} \dseq
					= C^D_{geo,eq} \dseq.
				\]
				To obtain \eqref{eq:dsimpLeqslantDsgeo} we apply \eqref{eq:dsimpLeqslantDseq}, Remark \ref{rmk:formGeometricCommutators},
				and Lemma \ref{lemma:smallnessEstimateGeometricCorrections} to see that
				\begin{align*}
					\dsimp
					&\leqslant C^D_{imp,eq} \dseq
					= C^D_{imp,eq} \brac{\dsgeo - \geocor_D}
					\leqslant C^D_{imp,eq} \brac{\dsgeo + C_{G,D} \sqrt\enimp \dsimp}\\
					\Rightarrow\quad
					\dsimp
					&\leqslant \frac{C^D_{imp,eq}}{1 - C_{G,D} C^D_{imp,eq} \sqrt\enimp} \dsgeo
					= \frac{2 C_{A,D}}{1 - 2 C_{G,D} C_{A,D} \sqrt\enimp} \dsgeo
					\stackrel{\brac{\star}}{\leqslant} 4 C_{A,D} \dsgeo = C^D_{imp,geo} \dsgeo,
				\end{align*}
				where $\brac{\star}$ holds since $\enimp \leqslant \delta \leqslant {\brac{\frac{1}{2 C_{G,D} C_{A,D}}}}^2$.
				
	\emph{Step 2}:  We apply the generic energy-dissipation relations computed in Propositions
			\ref{prop:ZerothOrderEnergyDissipationEstimates} and \ref{prop:HigherOrderEnergyDissipationEstimates}
			to the case where $\y=\partial^\alpha \x$, and then sum over $\parabolicOrder{\alpha}\leqslant 2$ to obtain the energy-dissipation relation:
			\begin{align*}
				\Dt \engeo + \dsgeo = \mathcal{C}^1 + \Dt \mathcal{C}^2
				\quad\Leftrightarrow\quad
				\Dt \brac{\engeo - \mathcal{C}^2} + \brac{\dsgeo - \mathcal{C}^1} = 0.
			\end{align*}
			
	\emph{Step 3}:  Recall that $\dsimp \geqslant \enimp$, i.e. the dissipation is coercive over the energy.
			We now use Steps 1 and 2 with this coercivity, as well as Lemma \ref{lemma:smallnessEstimateCommutators}, to obtain a Gronwall-type inequality:
			\begin{align*}
				\dsgeo - \mathcal{C}^1
				&\geqslant \frac{1}{4 C_{A,D}} \dsimp - \mathcal{C}^1
				\geqslant \frac{1}{4 C_{A,D}} \dsimp - C_{C,D} \sqrt\enimp \dsimp
				&\hspace{-3.5cm}
				\text{by \eqref{eq:dsimpLeqslantDsgeo} and Lemma \ref{lemma:smallnessEstimateCommutators}}
				\\
				&= \brac{\frac{1}{4 C_{A,D}} - C_{C,D} \sqrt\enimp} \dsimp
				\geqslant \frac{1}{8C_{A,D}} \dsimp
				\geqslant \frac{1}{8C_{A,D}} \enimp
				&\hspace{-3.5cm}
				\text{by } \enimp \leqslant\delta\leqslant{\brac{\frac{1}{8C_{A,D}C_{C,D}}}}^2
				\text{ and coercivity}
				\\
				&\geqslant \frac{1}{8 C_{A,D} \brac{1 + C_{C,E} \sqrt{\delta_0}}} \brac{\enimp - \mathcal{C}^2}
				&\hspace{-3.5cm}
				\text{by Lemma \ref{lemma:smallnessEstimateCommutators} and since } \enimp \leqslant \delta \leqslant \delta_0
				\\
				&\geqslant \frac{1}{8 C_{A,D} \brac{1 + C_{C,E} \sqrt{\delta_0}}} \brac{\frac{1}{1 + 2 C_{G,E} C_{A,E} \sqrt{\delta_0}} \engeo - \mathcal{C}^2}
				&\hspace{-3.5cm}
				\text{by \eqref{eq:engeoLeqslantEneq} and \eqref{eq:eneqLeqslantEnimp}}
				\\
				&\geqslant \frac{1}{8 C_{A,D} \brac{1 + C_{C,E} \sqrt{\delta_0}} \brac{1 + 2 C_{G,E} C_{A,E} \sqrt{\delta_0}}} \brac{\engeo - \mathcal{C}^2}
				= 2 \lambda \brac{\engeo - \mathcal{C}^2},
			\end{align*}
			and therefore
			\[
				\Dt \brac{\engeo - \mathcal{C}^2} + \lambda\brac{\engeo - \mathcal{C}^2} + \frac{1}{16 C_{A,D}} \dsimp \leqslant 0.
			\]
			Upon integrating in time, we obtain that for all $t\in \cobrac{0,T}$,
			\[
				\brac{\engeo - \mathcal{C}^2} \brac{\x} e^{\lambda t}
				+ \int_0^t \frac{1}{16 C_{A,D}} \dsimp \brac{\x} e^{\lambda s} ds
				\leqslant \brac{\engeo - \mathcal{C}^2} \brac{\x_0}.
			\]
			Now observe that using \eqref{eq:enimpLeqslantEngeo}, Lemma \ref{lemma:smallnessEstimateCommutators}, and the fact that
			$\enimp \leqslant \delta \leqslant {\brac{\frac{1}{8 C_{A,E} C_{C,E}}}}^2$, we obtain that
			\[
				\engeo - \mathcal{C}^2
				\geqslant \frac{1}{4 C_{A,E}} \enimp - C_{C,E} \sqrt\enimp \enimp
				= \brac{\frac{1}{4 C_{A,E}} - C_{C,E} \sqrt\enimp} \enimp
				\geqslant \frac{1}{8 C_{A,E}} \enimp,
			\]
			whilst using \eqref{eq:engeoLeqslantEneq}, \eqref{eq:eneqLeqslantEnimp}, and Lemma \ref{lemma:smallnessEstimateCommutators}, we obtain that
			\[
				\engeo - \mathcal{C}^2
				\leqslant \brac{1 + 2 C_{G,E} C_{A,E} \sqrt{\delta_0}} \enimp + C_{C,E} \sqrt{\delta_0} \enimp
				= \brac{1 + \brac{1 + 2 C_{A,E}} C_{G,E} \sqrt{\delta_0}} \enimp.
			\]
			Therefore, for all $t\in\cobrac{0,T}$,
			\begin{align*}
				\frac{1}{8C_{A,E}} \enimp\brac{\x} e^{\lambda t}
				+ \int_0^t \frac{1}{16 C_{A,D}} \dsimp\brac{\x} e^{\lambda s} ds
				\leqslant \brac{1 + \brac{1 + 2 C_{A,E}} C_{G,E} \sqrt{\delta_0}} \enimp \brac{\x_0}\\
				\leqslant 4 C_{A,E} \brac{1 + 2 C_{A,E} C_{G,E} \sqrt{\delta_0}} \brac{1 + \brac{1 + 2 C_{A,E}} C_{G,E} \sqrt{\delta_0}} \eneq\brac{\x_0},
			\end{align*}
			so indeed we have that
			\[
				\sup_{t\in\cobrac{0,T}} \enimp\brac{\x} e^{\lambda t}
				+ \int_0^T \dsimp\brac{\x} e^{\lambda s} \;ds
				\leqslant C_{ap} \eneq\brac{\x_0}.
			\]

\end{proof}

\section{Global well-posedness and decay}
\label{sec:gwpDecay}

In this section we prove the main result of the paper, namely Theorem \ref{thm:gwpDecay}.  Before proving this global existence and decay result, we first consider the issue of local well-posedness.  

\subsection{Local well-posedness}

The local existence theory can be rigorously developed by modifying the techniques used to prove the a~priori estimates (see for instance \cite{cheng-coutand-shkoller,cheng_shkoller_2010,guo_tice_local,wu_lwp,wang_tice_kim,zheng_lwp,zheng_tice_lwp}), so for the sake of brevity we will only sketch what can be obtained in this manner.

In order to discuss the local well-posedness theory, we will need the following notation.
\begin{definition}[Norm measuring the size of the initial condition]\label{def:normIC}
We define the following.
	\begin{itemize}
		\item	For $\mathcal{Z} = \brac{u_0, \eta_0}$ we write
			\[
				\mathcal{I}\brac{\mathcal{Z}} \defeq
				\normtypdom{\eta_0}{H}{9/2}{\T^2}^2
				+ \normtypdom{u_0}{H}{2}{\Omega}^2
				+ \normtypdom{u_0\cdot\nu_{\partial\Omega_0}^{\geo}}{H}{2}{\Sigma}^2,
			\]
			where we recall from Section \ref{sec:notGeoDiffOp} that, on $\Sigma$,
			\[
				\nu_{\partial\Omega_0} \vert_\Sigma = \frac{
					\brac{-\nablatwo\eta_0, 1}
				}{
					\sqrt{1 + \abs{\nablatwo\eta}^2}
				}
				\quad\text{and}\quad
				\nu_{\partial\Omega_0}^{\geo} \vert_\Sigma
					= \sqrt{1 + \abs{\nablatwo\eta}^2} \nu_{\partial\Omega_0} \vert_\Sigma
					= \brac{-\nablatwo\eta_0, 1}.
			\]
		\item	For $\x = \brac{u,p,\eta}$ we abuse notations slightly and also write $\mathcal{I}\brac{\x}\defeq\mathcal{I}\brac{u,\eta}$.
	\end{itemize}
\end{definition}

It is most natural to specify the initial data $u_0$ and $\eta_0$, but in our analysis we also need $\mathcal{E}(0)$, which means we must construct $\partial_t u\vert_{t=0}$, $p\vert_{t=0}$, and $\partial_t \eta \vert_{t=0}$.   We sketch how this construction proceeds in the following remark.

\begin{remark}
	\label{rmk:constructIC}
	In this remark we sketch how to construct $p_0$, $\pdt u_0$ and $\pdt\eta_0$ from $u_0$ and $\eta_0$.  Recall that the PDE is \eqref{NS_fixed_s}--\eqref{NS_fixed_e}.
       
       \emph{Constructing $p_0$:} Taking the $\geo$-divergence of \eqref{NS_fixed_s} and using \eqref{NS_fixed_div} yields
			\[
				-\Delta^\geo p = \nabla^\geo u : {\brac{\nabla^\geo u}}^T.
			\]
			Dotting \eqref{NS_fixed_dyn} with $\nugeo$ and dividing by $\brac{1 + \abs{\nablatwo\eta}^2}$ then yields
			\[
				p = {\brac{\symgrad^\geo u}}_{33} + \fvr\brac{\eta} + g.
			\]
			Finally, taking the trace of \eqref{NS_fixed_s}$\cdot e_3$ onto $\Sigma_b$ yields
			\[
				\partial_3^\geo p = \Delta^\geo u_3.
			\]
			So $p$ solves
			\[
			\left\{
			\begin{aligned}
				&-\Delta^\geo p = \nabla^\geo u : {\brac{\nabla^\geo u}}^T	&&\text{in }\Omega,\\
				&p = {\brac{\symgrad^\geo u}}_{33} + \fvr\brac{\eta} + g	&&\text{on }\Sigma\text{, and}\\
				&\partial_3^\geo p = \Delta^\geo u_3				&&\text{on }\Sigma_b.
			\end{aligned}
			\right.
			\]
			In particular, in the small energy regime where $\geo \sim I$,
			standard elliptic estimates coupled with product estimates in Sobolev spaces
			(to handle the nonlinear but small remainders) allows us to recover $p_0$ from $u_0$ and $\eta_0$ using this PDE.
			
	\emph{Constructing $\pdt u_0$ and $\pdt\eta_0$}:  We use \eqref{NS_fixed_s} and \eqref{NS_fixed_kin} to define
			\[
			\left\{
			\begin{aligned}
				&\pdt u_0 \defeq -\brac{u_0 \cdot \nabla^\geo} u_0 - \nabla^\geo p_0 + \Delta^\geo u_0\text{ and}\\
				&\pdt\eta_0 \defeq u_0 \cdot \nu_{\partial\Omega_0}^{\geo}.
			\end{aligned}
			\right.
			\]
\end{remark}

Following the procedure outlined in Remark \ref{rmk:constructIC} leads to the following result, which not only constructs the data, but provides an estimate in the small energy regime.

\begin{prop}[Constructing the initial conditions]\label{prop:constructingIC}
	There exist $\beta, C_{IC} > 0$ such that for every $T > 0$ for which $\x = \brac{u,p,\eta}$ is a solution on $\sbrac{0,T}$,
	if $\mathcal{I}\brac{\x\brac{0}} \leqslant \beta$ then $\eneq\brac{\x\brac{0}} \leqslant C_{IC} \mathcal{I}\brac{\x\brac{0}}$.
\end{prop}

Next we define the notion of admissible data.
\begin{definition}[Admissible initial condition]\label{def:admissibleIC}
	We say that $\brac{u_0, \eta_0} \in H^2\brac{\Omega;\R^3} \times H^{9/2}\brac{\T^2;\R}$ is an \emph{admissible initial condition} if it satisfies
	\begin{itemize}
		\item	$\nabla\cdot u_0 = 0$,
		\item	$\tr_{\Sigma_b} u_0 = 0$,
		\item	$\brac{I - \nu_{\partial\Omega_0}\otimes\nu_{\partial\Omega_0}}  {\brac{\tr_\Sigma \symgrad u_0 \cdot \nu_{\partial\Omega_0}}} = 0$,
		\item	$\tr_\Sigma u_0 \cdot \nu_{\partial\Omega_0}^{\geo} \in H^2 \brac{\Sigma}$,
		\item	$\int_{\T^2} \eta_0 = 0$, and
		\item	$\mathcal{I}\brac{u,\eta} \leqslant \beta$ for $\beta$ as in Proposition \ref{prop:constructingIC}.
	\end{itemize}
\end{definition}

A few remarks are in order.

\begin{remark}	\label{rmk:defAdmIC}
$\text{}$
	\begin{enumerate}
		\item	The first three items are nothing more than incompressibility and parts of the boundary conditions.
		\item	The fifth condition, namely requiring that $\int_{\T^2}\eta = 0$, is related to \eqref{intro_zero_avg}.
		\item	The fourth condition, namely requiring that $\tr u_0\cdot \nu_{\partial\Omega_0}^{\geo}$ be in $H^2$,
			is a \emph{compatibility condition}. Indeed, knowing that $u$ belong to $H^2$ and $\eta$ belongs to $H^{9/2}$
			only allows us to conclude that $\tr u_0\cdot \nu_{\partial\Omega_0} = \tr u_0 \cdot \brac{-\nablatwo\eta_0, 1}$ belongs to $H^{3/2}$.
			This gap in regularity means the procedure sketched in Remark \ref{rmk:constructIC} cannot close without assuming that this additional compatibility condition holds a priori.  Note that we will prove that this condition persists in time, so there is no trouble iteratively applying the local theory.
		\item	The sixth condition is there to ensure that the nonlinear PDEs used in the sketch from Remark \ref{rmk:constructIC}
			are sufficiently close to their linear counterpart (corresponding to $\eta = 0$ and $\geo = I$) such that the appropriate estimates can be made
			to produce the result from Proposition \ref{prop:constructingIC}.
	\end{enumerate}
\end{remark}

We now state the local existence result.

\begin{theorem}[Local well-posedness]\label{thm:lwp}
	There exist $T, \kappa_0, C_{lwp} > 0$ such that for every $\kappa\in\ocbrac{0,\kappa_0}$,
	if $\brac{u_0, \eta_0} \in H^2\brac{\Omega;\R^3} \times H^{9/2}\brac{\T^2;\R}$ is an admissible initial condition
	(c.f. Definition \ref{def:admissibleIC}) satisfying
	\[
		\mathcal{I}\brac{u_0, \eta_0} \leqslant \kappa
	\]
	(c.f. Definition \ref{def:normIC} for the definition of $\mathcal{I}$),
	then there exists a unique solution $\x = \brac{u, p, \eta}$ of \eqref{NS_fixed_s}--\eqref{NS_fixed_e} on $\sbrac{0,T}$ that satisfies
	\[
		\sup_{0 \leqslant t \leqslant T} \enimp\brac{\x\brac{t}}
		+ \int_0^T \dsimp\brac{\x\brac{t}} dt
		+ \norm{\pdt^2 u}{\mathcal{V}_T^*}^2
		\leqslant C_{lwp} \kappa,
	\]
	where
	\[
		\mathcal{V}_T \defeq \setdef{
			u \in L^2 \brac{
				\sbrac{0,T};H^1\brac{\Omega}
			}
		}{
			\tr_{\Sigma_b} u\brac{t} = 0
			\text{ and }
			\nabla^{\geo\brac{t}}\cdot u\brac{t} =0
			\text{ for a.e. }
			t\in\sbrac{0,T}
		}.
	\]
\end{theorem}
\begin{remark}
The local existence theorem is sufficient to justify our a priori estimates.
\end{remark}

Note that in light of Remark \ref{rmk:defAdmIC} (and item (2) therein, in particular) the admissibility of initial conditions is propagated by the flow (provided the solution remains small enough).

\begin{prop}[Propagation of admissibility for initial conditions]\label{prop:propagationAdmissibilityIC}
	Suppose that $\brac{u, p, \eta}$ is a solution on $\sbrac{0,T}$ such that $\brac{u_0, \eta_0}$ is an admissible initial condition.
	For every $t\in\sbrac{0,T}$, if $\mathcal{I}\brac{u_t, \eta_t} \leqslant \beta$, then $\brac{u_t, \eta_t}$ is an admissible initial condition (c.f. Definition \ref{def:admissibleIC}).
\end{prop}

\subsection{Proof of the main result}
Before stating and proving the main result, i.e. the global well-posedness and decay result,
we state and prove two preliminary lemmas.
The first lemma, Lemma \ref{lemma:eventualGWP}, is an eventual global well-posedness result that  shows that if small solutions exist past a critical time, then they exist globally in time.
The second lemma, Lemma \ref{lemma:arbFiniteTimeWP}, is a result about the existence of solution on arbitrarily large finite time intervals, provided the initial data is small enough.
Combining these two lemmas will then allow us to prove global well-posedness in Theorem \ref{thm:gwpDecay}.

We now prove our first lemma.  It says that past a critical time $T_{crit}$, the exponential decay from the a~priori estimates is sufficiently strong to ensure that we remain in a regime where the energy is small enough for the local well-posedness to hold at every time thereafter. This is \emph{eventual} well-posedness since it tells us that there exists a critical time past which the solution is globally well-defined.

\begin{lemma}[Eventual global well-posedness]\label{lemma:eventualGWP}
	Let $\delta$, $\lambda$, and $C_{ap}$ be as in Theorem \ref{thm:aPrioriEstimates}. Let $\kappa_0$, $C_{lwp}$, and $T$ be as in Theorem \ref{thm:lwp} and assume without loss of generality that $C_{lwp} \geqslant 1$.  Let $C_{IC}$ be as in Proposition \ref{prop:constructingIC},
	and let $T_{crit} > 0$ be such that $e^{\lambda \brac{T_{crit}-\frac{T}{2}}} \geqslant C_{ap} C_{IC}$.

	If $\x = \brac{u,p,\eta}$ is a solution on $\sbrac{0,\tau}$ for some $\tau \geqslant T_{crit}$,
	with $\brac{u_0, \eta_0}$ an admissible initial condition that also satisfies the smallness conditions
	\begin{equation}\label{eq:eventualGWP_1}
	\begin{cases}
		\mathcal{I}\brac{\x_0} \leqslant \min\cbrac{\kappa_0, \frac{\delta}{C_{lwp}}}
		\\
		\sup\limits_{0\leqslant t \leqslant \tau} \enimp\brac{\x} \leqslant \delta
		\quad\text{and}\quad
		\displaystyle\int_0^\tau \dsimp\brac{\x} < \infty,
	\end{cases}
	\end{equation}
	then the solution can be uniquely extended to a solution on $\cobrac{0,\infty}$ satisfying
	\[
		\sup_{t \geqslant 0} \enimp\brac{\x\brac{t}} \leqslant \delta
		\quad\text{and}\quad
		\int_0^{\infty} \dsimp\brac{\x\brac{t}} dt < \infty.
	\]
\end{lemma}

\begin{proof}
	Let $\tau\geqslant T_{crit}$ and let $\x = \brac{u,p,\eta}$ be a solution on $\sbrac{0,\tau}$
	starting from admissible data $(u_0,\eta_0)$ and satisfying \eqref{eq:eventualGWP_1}.  Define $T_{max}>0$ to be
	\[
		T_{max} \defeq \sup\setdef{
			T \geqslant 0
		}{
			\text{solution $\x$ exists on } \sbrac{0,T}
			\text{ and satisfies }
			\sup_{0\leqslant t \leqslant T} \enimp\brac{\x} \leqslant \delta
			\text{ and } \int_0^T \dsimp\brac{\x} < \infty
		}.
	\]

	First note that $T_{max} \geqslant \tau \geqslant T_{crit}$.
	Now suppose, by way of contradiction, that $T_{max} < \infty$.
	Let $\tilde{T} \defeq T_{max} -\frac{T}{2} > 0$.
	By Theorem \ref{thm:aPrioriEstimates}, Proposition \ref{prop:constructingIC},
	and the definition of $T_{crit}$, which is smaller than $T_{max}$, we have
	\[
		\enimp\brac{\x\brac{\tilde{T}}}
		\leqslant
		C_{ap} e^{-\lambda\brac{T_{max}-\frac{T}{2}}} \eneq\brac{\x\brac{0}}
		\leqslant
		C_{ap} C_{IC} e^{-\lambda\brac{T_{crit}-\frac{T}{2}}} \mathcal{I}\brac{\x\brac{0}}
		\leqslant
		\mathcal{I}\brac{\x\brac{0}}.
	\]
	Therefore, since $\mathcal{I}\brac{\x\brac{0}} \leqslant \min\cbrac{\kappa_0, \frac{\delta}{C_{lwp}}}$, we may employ Proposition \ref{prop:propagationAdmissibilityIC} and Theorem \ref{thm:lwp}
	to obtain a unique extension of the solution on $\sbrac{0, T_{max} + \frac{T}{2}}$ satisfying
	\[
		\sup_{0 \leqslant t \leqslant T_{max} + \frac{T}{2}} \enimp\brac{\x}
		+ \int_{0}^{T_{max}+\frac{T}{2}} \dsimp\brac{\x}
		\leqslant C_{lwp} \frac{\delta}{C_{lwp}} \leqslant \delta.
	\]
	We can thus use Theorem \ref{thm:aPrioriEstimates}, Proposition \ref{prop:constructingIC}, and the definition of $T_{crit}$ once more,
	this time on $\sbrac{0, T_{max} + \frac{T}{2}}$, to obtain that
	\[
		\enimp\brac{\x\brac{T_{max}+\frac{T}{2}}}
		\leqslant
		C_{ap} e^{-\lambda\brac{T_{max}+\frac{T}{2}}} \eneq\brac{\x\brac{0}}
		\leqslant
		C_{ap} C_{IC} e^{-\lambda T_{crit}} \mathcal{I}\brac{\x\brac{0}}
		\leqslant
		\mathcal{I}\brac{\x\brac{0}}
		\leqslant
		\delta
	\]
	which contradicts the definition of $T_{max}$. So indeed $T_{max} = \infty$.
\end{proof}

We now prove our second key lemma.

\begin{lemma}[Arbitrary finite-time well-posedness]\label{lemma:arbFiniteTimeWP}
	For every $\tau > 0$ there exists $\gamma > 0$ such that if $\brac{u_0, \eta_0}$ is an admissible initial condition
	with
	\[
		\mathcal{I}\brac{u_0,\eta_0} \leqslant \gamma,
	\]
	then there exists a unique solution $\x = \brac{u,p,\eta}$ on $\sbrac{0,\tau}$
	satisfying
	\[
		\sup_{0\leqslant t \leqslant \tau} \enimp\brac{\x\brac{t}}
		+ \int_0^{\tau} \dsimp\brac{\x\brac{t}} dt
		\leqslant \delta
	\]
	for $\delta$ as in Theorem \ref{thm:aPrioriEstimates}.
\end{lemma}
\begin{proof}
	Let $\tau > 0$, let $T$ be as in Theorem \ref{thm:lwp}, and pick $N\in\N$ such that $NT \geqslant \tau$. Let $C_{lwp}$ be as in Theorem \ref{thm:lwp}, and note that without loss of generality we may assume that $C_{lwp} > 1$.  Let $\beta$ be as in Proposition \ref{prop:constructingIC} and let $\gamma \defeq \frac{\beta}{C_{lwp}^N} > 0$.
	
	Let $\brac{u_0, \eta_0}$ be an admissible initial condition satisfying $\mathcal{I}\brac{u_0, \eta_0} \leqslant \gamma$.
	Then we apply the local well-posedness result, i.e. Theorem \ref{thm:lwp}, $N$ times
	(using Proposition \ref{prop:propagationAdmissibilityIC} to ensure the `initial conditions' are admissible at every step).
	More precisely, at step 1 we use Theorem \ref{thm:lwp} to obtain a unique solution $\x = \brac{u,p,\eta}$ on $\sbrac{0,T}$ satisfying
	\[
		\sup_{0\leqslant t \leqslant T} \enimp\brac{\x}
		+ \int_0^T \dsimp\brac{\x}
		\leqslant C_{lwp} \gamma.
	\]
	Since $\gamma\leqslant\frac{\beta}{C^N_{lwp}} \leqslant \frac{\beta}{C_{lwp}}$, it follows that $\brac{u_T, \eta_T}$ is an admissible initial condition.
	Then, at step $n$ for $n=2,\dots,N$, suppose that we have solution on $\sbrac{0,\brac{n-1}T}$ satisfying
	\[
		\sup_{0\leqslant t \leqslant \brac{n-1}T} \enimp\brac{\x}
		+ \int_0^{\brac{n-1}T} \dsimp\brac{\x}
		\leqslant C_{lwp}^{n-1} \gamma
	\]
	such that $\brac{u_{\brac{n-1}T}, \eta_{\brac{n-1}T}}$ is an admissible initial condition.
	We may then apply Theorem \ref{thm:lwp} to extend the solution uniquely to $\sbrac{0,nT}$
	such that it satisfies
	\[
		\sup_{0\leqslant t \leqslant nT} \enimp\brac{\x}
		+ \int_0^{nT} \dsimp\brac{\x}
		\leqslant C_{lwp} \brac{C_{lwp}^{n-1} \gamma} = C_{lwp}^n \gamma.
	\]
	In particular, since $\gamma \leqslant \frac{\beta}{C^N_{lwp}} \leqslant \frac{\beta}{C^n_{lwp}}$, it follows from
	Proposition \ref{prop:propagationAdmissibilityIC} that $\brac{u_{nT}, \eta_{nT}}$ is also an admissible initial condition.
	Finally, after step $N$, we have a solution on $\sbrac{0,NT} \supseteq \sbrac{0,\tau}$ satisfying
	\[
		\sup_{0\leqslant t \leqslant NT} \enimp\brac{\x}
		+ \int_0^{NT} \dsimp\brac{\x}
		\leqslant C_{lwp}^N \gamma
		\leqslant\delta.
	\]
\end{proof}

With the key lemmas in hand, we can now prove our main result.

\begin{theorem}[Global well-posedness and decay]\label{thm:gwpDecay}
	There exists $\epsilon > 0$ such that for every admissible initial condition $\brac{u_0, \eta_0}$ satisfying
	\[
		\mathcal{I}\brac{u_0, \eta_0} \leqslant \epsilon
	\]
	there exists a unique solution $\x = \brac{u,p,\eta}$ on $\cobrac{0,\infty}$ such that
	\[
		\sup_{t \geqslant 0} \enimp\brac{\x\brac{t}} e^{\lambda t}
		+ \int_0^\infty \dsimp\brac{\x\brac{t}} e^{\lambda t} dt
		\leqslant C \eneq\brac{\x\brac{0}},
	\]
	where $C=C_{ap} > 0$ and $\lambda > 0$ are as in Theorem \ref{thm:aPrioriEstimates}.
	Recall that admissible initial conditions are defined in Definition \ref{def:admissibleIC}.
\end{theorem}
\begin{proof}
	Let $\delta$ be as in the a~priori estimates (i.e. Theorem \ref{thm:aPrioriEstimates}),
	let $\kappa_0$ and $C_{lwp}$ be as in the local well-posedness result (i.e. Theorem \ref{thm:lwp}),
	let $T_{crit}$ be as in the eventual global well-posedness result (i.e. Lemma \ref{lemma:eventualGWP}),
	and let $\gamma = \gamma\brac{T_{crit}}$ be as in the arbitrary finite time existence result (i.e. Lemma \ref{lemma:arbFiniteTimeWP}).
	Pick $\epsilon = \min\brac{\gamma, \kappa_0, \frac{\delta}{C_{lwp}}}$.
	Now let $\brac{u_0, \eta_0}$ be an admissible initial condition satisfying $\mathcal{I}\brac{u_0,\eta_0} \leqslant \epsilon$.
	By the arbitrary finite time existence result (i.e. Lemma \ref{lemma:arbFiniteTimeWP}) and the choice of $\epsilon$,
	there exists a unique solution $\x = \brac{u,p,\eta}$ on $\sbrac{0, T_{crit}}$ satisfying
	\[
		\sup_{0 \leqslant t \leqslant T_{crit}} \enimp\brac{\x\brac{t}}
		+ \int_0^{T_{crit}} \dsimp\brac{\x\brac{t}}
		\leqslant \delta
	\]
	and therefore by the eventual global well-posedness result (i.e. Lemma \ref{lemma:eventualGWP}) and the choice of $\epsilon$
	there exists a unique extension of this solution to $\cobrac{0,\infty}$ satisfying
	\[
		\sup_{t \geqslant 0} \enimp\brac{\x\brac{t}} \leqslant \delta
		\quad\text{and}\quad
		\int_0^{\infty} \dsimp\brac{\x\brac{t}} < \infty.
	\]
	Finally we establish the exponential decay of the energy of this unique global solution.
	The a~priori estimates (i.e. Theorem \ref{thm:aPrioriEstimates}) tell us that for every $T > 0$
	\[
		\sup_{0\leqslant t \leqslant T} \enimp\brac{\x\brac{t}} e^{\lambda t}
		+ \int_0^T \dsimp\brac{\x\brac{t}} e^{\lambda t} dt
		\leqslant C \eneq\brac{\x\brac{0}}
	\]
	and so indeed, taking the supremum over $T>0$ yields the global decay estimate
	\[
		\sup_{t \geqslant 0} \enimp\brac{\x\brac{t}} e^{\lambda t}
		+ \int_0^\infty \dsimp\brac{\x\brac{t}} e^{\lambda t} dt
		\leqslant C \eneq\brac{\x\brac{0}}.
	\]
\end{proof}

\appendix
\addtocontents{toc}{\protect\setcounter{tocdepth}{1}}

\section[Test]{Intermediate results}
In this first part of the appendix we record various intermediate results of particular interest to the problem discussed in this paper.
We record computations and estimates for the geometric coefficients $\Phi$, $\geo$, $J$ and $\nugeo$,
as well as computations and estimates for the variations of the surface energy.
We also record details of the computations of various commutators.
\subsection{Geometric coefficients and differential operators}
\label{sec:geoCoeffAndDiffOp}
	In this section we record estimates for the geometric coefficients $\Phi$, $\geo$, $J$ and $\nugeo$ (as defined in Section \ref{sec:dom_coeff}) in Lemma \ref{lemma:estGeoCoeff},
	and we record the $\geo$-divergence and $\geo$-transport theorems in Propositions \ref{prop:geoDivThm} and \ref{prop:geoTransportThm} respectively.

	\begin{lemma}[Estimates for the geometric coefficients]	\label{lemma:estGeoCoeff}
	Recall the notational conventions of Section \ref{sec:dom_coeff}.
	Suppose that we are in the small energy regime (see Definition \ref{def:smallEnergyRegime}). 
	On the upper surface we have the bounds 
			\begin{equation*}
			 \normtypdom{\tr\brac{\geo - I}}{H}{7/2}{\T^2} 
			 + \normtypdom{\nugeo - e_3}{H}{7/2}{\T^2} \lesssim \sqrt\enimp.
			\end{equation*}
    In the bulk we have the bounds 
    \begin{equation*}
		\normtypdom{\Phi - \text{id}}{H}{5}{\Omega} 
		+ \normtypdom{\pdt\Phi}{H}{5/2}{\Omega} 
		+ \normtypdom{J-1}{H}{4}{\Omega} 
		+ \normtypdom{\pdt J}{H}{3/2}{\Omega} 
		+ \normtypdom{\geo - I}{H}{4}{\Omega}  
		+ \normtypdom{\pdt\geo}{H}{3/2}{\Omega} \lesssim \sqrt\enimp 
    \end{equation*}
    and 
    \begin{equation*}
		\normtypdom{\Phi - \text{id}}{H}{6}{\Omega}
		+ \normtypdom{\pdt\Phi}{H}{3}{\Omega}
        + \normtypdom{\pdt^2 \Phi}{H}{1}{\Omega}
		+ \normtypdom{J-1}{H}{5}{\Omega} 
		+ \normtypdom{\pdt J}{H}{2}{\Omega}
		+ \normtypdom{\pdt\geo}{H}{2}{\Omega}
		+ \normtypdom{\pdt^2 \geo}{H}{0}{\Omega} \lesssim \sqrt\dsimp.
    \end{equation*}
	\end{lemma}
	\begin{proof}
		We begin with estimating $\Phi = \id + \rchi \ext\eta\, e_3$ and its time derivatives: $\pdt\Phi = \rchi \ext\pdt\eta\, e_3$ and $\pdt^2 \Phi = \rchi \ext\pdt^2 \eta\, e_3.$	We estimate $\Phi - \id$ using Proposition \ref{prop:prodEstSobCtsMult}, Corollary \ref{cor:boundHarmExt},
		and the definitions of $\enimp$ and $\dsimp$ (c.f. equations (\ref{defeq:enimp}) and (\ref{defeq:dsimp}), respectively):
		\[
			\normtypdom{\Phi-\id}{H}{5}{\Omega}
			=
			\normtypdom{\rchi\ext\eta\,e_3}{H}{5}{\Omega}
			\lesssim
			\normtypdom{\rchi}{H}{13/2+}{\Omega}
			\normtypdom{\ext\eta}{H}{5}{\Omega}
			\lesssim
			\normtypdom{\eta}{H}{9/2}{\T^2}
			\leqslant \sqrt\enimp
		\]
		and
		\[
			\normtypdom{\Phi-\id}{H}{6}{\Omega}
			\lesssim
			\normtypdom{\rchi}{H}{15/2+}{\Omega}
			\normtypdom{\eta}{H}{11/2}{\T^2}
			\lesssim \sqrt\dsimp.
		\]
		We proceed similarly to estimate the time derivatives of $\Phi$:
		\[
		\begin{cases}
			\normtypdom{\pdt\Phi}{H}{5/2}{\Omega}
			\lesssim
			\normtypdom{\pdt\eta}{H}{2}{\T^2}
			\leqslant \sqrt\enimp,
			\\
			\normtypdom{\pdt\Phi}{H}{3}{\Omega}
			\lesssim
			\normtypdom{\pdt\eta}{H}{5/2}{\T^2}
			\leqslant \sqrt\dsimp\text{, and}
			\\
			\normtypdom{\pdt^2 \Phi}{H}{1}{\Omega}
			\lesssim
			\normtypdom{\pdt^2 \eta}{H}{1/2}{\T^2}
			\leqslant \sqrt\dsimp.
		\end{cases}
		\]
		Now we compute $J$, noting first that $\nabla\Phi = I + e_3 \otimes \nabla \brac{\rchi\ext\eta}.$  Therefore, by Lemma \ref{lemma:linAlgDetRank1}, by definition of $\rchi$ (c.f. Section \ref{sec:reformulation}), and by Lemma \ref{lemma:harmExtIdentities}
		\[
			J = \det\nabla\Phi
			= 1 + \partial_3 \brac{\rchi\ext\eta}
			= 1 + \frac{\ext\eta}{b} + \rchi\ext\sqrt{-\Delta}\eta
\text{ and }
            \pdt J = \partial_3 \brac{\rchi \ext\pdt\eta}.
		\]
		We may now estimate $J$ and its time derivatives
		\[
			\normtypdom{J-1}{H}{4}{\Omega}
			= \normtypdom{\partial_3\brac{\rchi\ext\eta}}{H}{4}{\Omega}
			\leqslant \normtypdom{\rchi\ext\eta}{H}{5}{\Omega}
			\lesssim \normtypdom{\rchi}{H}{13/2+}{\Omega} \normtypdom{\ext\eta}{H}{5}{\Omega}
			\lesssim \normtypdom{\eta}{H}{9/2}{\T^2}
			\leqslant \sqrt\enimp
		\]
		and similarly
		\[
		\begin{cases}
			\normtypdom{J-1}{H}{5}{\Omega} \lesssim \normtypdom{\eta}{H}{11/2}{\T^2} \leqslant \sqrt\dsimp,\\
			\normtypdom{\pdt J}{H}{3/2}{\Omega} \lesssim \normtypdom{\pdt\eta}{H}{2}{\T^2} \leqslant \sqrt\enimp\text{, and}\\
			\normtypdom{\pdt J}{H}{2}{\Omega} \lesssim \normtypdom{\pdt\eta}{H}{5/2}{\T^2} \leqslant \sqrt\dsimp.
		\end{cases}
		\]
		Now we compute $\geo$. Recall that $\geo \defeq {\brac{\nabla\Phi}}^{-T}$ with $\nabla\Phi = I + e_3 \otimes \nabla\brac{\rchi\ext\eta}$.
		Therefore, by Lemma \ref{lemma:linAlgDetRank1},
		\[
			\geo = I - \frac{\nabla\brac{\rchi\ext\eta}\otimes e_3}{1 + \partial_3 \brac{\rchi\ext\eta}}
		\]
		i.e. $\geo - I = g\brac{\nabla\brac{\rchi\ext\eta}}$ for $g\brac{\p} \defeq \frac{-\p\otimes e_3}{1+\p\cdot e_3}$ for every $\p\in\R^3$ such that $\p_3 \neq 1$.
		We may now estimate $\geo-I$ using Proposition \ref{prop:postCompEstSob} to obtain
		\begin{align*}
			\normtypdom{\geo - I}{H}{4}{\Omega}
			&\lesssim
			\normtypdom{g\brac{\nabla\brac{\rchi\ext\eta}}}{L}{2}{\Omega}
			\\
			&\quad+ \underbrace{\normtypdom{g}{C}{4,1}{
				\overline{B\brac{
					\norm{\nabla\brac{\rchi\ext\eta}}{\infty}
				}}
			}}_{\eqdef\brac{\star}}
			\brac{
				\normtypdom{\nabla\brac{\rchi\ext\eta}}{H}{4}{\Omega}
				+ \normtypdom{\nabla\brac{\rchi\ext\eta}}{H}{4}{\Omega}^4
			}.
		\end{align*}
		Crucially, since we are in the small energy regime (c.f. Definition \ref{def:smallEnergyRegime}),
		\[
			\norm{\partial_3\brac{\rchi\ext\eta}}{\infty} \leqslant C_0 \delta_0 < 1
		\]
		such that $\brac{\star} < \infty$ (since $g$ is and well-defined and hence smooth on the compact set $
			\overline{B\brac{
				\norm{\nabla\brac{\rchi\ext\eta}}{\infty}
			}}
		$) and
		\begin{align*}
			\normtypdom{g\brac{\nabla\brac{\rchi\ext\eta}}}{L}{2}{\Omega}^2
			= \int_\Omega \frac{
				{\vbrac{\nabla\brac{\rchi\ext\eta}\otimes e_3}}^2
			}{
				{\brac{1 + \partial_3 \brac{\rchi\ext\eta}}}^2
			}
			\leqslant \int_\Omega \frac{
				{\vbrac{\nabla\brac{\rchi\ext\eta}}}^2
			}{
				{\brac{1 - \norm{\partial_3 \brac{\rchi\ext\eta}}{\infty}}}^2
			}
			\\
			\leqslant \frac{1}{{\brac{1 - C_0 \delta_0}}^2} \int_\Omega {\brac{\nabla\brac{\rchi\ext\eta}}}^2.
		\end{align*}
		Therefore, employing Proposition \ref{prop:prodEstSobCtsMult}, Corollary \ref{cor:boundHarmExt}, and the definition of $\enimp$ (c.f. equation (\ref{defeq:enimp})), we obtain
		\begin{align*}
			\normtypdom{\geo - I}{H}{4}{\Omega}
			&\lesssim
			\normtypdom{\nabla\brac{\rchi\ext\eta}}{L}{2}{\Omega}
			+ \normtypdom{\nabla\brac{\rchi\ext\eta}}{H}{4}{\Omega}
			+ \normtypdom{\nabla\brac{\rchi\ext\eta}}{H}{4}{\Omega}^4\\
			&\lesssim
			\normtypdom{\rchi}{H}{13/2+}{\Omega}
			\normtypdom{\ext\eta}{H}{5}{\Omega}
			+ {\brac{
				\normtypdom{\rchi}{H}{13/2+}{\Omega}
				\normtypdom{\ext\eta}{H}{5}{\Omega}
			}}^4\\
			&\lesssim
			\normtypdom{\eta}{H}{9/2}{\T^2}
			+ \normtypdom{\eta}{H}{9/2}{\T^2}^4
			\quad\lesssim
			\sqrt\enimp + \enimp^2
			\quad\lesssim \sqrt\enimp
			\quad\text{since }\enimp \leqslant \delta_0 < 1.
		\end{align*}
		We now compute the time derivatives of $\geo$:
		\[
		\begin{cases}
			\pdt\geo
			&= \pdt\brac{g\brac{\nabla\brac{\rchi\ext\eta}}}
			= \brac{\nabla g}\brac{\nabla\brac{\rchi\ext\eta}} \cdot \nabla\brac{\rchi\ext\pdt\eta},\\
			\pdt^2 \geo
			&= \pdt\brac{
				\brac{\nabla g}\brac{\nabla\brac{\rchi\ext\eta}} \cdot \nabla\brac{\rchi\ext\pdt\eta}
			}\text{, and}\\
			&= \brac{\nabla^2 g}\brac{\nabla\brac{\rchi\ext\eta}} : \brac{ {\nabla\brac{\rchi\ext\pdt\eta}^{\otimes 2}} }
			+ \brac{\nabla g}\brac{\nabla\brac{\rchi\ext\eta}} \cdot \nabla\brac{\rchi\ext\pdt^2\eta}\\
		\end{cases}
		\]
		such that we may now estimate them, using Proposition \ref{prop:prodEstSobCtsMult}, Proposition \ref{prop:postCompEstSob},
		Corollary \ref{cor:boundHarmExt}, equations (\ref{defeq:enimp}) and (\ref{defeq:dsimp}), and the fact that we are in the small energy regime.  Doing so, we obtain
		\begin{align*}
			\normtypdom{\pdt\geo}{H}{3/2}{\Omega}
			&\lesssim
			\normtypdom{\brac{\nabla g}\brac{\nabla\brac{\rchi\ext\eta}}}{H}{7/2}{\Omega}
			\normtypdom{\nabla\brac{\rchi\ext\pdt\eta}}{H}{3/2}{\Omega}\\
			&\lesssim\Big(
				\normtypdom{\nabla g}{L}{\infty}{
					\overline{B\brac{
						\norm{\nabla\brac{\rchi\ext\eta}}{\infty}
					}}
				}
				\\&
				+ \normtypdom{\nabla g}{C}{3,1}{
					\overline{B\brac{
						\norm{\nabla\brac{\rchi\ext\eta}}{\infty}
					}}
				}
				\brac{
					\normtypdom{\nabla\brac{\rchi\ext\eta}}{H}{7/2}{\Omega}
					+ \normtypdom{\nabla\brac{\rchi\ext\eta}}{H}{7/2}{\Omega}^4
				}
			\Big)
			\normtypdom{\pdt\eta}{H}{2}{\T^2}\\
			&\lesssim\sqrt\enimp,
		\end{align*}
		and
		\begin{align*}
			\normtypdom{\pdt\geo}{H}{2}{\Omega}
			&\lesssim
			\normtypdom{\brac{\nabla g}\brac{\nabla\brac{\rchi\ext\eta}}}{H}{4}{\Omega}
			\normtypdom{\nabla\brac{\rchi\ext\pdt\eta}}{H}{2}{\Omega}\\
			&\lesssim\Big(
				\normtypdom{\nabla g}{L}{\infty}{
					\overline{B\brac{
						\norm{\nabla\brac{\rchi\ext\eta}}{\infty}
					}}
				}
				\\&\quad
				+ \normtypdom{\nabla g}{C}{4,1}{
					\overline{B\brac{
						\norm{\nabla\brac{\rchi\ext\eta}}{\infty}
					}}
				}
				\brac{
					\normtypdom{\nabla\brac{\rchi\ext\eta}}{H}{4}{\Omega}
					+ \normtypdom{\nabla\brac{\rchi\ext\eta}}{H}{4}{\Omega}^4
				}
			\Big)
			\normtypdom{\pdt\eta}{H}{5/2}{\T^2}\\
			&\lesssim\sqrt\dsimp.
		\end{align*}
		Similarly, using H\"{o}lder's inequality, the Sobolev embedding $H^{3/4}\brac{\Omega} \hookrightarrow L^4 \brac{\Omega}$, Proposition \ref{prop:prodEstSobCtsMult},
		Corollary \ref{cor:boundHarmExt}, equations (\ref{defeq:enimp}) and (\ref{defeq:dsimp}), and the fact that we are in the small energy regime, we obtain
		\begin{align*}
			\normtypdom{\pdt^2 \geo}{L}{2}{\Omega}
			&\lesssim
			\normtypdom{\brac{\nabla^2 g}\brac{\nabla\brac{\rchi\ext\eta}}}{L}{\infty}{\Omega}
			\normtypdom{\nabla\brac{\rchi\ext\pdt\eta}}{H}{3/4}{\Omega}^2
			\\&\quad
			+ \normtypdom{\brac{\nabla g} \brac{\nabla\brac{\rchi\ext\eta}}}{L}{\infty}{\Omega}
			\normtypdom{\nabla\brac{\rchi\ext\pdt^2 \eta}}{L}{2}{\Omega}^2\\
			&\lesssim
			\normtypdom{\pdt\eta}{H}{5/4}{\T^2}^2
			+ \normtypdom{\pdt^2 \eta}{H}{1/2}{\T^2}
			\lesssim \sqrt\enimp + \sqrt\dsimp
			\lesssim \sqrt\dsimp.
		\end{align*}
		Finally we estimate $\nugeo$ on $\Sigma$. First we compute:
		\begin{align*}
			\nugeo \vert_\Sigma
			&= J \brac{\tr_\Sigma \geo} \cdot \nugeo \vert_\Sigma
			= J \tr_\Sigma \brac{
				I - \frac{\nabla\brac{\rchi\ext\eta}\otimes e_3}{1 + \partial_3 \brac{\rchi\ext\eta}}
			} \cdot e_3\\
			&= J \tr_\Sigma \brac{
				I - \frac{
					\ext\nablatwoemb\eta\otimes e_3 + \partial_3 \brac{\rchi\ext\eta}e_3 \otimes e_3
				}{
					1+\partial_3 \brac{\rchi\ext\eta}
				}
			} \cdot e_3
			= J \brac{e_3 - \frac{\nablatwoemb\eta + \brac{J-1}e_3}{J}}
			= -\nablatwoemb\eta + e_3.
		\end{align*}
		Therefore
		\[
			\normtypdom{\nugeo - e_3}{H}{7/2}{\T^2}
			= \normtypdom{\nabla\eta}{H}{7/2}{\T^2}
			\lesssim \sqrt\enimp.
		\]
	\end{proof}
	We now record versions of the divergence and transport theorem adapted to the differential operators appearing in the PDE
	after performing the time-dependent change of variables which fixes the domain. In particular, we prove the
	$\geo$-divergence theorem in Proposition \ref{prop:geoDivThm} and we prove the $\geo$-transport theorem in Proposition \ref{prop:geoTransportThm}.
	The key differences between these theorems and the standard divergence and transport theorems are that:
	\begin{itemize}
		\item	standard operators involving $\nabla$ are replaced by their counterparts involving $\nabla^\geo$, and
		\item	bulk integrands, i.e. integrands over $\Omega$, are multiplied by $J$
	\end{itemize}
	(see Sections \ref{sec:geoCoeff} and \ref{sec:notGeoDiffOp} for the definitions of $J$ and $\nabla^\geo$ respectively).
	\begin{prop}[$\geo$-divergence theorem]	\label{prop:geoDivThm}
		For any $v : \sbrac{0,T} \times \Omega \to \R^3$ sufficiently regular and integrable
		\[
			\int_\Omega \brac{\divgeo v} J
			= \int_{\partial\Omega} v\cdot\nugeo.
		\]
	\end{prop}
	\begin{proof}
		This result follows from the divergence theorem and the Piola identity. Indeed, we compute:
		\begin{align*}
			\int_\Omega \brac{\divgeo v} J
			= \int_\Omega \geo : \brac{\nabla v} J
			&= \int_\Omega \nabla\cdot\brac{\geo^T v J}
			- \int_\Omega \underbrace{\brac{\nabla\cdot\brac{\geo J}}}_{\stackrel{\brac{\star}}{=}0}v
			\\
			&= \int_{\partial\Omega} \brac{\geo^T \cdot v} \cdot \nu_{\partial\Omega} J
			= \int_{\partial\Omega} v \cdot \underbrace{\brac{\geo \cdot \nu_{\partial\Omega}}J}_{=\nugeo},
		\end{align*}
		where in $\brac{\star}$ we have used the Piola identity which says that cofactor matrices of gradients are divergence-free:
		$
			\nabla\cdot\brac{\geo J} = \nabla\cdot\brac{\cof\nabla\Phi} = 0.
		$
	\end{proof}
	
	Next we record a version of the transport theorem.
	
	\begin{prop}[$\geo$-transport theorem]\label{prop:geoTransportThm}
		For any $f : \sbrac{0,T} \times \Omega \to \R$ sufficiently regular and integrable
		\[
			\Dt\brac{
				\int_\Omega fJ
			}
			= \int_\Omega \brac{\pdtm^{u,\geo} f}J,
		\]
		where the differential operator $\pdtm^{u,\geo}$ is as defined in Section \ref{sec:notGeoDiffOp}.
	\end{prop}
	The proof of the $\geo$-transport theorem relies on two small computations, recorded in the following lemma.
	\begin{lemma}
		\label{lemma:geoTransportThm}
		We have that $\pdt J = \brac{\divgeo\pdt\Phi}J$, and $u\cdot\nugeo = \pdt\Phi\cdot\nugeo$ on $\partial\Omega$.
	\end{lemma}
	\begin{proof}[Proof of the $\geo$-transport theorem]
		Using Lemma \ref{lemma:geoTransportThm}, this is a direct computation:
		\[
			\Dt \brac{\int_\Omega fJ}
			= \int_\Omega \brac{\pdt f} J
			+ \int_\Omega f\brac{\pdt J},
		\]
		where
		\[
			\int_\Omega f\brac{\pdt J}
			= \int_\Omega f \brac{\divgeo\pdt\Phi} J
			= \int_\Omega \divgeo\brac{f\pdt\Phi}J
			- \int_\Omega \brac{\pdt\Phi\cdot\nabla^\geo f}J.
		\]
		To compute $\int_\Omega \divgeo\brac{f\pdt\Phi}J$ we use the $\geo$-divergence theorem and the fact that $u\cdot\nugeo = \pdt\Phi\cdot\nugeo$ on $\partial\Omega$:
		\[
			\int_\Omega \divgeo\brac{f\pdt\Phi}J
			= \int_{\partial\Omega} f \brac{\pdt\Phi\cdot\nugeo}J
			= \int_{\partial\Omega} f \brac{u\cdot\nugeo}J
			= \int_\Omega \divgeo\brac{fu}J.
		\]
		So, finally
		\[
			\Dt\brac{\int_\Omega fJ}
			= \int_\Omega 
				\underbrace{
					\brac{\pdt f - \pdt\Phi\cdot\nabla^\geo f}
				}_{
					= \pdt^\geo f
				} J
				+ \divgeo\brac{fu} J
				= \int_\Omega \brac{\pdtm^{u,\geo} f} J
		\]
		since $\divgeo u = 0$.
	\end{proof}
	\begin{proof}[Proof of Lemma \ref{lemma:geoTransportThm}]
		Computing $\pdt J$ and $\pdt\Phi\cdot\nugeo$ is nothing more than unpacking the relevant notation
		(c.f. Section \ref{sec:reformulation} for the definition of $\Phi$ and Section \ref{sec:dom_coeff} for other associated quantities).  Indeed, 
        \[
					\pdt J
					= \pdt \det \nabla\Phi
					= \det\nabla\Phi \tr \brac{{\nabla\Phi}\inv \cdot \pdt\nabla\Phi}
					= \underbrace{
						\det\brac{\nabla\Phi}
					}_{J}
					\underbrace{
						{\brac{\nabla\Phi}}^{-T}
					}_{\geo}
					: \nabla\pdt\Phi
					=\brac{\divgeo\pdt\Phi}J,
		\]
		which proves the first identity.  For the second note that on $\partial\Omega$,
				\begin{align*}
					\pdt\Phi\cdot\nugeo
					= \pdt\brac{
						\id + \rchi\ext\eta\,e_3
					} \cdot\nugeo
					= \rchi\ext\pdt\eta\,e_3\cdot\nugeo
					= \begin{cases}
						\pdt\eta	&\text{on }\Sigma_b\\
						0		&\text{on }\Sigma
					\end{cases}
				\end{align*}
				which means that $\pdt\Phi\cdot\nugeo = u\cdot\nugeo$.
	\end{proof}
	
\subsection{Commutators associated with the surface energy}
	Recall from Section \ref{sec:enerDissEst} that
	\begin{equation*}
		\mathcal{C}^{\will, \alpha} \brac{\eta}
		\defeq	\Bigbrac{
			\brac{\nugeo \svr{\eta}} \circ \partial^\alpha
			- \partial^\alpha \circ \brac{\nugeo \fvr}
		} \brac{\eta}.
	\end{equation*}
	We compute these commutators in the lemma below (for $\abs{\alpha} = 1,2$), using Remark \ref{rem:partial_der_var}.
	\begin{lemma}[Computing the commutators $\mathcal{C}^{\will,\alpha}$]\label{lemma:computingCommSurfEner}
    For $\abs{\alpha} = 1$ we have that
				\[
					\mathcal{C}^{\will,\alpha} \brac{\eta}
					= \brac{\partial^\alpha \nugeo} \brac{\fvr} \brac{\eta}.
				\]
	Also,  for $\abs{\alpha} = \abs{\beta} = 1$ we have that 
				\begin{equation*}
				\begin{split}
					\mathcal{C}^{\will, \alpha + \beta} \brac{\eta}
					&=	\brac{\partial^{\alpha + \beta} \nugeo} \brac{\fvr} \brac{\eta}
						+ \brac{\partial^\alpha \nugeo} \brac{\svr{\eta}} \brac{\partial^\beta \eta}\\
					&+ \brac{\partial^\beta \nugeo} \brac{\svr{\eta}} \brac{\partial^\alpha \eta}
						+ \nugeo \brac{\hovr{3}{\eta}} \brac{\partial^\alpha \eta, \partial^\beta \eta}.
				\end{split}
				\end{equation*}
	\end{lemma}
	\begin{proof}
		Using Remark \ref{rem:partial_der_var}, both of these results follow from direct computations:
		for $\abs{\alpha} = \abs{\beta} = 1$,
		\[
			\partial^\alpha \Bigbrac{
				\nugeo \fvr\brac{\eta}
			}
			=	\brac{\partial^\alpha \nugeo} \brac{\fvr} \brac{\eta}
				+ \nugeo \brac{\svr{\eta}} \brac{\partial^\alpha \eta}
		\]
		and
		\begin{align*}
			\partial^{\alpha + \beta} \Bigbrac{
				\nugeo \brac{\fvr} \brac {\eta}
			}
			&=	\partial^\beta \Bigbrac{
					\brac{\partial^\alpha \nugeo} \brac{\fvr} \brac{\eta}
					+ \nugeo \brac{\svr{\eta}} \brac{\partial^\alpha \eta}
				}\\
			&=	\brac{\partial^{\alpha + \beta} \nugeo} \brac{\fvr} \brac{\eta}
				+ \brac{\partial^\alpha \nugeo} \brac{\svr{\eta}} \brac{\partial^\beta \eta}\\
			&\quad	+ \brac{\partial^\beta \nugeo} \brac{\svr{\eta}} \brac{\partial^\alpha \eta}
				+ \nugeo \brac{\hovr{3}{\eta}} \brac{\partial^\alpha \eta, \partial^\beta \eta}
				+ \nugeo \brac{\svr{\eta}} \brac{\partial^{\alpha + \beta} \eta}.
		\end{align*}
	\end{proof}
	
\subsection{Form of the geometric corrections}
\label{sec:formGeoCorr}
	Recall from Section \ref{sec:geoCorr} that the geometric corrections are
	\[
		\geocor_E \brac{\x} = \engeo\brac{\x;\x} - \eneq\brac{\x} \text{ and } 
		\geocor_D \brac{\x} = \dsgeo\brac{\x;\x} - \dseq\brac{\x}.
	\]
	In this section we show that they can be computed to be
	\begin{equation}\label{eq:geoCorComp}
		\begin{cases}
			\geocor_E \brac{\x}
			&=	\sum\limits_{
					\parabolicOrder{\alpha} \leq 1, 2
				} \brac{
					\frac{1}{2} \int_\Omega \abs{\partial^\alpha u}^2 \brac{J-1}
					+ \frac{1}{2} \int_{\T^2} \brac{
						\int_0^1 g_\alpha \brac{t} \nabla^3 f \brac{t\jet\eta} \diff t
					} \bullet \brac{
						\jet\eta \otimes J\partial^\alpha \eta \otimes J\partial^\alpha \eta
					}
				}\text{ and}\\
			\geocor_D \brac{\x}
			&= 	\sum\limits_{
					\abs{\alpha}_{t,\overline{\x}^2} \leq 1
				} \brac{
					\frac{1}{2} \int_\Omega \abs{
						\symgrad^{\geo - I} \partial^\alpha u
					}^2 J
					- \int_\Omega\brac{
						\symgrad^{\geo - I} \partial^\alpha u
						:
						\symgrad \partial^\alpha u
					} J
					+ \frac{1}{2} \int_\Omega \abs{
						\symgrad \partial^\alpha u
					}^2 \brac{J-1}
				}
		\end{cases}
	\end{equation}
	where
	\begin{equation}\label{g_alpha_def}
	g_\alpha\brac{t} \defeq
		\begin{cases}
			\frac{1}{3} {\brac{1-t}}^2		&\text{for }\alpha = 0\\
			1					&\text{for }\alpha \neq 0.
		\end{cases}
	\end{equation}
	We first compute the geometric correction to the energy:
	\begin{align*}
		\geocor_E \brac{\x}
		&= \engeo\brac{\x;\x} - \eneq\brac{\x}
		= \brac{
			\engeo^0\brac{\x} - \eneq^0\brac{\x}
		} + \sum_{\parabolicOrder{\alpha} = 1, 2} \brac{
			\engeo\brac{\partial^\alpha \x;\x} - \eneq\brac{\partial^\alpha \x}
		}\\
		&= \half\int_\Omega \abs{u}^2 \brac{J-1}
		+ \brac{\will - \quadw{0}}\brac{\eta}
		+ \sum_{\parabolicOrder{\alpha}=1, 2}
			\half\int_\Omega {\vbrac{\partial^\alpha u}}^2 \brac{J-1}
			+ \brac{\quadw{\eta} - \quadw{0}} \brac{\partial^\alpha \eta}.
	\end{align*}
	Now we can compute $\will - \quadw{0}$ using Taylor's theorem (using the same notation, namely $\mathcal{P}_2$ and $\mathcal{R}_2$ as in Proposition \ref{prop:TaylorThm}),
	recalling that $f\brac{0} = 0$ and $\nabla f \brac{0} = 0$,
	\begin{align*}
		\brac{\will - \quadw{0}}\brac{\eta}
		&= \int_{\T^2} f\brac{\jet\eta} - \half\nabla^2 f\brac{0} \bullet \brac{\jet\eta\otimes \jet\eta}
		= \int_{\T^2} \brac{f - \mathcal{P}\sbrac{f,0}}\brac{\jet\eta}
		\\
		&= \int_{\T^2} \mathcal{R}\sbrac{f,0}\brac{\jet\eta}
		= \frac{1}{6} \int_{\T^2} \brac{
			\int_0^1 {\brac{1-t}}^2 \nabla^3 f\brac{t\jet\eta}dt
		} \bullet {\brac{\jet\eta}}^{\otimes 3}.
	\end{align*}
	Similarly we can compute $\brac{\quadw{\eta}-\quadw{0}}\brac{\zeta}$ for $\zeta\in {\cbrac{\partial^\alpha u}}_{\parabolicOrder{\alpha} = 1, 2}$ using the fundamental theorem of calculus:
	\begin{equation*}
		\brac{\quadw{\eta}-\quadw{0}}\brac{\zeta}
		= \half\int_{\T^2} \brac{
			\nabla^2 f \brac{\jet\eta} - \nabla^2 f \brac{0}
		} \bullet \brac{
			J\zeta \otimes J\zeta
		}
		= \half \int_{\T^2} \brac{
			\int_0^1 \nabla^3 f \brac{t\jet\eta} dt
		} \bullet \brac{
			\jet\eta \otimes J\zeta \otimes J\zeta
		},
	\end{equation*}
	which means equations \eqref{eq:geoCorComp} hold for $g_\alpha$ given by \eqref{g_alpha_def}. 
	
	We now compute the geometric correction to the dissipation.
	Note that $M \mapsto \symgrad^M v$ is linear, so in particular
	$	
		\abs{\symgrad^\geo u}^2
		= \abs{\symgrad^{\geo-I}u - \symgrad u}^2
		= \abs{\symgrad^{\geo - I}u}^2 - 2 \symgrad^{\geo - I}:\symgrad u + \abs{\symgrad u}^2.
	$
	Therefore,
	\begin{align*}
		\geocor_D \brac{\x}
		&= \dsgeo\brac{\x;\x} - \dseq\brac{\x}
		= \brac{\dsgeo^0 \brac{\x} - \dseq^0\brac{\x}}
		+ \sum_{\parabolicOrder{\alpha} = 1, 2} \brac{\dsgeo \brac{\partial^\alpha \x;\x} - \dseq\brac{\partial^\alpha \x}}\\
		&= \sum_{\parabolicOrder{\alpha} = 1,2} \brac{
			\half\int_\Omega {\vbrac{\symgrad^{\geo-I} u}}^2 J
			- \int_\Omega \brac{\symgrad^{\geo-I} u : \symgrad u} J
			+ \half\int_\Omega {\vbrac{\symgrad u}}^2 {J-1}
		}.
	\end{align*}
	
\subsection{More commutators}
	In this section we record the commutators arising when differentiating the problem.
	We record them in a form readily amenable to estimates by writing them as commutators between
	partial derivatives and linear operators with multilinear dependence on parameters which we control,
	namely $\Phi$, $\geo$, $J$, and $\nugeo$.
	\begin{lemma}[Computation of the commutators in multilinear form]
		Suppose that $\brac{u, p, \eta}$ solves \eqref{NS_fixed_s}--\eqref{NS_fixed_e}.
		Then, for each $\partial^\alpha \in \cbrac{\pdt,\nablatwo,\nablatwo^2}$,
		$\brac{\partial^\alpha u, \partial^\alpha p, \partial^\alpha \eta}$ satisfies
		\begin{subnumcases}{}
			\nonumber \pdtm^{u,\geo} v + \divgeo T^\geo = C^{1,\alpha}							&in $\Omega$,\\
			\nonumber \divgeo v = C^{2,\alpha}										&in $\Omega$,\\
			\nonumber \Bigbrac{\brac{\svr{\eta}}\zeta + g\zeta} \nugeo - T^\geo \cdot \nugeo = C^{3,\alpha}			&on $\Sigma$,\\
			\nonumber \pdt\zeta - v \cdot \nugeo = C^{4,\alpha}								&on $\Sigma$, and\\
			\nonumber v = 0													&on $\Sigma_b$
		\end{subnumcases}
		where
		\begin{subnumcases}{}
			\nonumber C^{1,\alpha} = \Bigbrac{
				- \sbrac{\partial, \pdt\Phi\cdot\nabla^\geo}
				+ \sbrac{\partial, u\cdot\nabla^\geo}
			},\\
			\nonumber \quad
			- \Bigbrac{
				\sbrac{\partial, \brac{\nabla^\geo \cdot \geo^T} \cdot \nabla}
				+ \sbrac{\partial, \brac{\geo^T \cdot \geo} : \nabla^2}
			}
				+ \sbrac{\nabla^\geo, \partial^\alpha} p\\
			\nonumber C^{2,\alpha} = \sbrac{\divgeo, \partial^\alpha} u,\\
			\nonumber C^{3,\alpha} = \Bigbrac{
				\sbrac{\nugeo\cdot\symgrad^\geo,\partial^\alpha} u
				- \sbrac{\nugeo, \partial^\alpha} p
			}
				+ g \sbrac{\nugeo,\partial^\alpha} \eta
				+ \mathcal{C}^{\will,\alpha} \brac{\eta}\text{, and}\\
			\nonumber C^{4,\alpha} = - \sbrac{\nugeo\cdot,\partial^\alpha} u.
		\end{subnumcases}
	\end{lemma}
	\begin{proof}
    Upon applying $\partial^\alpha$ to \eqref{NS_fixed_s}, we find that 
		\begin{align*}
			C^{1,\alpha}
			& = \sbrac{\pdtm^{u,\geo} - \Delta^\geo, \partial^\alpha} u + \sbrac{\nabla^\geo, \partial^\alpha} p\\
			&= \sbrac{\partial, \Bigbrac{\pdt - \pdt\Phi\cdot\nabla^\geo} + u\cdot\nabla^\geo}
				- \sbrac{\partial, \brac{\divgeo}\circ\brac{\nabla^\geo}}\\
			&= - \sbrac{\partial, \pdt\Phi\cdot\nabla^\geo}
				+ \sbrac{\partial, u\cdot\nabla^\geo}
				- \sbrac{\partial, \brac{\divgeo}\circ\brac{\nabla^\geo}}\\
			&= - \sbrac{\partial, \pdt\Phi\cdot\nabla^\geo}
				+ \sbrac{\partial, u\cdot\nabla^\geo}
				- \sbrac{\partial, \brac{\nabla^\geo \cdot \geo^T} \cdot \nabla}
				- \sbrac{\partial, \brac{\geo^T \cdot \geo} : \nabla^2}.
		\end{align*}
        The other commutators are computed by similarly differentiating \eqref{NS_fixed_div}--\eqref{NS_fixed_e}
	\end{proof}
	\begin{remark}[Explicit form of the commutators] 	Since the commutators above are written in terms of linear operators with multilinear dependence on parameters,
		we may use Proposition \ref{prop:commLinOpMultilinDepParam} to expand them into pieces that may be estimated
		using the strategy described in Proposition \ref{prop:controlInteractSobNorm}.
		Indeed: (where for the sake of readability we suppress the conditions $\beta + \sum \gamma_i = \alpha$, $\beta < \alpha$,
		from Proposition \ref{prop:commLinOpMultilinDepParam}, in the summations below)
		\begin{subnumcases}{}
			\nonumber \sbrac{\partial^\alpha,v\cdot\nabla^\geo}
				= \sum \Bigbrac{
					\brac{\partial^{\gamma_1} v}\cdot\nabla^{\partial^{\gamma_2}\geo}
				} \circ \partial^\beta,
			&where $v = -\pdt\Phi, u, \divgeo\geo^T$,\\
			\nonumber \sbrac{\partial^\alpha, M:\nabla^2}
				= \sum \Bigbrac{ \brac{\partial^\gamma M} : \nabla^2} \circ \partial^\beta,
			&where $M  =\geo^T \cdot \geo$,\\
			\nonumber \sbrac{\partial^\alpha, \nabla^\geo} = \sum \nabla^{\partial^{\gamma}\geo} \circ \partial^\beta,\\
			\nonumber \sbrac{\partial^\alpha, \nugeo \cdot \symgrad^\geo}
				= \sum \Bigbrac{
					\partial^{\gamma_1} \brac{\nugeo} \cdot \symgrad^{\nabla^{\partial^{\gamma_2}\geo}}
				} \circ \partial^\beta\text{, and}\\
			\nonumber \sbrac{\partial^\alpha, \nugeo} = \sum \partial^\gamma \brac{\nugeo} \circ \partial^\beta.
		\end{subnumcases}
	\end{remark}
	
\subsection{Computing the variations of the surface energy}
	We record in this section a more explicit expression for the first variation of the surface energy.
	This is useful when performing some critical estimates where more compact expressions for the first variation are not sufficient to close the estimates.
	\begin{lemma}
	\label{lemma:CompVarSurfEner}
		Let $\will\brac{\eta} \defeq \displaystyle\int_{\T^2} f\brac{\jet\eta}$ where we write $f = f\brac{\p,M}$.  Then the first variation of the surface energy can be written as
				\begin{align*}
					\fvr\brac{\eta}
					&=
					\nabla^2_{M,M} f\brac{\nabla\eta, \nabla^2 \eta}
							\bullet \nabla^4 \eta
					- \nabla^2_{\p,\p} f\brac{\nabla\eta, \nabla^2 \eta}
							\bullet \nabla^2 \eta
                    + \nabla^3_{M,M,M} f\brac{\nabla\eta, \nabla^2 \eta}
							\bullet \brac{\nabla^3 \eta \otimes \nabla^3 \eta}\\
					&\quad+ 2 \nabla^3_{M,M,\p} f\brac{\nabla\eta, \nabla^2 \eta}
							\bullet \brac{\nabla^3 \eta \otimes \nabla^2 \eta}s
					+ \nabla^3_{\p,M,\p} f\brac{\nabla\eta, \nabla^2 \eta}
							\bullet \brac{\nabla^2 \eta \otimes \nabla^2 \eta}.
				\end{align*}
        The second variation at the equilibrium is given by
				\[
					\brac{\svr{0}}\phi
					= \nabla^2_{M,M} f\brac{0,0} \bullet \nabla^4 \phi
					- \nabla^2_{\p,\p} f\brac{0,0} \bullet \nabla^2 \phi.
				\]
	\end{lemma}
\subsection{Estimates of the variations of the surface energy}
	In this section we obtain estimates on the variations of the surface energy,
	obtaining estimates on $\fvr$ (Lemma \ref{lemma:smallFirstVar}), $\svr{\eta}$ (Lemma \ref{lemma:boundedSecondVar}),
	and $\hovr{3}{\eta}$ (Lemma \ref{lemma:boundedThirdVar}), as well as estimates on auxiliary functions derived from $f$
	by Taylor expanding $f$ about the equilibrium, i.e. about 0 (Lemma \ref{lemma:auxFuncTaylorExp} and Corollary \ref{cor:auxFuncTaylorExp}).

	\begin{lemma}[Smallness of the first variation]	\label{lemma:smallFirstVar}
    The following hold.
		\begin{enumerate}
			\item	For all $s > -1$ there exists $C>0$ such that for every $\eta : \T^2 \to \R$ sufficiently regular
				\[
					\normtypdom{\fvr\brac{\eta}}{H}{s}{\T^2}
					\leqslant C \normtypdom{\tayh\brac{\jet\eta}}{H}{s+2}{\T^2}
						\normtypdom{\eta}{H}{s+4}{\T^2}
				\]
				for $\tayh\brac{z} \defeq \int_0^1 \nabla^2 f \brac{tz} dt$,
				where $z = \brac{\p,M} \in \R^n \times \R^{n \times n}$.
			\item	In the small energy regime, for all $s\in\sbrac{0,\half}$ and for every $\eta : \T^2 \to \R$ sufficiently regular,
				\[
					\normtypdom{\fvr\brac{\eta}}{H}{s}{\T^2} \lesssim \sqrt\enimp.
				\]
		\end{enumerate}
	\end{lemma}
	\begin{proof}
    The key observation is that we can rewrite $\fvr$ in a more amenable way using the fundamental theorem of calculus.  So let $s > -1$ and observe that
				\begin{align*}
					\normtypdom{\fvr\brac{\eta}}{H}{s}{\T^2}
					&=
						\normtypdom{\jet^* \brac{ \nabla f \brac{\jet\eta} }}{H}{s}{\T^2}
					&&=
						\normtypdom{\jet^* \brac{
							\int_0^1 \nabla^2 f \brac{t\jet\eta} dt
							\bullet \jet\eta
						}}{H}{s}{\T^2}
					\\&=
						\normtypdom{\jet^* \brac{ \tayh\brac{\jet\eta} \bullet \jet\eta}}{H}{s}{\T^2}
					&&\lesssim
						\normtypdom{\tayh\brac{\jet\eta} \bullet \jet\eta}{H}{s+2}{\T^2}
					\\&\lesssim
						\normtypdom{\tayh\brac{\jet\eta}}{H}{s+2}{\T^2}
						\normtypdom{\eta}{H}{s+4}{\T^2}
				\end{align*}
				where in the last step we have used that $s+2 > 1$ since $s > -1$.
    
    Next note that in the small energy regime we may use Corollary \ref{cor:auxFuncTaylorExp} to obtain, for any $s\in\sbrac{0,\half}$,
				\begin{align*}
					\normtypdom{\fvr\brac{\eta}}{H}{s}{\T^2}
					\lesssim
					\normtypdom{\tayh\brac{\jet\eta}}{H}{s+2}{\T^2}
					\normtypdom{\eta}{H}{s+4}{\T^2}
					\lesssim
					\normtypdom{\eta}{H}{9/2}{\T^2}
					\lesssim
					\sqrt\enimp.
				\end{align*}
	\end{proof}
	
	Next we consider the second variation.
	
	\begin{lemma}[Boundedness of the second variation of the surface energy]\label{lemma:boundedSecondVar}
		Let $s_0 > 3$ and recall the constants $C^{\brac{k}}_f$ defined in Definition \ref{def:universalConstants}.   If $\eta\in H^{s_0}\brac{\T^2}$, then for every $s\in\cobrac{2,s_0-1}$ and every $s\in\ocbrac{3,s_0}$,
		there exists a constant
		$0 < C = C\brac{
				\normtypdom{\eta}{H}{s_0}{\T^2},
				C^{\brac{\floor{s_0}+1}}_f
			}$
		such that
		\[
			\svr{\eta}\in\mathcal{L}\brac{
				H^{s}\brac{\T^2};\; H^{s-4}\brac{\T^2}
			}
			\quad\text{with}\quad
			\norm{\svr{\eta}}{\mathcal{L}\brac{H^s\brac{\T^2};\,H^{s-4}\brac{\T^2}}}
			\lesssim C,
		\]
		i.e. past a certain regularity threshold for $\eta$, we obtain that $\svr{\eta}$ is a differential operator of order 4, as expected.
	\end{lemma}
	\begin{proof}
		Let $\eta\in H^{s_0}\brac{\T^2}$ and let $\phi\in H^s \brac{\T^2}$ for some $s\in\cobrac{2,s_0-1}$.
		If $s\in\cobrac{2,s_0-1}$, then we may use Propositions \ref{prop:ProdEstSobSpaces} and \ref{prop:postCompEstSob} to see that
		\begin{align*}
			\normtypdom{\brac{\svr{\eta}}\phi}{H}{s-4}{\T^2}
			&= \normtypdom{J^*\brac{\nabla^2 f\brac{\jet\eta}\bullet J\phi}}{H}{s-4}{\T^2}
			\lesssim \normtypdom{\nabla^2 f\brac{\jet\eta}\bullet J\phi}{H}{s-2}{\T^2}\\
			&\lesssim \normtypdom{\nabla^2 f\brac{\jet\eta}}{H}{s_0-2}{\T^2}
			\normtypdom{J\phi}{H}{s-2}{\T^2}\\
			&\lesssim \Bigg(
				C^{\brac{2}}_f
				+ C^{\brac{\floor{s_0}}}_f
				\brac{
					\normtypdom{\jet\eta}{H}{s_0-2}{\T^2}
					+ \normtypdom{\jet\eta}{H}{s_0-2}{\T^2}^{\ceil{s_0}-2}
				}
			\Bigg) \normtypdom{\phi}{H}{s}{\T^2}\\
			&\lesssim \underbrace{\brac{
			C^{\brac{\floor{s_0}+1}}_f
				\brac{
					1 + \normtypdom{\eta}{H}{s_0}{\T^2}
					+ \normtypdom{\eta}{H}{s_0}{\T^2}^{\floor{s_0}-2}
				}
			}}_{\eqdef C} \normtypdom{\phi}{H}{s}{\T^2}.
		\end{align*}
		If $s\in\ocbrac{3,s_0}$, we proceed with the same estimates as above, but replacing $s_0$ with $s$.
		In particular, the key difference is that now, since $s>3$, $H^{s-2}\brac{\T^2}$ is an algebra.
	\end{proof}
	
	Next we consider the third variation.
	
	\begin{lemma}[Boundedness of the third variation of the surface energy]\label{lemma:boundedThirdVar}
		Let $s_0 > 4$ and recall that the constants $C^{\brac{k}}_f$ are given in Definition \ref{def:universalConstants}.   If $\eta\in H^{s_0} \brac{\T^2}$, then for very $s\in\brac{3,s_0-1}$ and every $p,q \geqslant 0$ such that $p + q > s + 3$	there exists a constant 
		$0 < C = C \brac{
				\normtypdom{\eta}{H}{s_0}{\T^2},
				C^{\brac{\floor{s_0}+2}}_f
			}$
		such that
		\[
			\hovr{3}{\eta} \in \mathcal{L}_2
			\brac{H^p \times H^q;\, H^{s-4}}
			\quad\text{with}\quad
			\norm{\hovr{3}{\eta}}{
				\mathcal{L}_2 \brac{H^p \times H^q ;\, H^{s-4}}
			}
			\lesssim C
		\]
		where for any normed vector spaces $V,W,X$,
		$\mathcal{L}_2\brac{V\times W; X}$ denotes the set of continuous bilinear forms on $V\times W$ mapping into $X$.
	\end{lemma}
	\begin{proof}
		Let $\eta\in H^{s_0}\brac{\T^2}$ and let $\phi\in H^p \brac{\T^2}$ and $\psi\in H^q \brac{\T^2}$ for some $s, p, q \geqslant 0$
		such that $s\in\brac{2,s_0-1}$ and $p+q > s+3$.
		Then, using Propositions \ref{prop:ProdEstSobSpaces} and \ref{prop:postCompEstSob} we obtain that
		\begin{align*}
			\normtypdom{\brac{\hovr{3}{\eta}}\brac{J\phi,J\psi}}{H}{s-4}{\T^2}
			&=
				\normtypdom{J^*\brac{\nabla^3 f\brac{\jet\eta} \bullet \brac{J\phi \otimes J\psi}}}{H}{s-4}{\T^2}\\
			&\hspace{-2cm}\lesssim
				\normtypdom{\nabla^3 f\brac{\jet\eta} \bullet \brac{J\phi \otimes J\psi}}{H}{s-2}{\T^2}
			\quad\lesssim
				\normtypdom{\nabla^3 f\brac{\jet\eta}}{H}{s_0-2}{\T^2}
				\normtypdom{J\phi \otimes J\psi}{H}{s-2}{\T^2}\\
			&\hspace{-2cm}\lesssim \Bigg(
				C^{\brac{3}}_f
				+ C^{\brac{\floor{s_0}+1}}_f
				\brac{
					\normtypdom{\jet\eta}{H}{s_0-2}{\T^2}
					+ \normtypdom{\jet\eta}{H}{s_0-2}{\T^2}^{\ceil{s_0}-2}
				}
			\Bigg)
			\normtypdom{\phi}{H}{p-2}{\T^2}
			\normtypdom{\psi}{H}{q-2}{\T^2}\\
			&\hspace{-2cm}\lesssim \underbrace{\brac{
				C^{\brac{\floor{s_0}+2}}_f
				\brac{
					1 + \normtypdom{\eta}{H}{s_0}{\T^2}
					+ \normtypdom{\eta}{H}{s_0}{\T^2}^{\floor{s_0}-2}
				}
			}}_{\eqdef C}
			\normtypdom{\phi}{H}{p}{\T^2}
			\normtypdom{\psi}{H}{q}{\T^2}.
		\end{align*}
	\end{proof}
	
	Next we control terms related to Taylor expansions of the surface energy.
	\begin{lemma}[Estimates for the auxiliary functions from the Taylor expansions of the variations of the surface energy]	\label{lemma:auxFuncTaylorExp} 
	For any $s \geqslant 2$, $f:\R^2 \times \R^{2 \times 2} \to\R$, and $\eta:\T^2\to\R$
		we have that
		\begin{align*}
			\normtypdom{
				r_k \sbrac{f,0} \brac{\jet\eta}
			}{H}{s}{\T^2}
			&\lesssim
			C^{\brac{k+1}}_f
			+ C^{\brac{\floor{s}+k+2}}_f
			\brac{
				\normtypdom{\eta}{H}{s+2}{\T^2}
				+ \normtypdom{\eta}{H}{s+2}{\T^2}^{\ceil{s}}
			},
		\end{align*}
		where $r_k$ is defined in Proposition \ref{prop:TaylorThm} and $C^{\brac{k}}_f$ is defined in Definition \ref{def:universalConstants}.  Moreover, in the small energy regime (see Definition \ref{def:smallEnergyRegime}), if $s\in\sbrac{2,\frac{5}{2}}$, then
		\[
			\normtypdom{
				r_k \sbrac{f,0} \brac{\jet\eta}
			}{H}{s}{\T^2}
			\lesssim 1.
		\]
	\end{lemma}
	\begin{proof}
		The result then follows from post-composition estimates in Sobolev spaces (see Proposition \ref{prop:postCompEstSob}) 	and from the observation that
		\[
			\brac{\partial^\alpha \brac{r_k \sbrac{f,0}}}
			\brac{z} = \int_0^1 {\brac{1-t}}^{k} t^{\abs{\alpha}} \partial^\alpha \nabla^{k+1} f\brac{tz} dt
		\]
		such that, for any $R\geqslant 0$ and any $l\in\N$,
		\[
			\normtypdom{r_k \sbrac{f,0}}{C}{l,1}{\overline{B\brac{0,R}}}
			\leqslant
			\normtypdom{\nabla^{k+1} f}{C}{l,1}{\overline{B\brac{0,R}}}.
		\]
		Therefore, since $s\geqslant 2$, we obtain from Proposition \ref{prop:postCompEstSob} that
		\begin{align*}
			\normtypdom{r_k \sbrac{f,0}\brac{\jet\eta}}{H}{s}{\T^2}
			&\lesssim
			\normtypdom{r_k \sbrac{f,0}}{L}{\infty}{\overline{
				B\brac{0,\normtypdom{\jet\eta}{L}{\infty}{\T^2}}
			}}
			\\&\quad
			+ \normtypdom{r_k \sbrac{f,0}}{C}{\floor{s},1}{\overline{
				B\brac{0,\normtypdom{\jet\eta}{L}{\infty}{\T^2}}
			}}
			\brac{
				\normtypdom{\jet\eta}{H}{s}{\T^2}
				+ \normtypdom{\jet\eta}{H}{s}{\T^2}^{\ceil{s}}
			}\\
			&\lesssim
			C^{\brac{k+1}}_f
			+ C^{\brac{\floor{s}+k+2}}_f
			\brac{
				\normtypdom{\eta}{H}{s+2}{\T^2}
				+ \normtypdom{\eta}{H}{s+2}{\T^2}^{\ceil{s}}
			}.
		\end{align*}
		In particular, in the small energy regime where
		$\normtypdom{\eta}{H}{9/2}{\T^2}
		\lesssim\sqrt\enimp\leqslant \sqrt{\delta_0}$,
		if $s\in\sbrac{2,\frac{9}{2}}$ then
		\[
			\normtypdom{\eta}{H}{3+}{\T^2} \lesssim 1
			\quad\text{and}\quad
			\normtypdom{\eta}{H}{s+2}{\T^2} \lesssim 1,
		\]
		and hence
		\[
			\normtypdom{r_k \sbrac{f,0} \brac{\jet\eta}}{H}{s}{\T^2}
			\lesssim
			\normtypdom{\nabla^{k+1} f}{L}{\infty}{\overline{B\brac{0,C}}}
			+ \normtypdom{\nabla^{k+1} f}{C}{k,1}{\overline{B\brac{0,C}}}
			\lesssim 1.
		\]
	\end{proof}

	Lemma \ref{lemma:auxFuncTaylorExp} has the following immediate corollary.
	
	\begin{cor}		\label{cor:auxFuncTaylorExp}
		If for $z=\brac{\p,M}\in\R^n\times\R^{n\times n}$ we set
		\[
			\tayh\brac{z} \defeq \int_0^1 \nabla^2 f \brac{tz} dt
			= r_0 \sbrac{\nabla f, 0} \brac{z}
            \text{ and } 
            \tayg\brac{z} \defeq \half \int_0^1 \brac{1-t} \nabla^3 f \brac{tz} dt
			= r_1 \sbrac{\nabla f, 0} \brac{z},
		\]
		then for any $s\geqslant 2$ we have the bounds
		\[
		\begin{cases}
			\normtypdom{\tayh\brac{\jet\eta}}{H}{s}{\T^2}
			\lesssim
			C^{\brac{2}}_f
			+ C^{\brac{\floor{s}+3}}_f
			\brac{
				\normtypdom{\eta}{H}{s+2}{\T^2}
				+ \normtypdom{\eta}{H}{s+2}{\T^2}^{\ceil{s}}
			}\text{ and}\\
			\normtypdom{\tayg\brac{\jet\eta}}{H}{s}{\T^2}
			\lesssim
			C^{\brac{3}}_f
			+ C^{\brac{\floor{s}+4}}_f
			\brac{
				\normtypdom{\eta}{H}{s+2}{\T^2}
				+ \normtypdom{\eta}{H}{s+2}{\T^2}^{\ceil{s}},
			}
		\end{cases}
		\]
		where the constants $C^{\brac{k}}_f$ are defined in Definition \ref{def:universalConstants}.  In particular, in the small energy regime of Definition \ref{def:smallEnergyRegime}, if $s\in\sbrac{2,\frac{5}{2}}$ then
		\[
			\normtypdom{\tayh\brac{\jet\eta}}{H}{s}{\T^2} \lesssim 1
			\quad\text{and}\quad
			\normtypdom{\tayg\brac{\jet\eta}}{H}{s}{\T^2} \lesssim 1.
		\]
	\end{cor}

\section{Generic tools}
In this second part of the appendix we record generic tools,
i.e. results that are employed throughout this paper but whose applicability is not reduced to the problem in this paper.
In particular, these results are either well-known or slight modifications of standard results.
They are therefore recorded here so that they may be precisely stated as reference for when they are invoked elsewhere in this paper.
\subsection{Variations/derivatives of the surface energy}
	In this section we record various expressions for variations of, and functionals associated with, the surface energy.
	Recall from Section \ref{sec:notSurfEner} that the surface energy associated with a surface given as the graph of $\eta$ is
	\[
		\will\brac{\eta} = \int_{\T^2} f\brac{\jet\eta}
	\]
	where the jet $\jet\eta$ is given by $\jet\eta = \brac{\nabla\eta,\nabla^2 \eta}$. 	Similarly, the definitions of $\delta_\phi$, $\fvr, \svr{\eta}, \dots$, $D\will, D^2 \will,$ etc 	are in Section \ref{sec:notSurfEner}.  We begin by giving the form of variations of $\will$. 
	
	\begin{lemma}[Various representations of the variations/derivatives of the surface energy]		\label{lemma:repDerivSurfEner}
	For any sufficiently regular functions  $\eta, \phi, \psi, \phi_i : \T^2 \to \R$, $i=1, \dots, k$, the following hold:
		\begin{enumerate}
			\item	$$
					\abrac{D\will\brac{\eta}, \phi}
					= \int_{\T^2} \fvr\brac{\eta} \phi
					= \int_{\T^2} \nabla f \brac{\jet\eta} \cdot \jet\phi,
				$$
			\item	$$
					\abrac{D^2 \will\brac{\eta}, \brac{\phi, \psi}}
					= \int_{\T^2} \brac{\brac{\svr{\eta}}\phi} \psi
					= \int_{\T^2} \nabla^2 f \brac{\jet\eta} \bullet \brac{\jet\phi \otimes \jet\psi},
				$$
			\item	\begin{align*}
					\abrac{D^k \will\brac{\eta}, \brac{\phi_1, \phi_2, \dots, \phi_{k-1}, \phi_k}}
					&= \int_{\T^2} \Bigbrac{\brac{\hovr{k}{\eta}} \brac{\phi_1, \phi_2, \dots, \phi_{k-1}}} \phi_k\\
					&= \int_{\T^2} \nabla^k f \brac{\jet\eta} \bullet \brac{\jet\phi_1 \otimes \jet\phi_2 \otimes \dots \otimes \jet\phi_{k-1} \otimes \jet\phi_k}.
				\end{align*}
		\end{enumerate}
	\end{lemma}
	\begin{remark}\label{rem:partial_der_var}
		We record here formulae for partial derivatives of the first and second variation.
		For $\alpha, \beta$ multi-indices such that $\abs{\alpha} = \abs{\beta} = 1$, we have
		\[
			\partial^\alpha \Bigbrac{\fvr\brac{\eta}} = \svr{\eta}\brac{\partial^\alpha \eta}
			\quad\text{and}\quad
			\partial^{\alpha + \beta} \Bigbrac{\fvr\brac{\eta}}
				= \hovr{3}{\eta}\brac{\partial^\alpha \eta,\partial^\beta \eta}
				+ \svr{\eta}\brac{\partial^{\alpha + \beta} \eta}.
		\]
	\end{remark}
	We now record a lemma that comes in handy when computing second variations.
	\begin{lemma}[Computing the second variation]\label{lemma:compSVR}
	For any $\eta,\phi : \T^2 \to \R$ sufficiently regular,
		\[
			\delta_\phi\brac{\fvr\brac{\eta}} = \brac{\svr{\eta}}\phi.
		\]
	\end{lemma}
	\begin{proof}
        For any $\psi\in C^\infty_c \brac{\T^2}$, we compute
		\begin{align*}
			\int_{\T^2} \delta_\phi\brac{\fvr\brac{\eta}} \psi
			= \delta_\phi \int_{\T^2} \fvr\brac{\eta}\psi
			= \delta_\phi \abrac{D\will\brac{\eta},\psi}
			= \delta_\phi \delta_\psi \will\brac{\eta}
			\\
			= \abrac{D^2 \will \brac{\eta},\brac{\phi,\psi}}
			= \int_{\T^2} \brac{\brac{\svr{\eta}}\phi} \psi.
		\end{align*}
	\end{proof}
	We now record a computation telling us how the quadratic approximation to the surface energy behaves when differentiated in time,
	which comes in handy when estimating the commutators.
	
	\begin{prop}
		For any $\eta, \zeta : \T^2 \to \R$ sufficiently regular,
		\begin{equation*}
			\Dt \brac{\quadw{\eta}\brac{\zeta}}
			= \quadw{\dot{\eta}}\brac{\zeta}
				+ \abrac{D^2 \will \brac{\eta}, \brac{\zeta, \pdt\zeta}}
			=\quadw{\dot{\eta}}\brac{\zeta}
				+ \int_\mathcal{U} \Bigbrac{\brac{\svr{\eta}}\zeta}\pdt\zeta.
		\end{equation*}
	\end{prop}
	\begin{proof}
		This result is nothing more than the product rule transcribed into our notation.
		This is apparent when rewriting the formula above as:
		\[
			\Dt \brac{
				\half \int_{\T^2} \nabla^2 f \brac{\jet\eta} \otimes \brac{\jet\zeta \otimes \jet\zeta}
			}
			= \half \int_{\T^2} \pdt\brac{\nabla^2 f \brac{\jet\eta}} \otimes \brac{\jet\zeta \otimes \jet\zeta}
			+ \int_{\T^2} \nabla^2 f \brac{\jet\eta} \otimes \brac{\jet\zeta \otimes \jet\pdt\zeta}.
		\]
	\end{proof}
	
\subsection{Harmonic extension}
In this section we record the standard definition and estimates of the harmonic extension of a function from $\T^2$ to $\T^2 \times \brac{-\infty,0}$.  Although the extension is defined in this large set, we will typically only need in on $\T^2 \times \brac{-b,0}$.
\label{sec:harmExt}
	\begin{definition}[Harmonic extension]  We define the following.
		\begin{enumerate}
			\item For any $f \in L^1(\T^2)$, define $\ext f : \T^2 \times \brac{-\infty,0} \to \R$ by, for every $x\in \T^2 \times \brac{-\infty,0}$,
				\[
					\brac{\ext f}\brac{x} \defeq \sum_{\bar{k}\in\Z^2} \brac{
						\hat{f}\brac{\bar{k}} e^{2\pi \abs{\bar{k}} x_3}
					} e^{2\pi i \bar{k} \cdot \bar{x}},
				\]
	where $\hat{\cdot}$ denotes the Fourier transform
	and where recall that $x=\brac{\bar{x},x_3}$.
			\item For any $f : \cobrac{0,T} \times \T^2 \to \R$, define $\ext f : \cobrac{0,T} \times \T^2 \times \brac{-\infty,0} \to \R$ by ${\brac{\ext f}}(t,\cdot) \defeq \ext \brac{f(t,\cdot)}$.
		\end{enumerate}
	\end{definition}
	\begin{remark}
		Recall that $\ext f$ as defined above is called the \emph{harmonic} extension of $f$ because it solves
		\[
			\begin{cases}
				-\Delta\ext f = 0		&\text{in } \T^2 \times \brac{-b,0},\\
				\ext f = f			&\text{on } \cbrac{x_3 = 0}.
			\end{cases}
		\]
	\end{remark}

	Next we record some identities related to the harmonic extension.
	\begin{lemma}[Identities for the derivatives of the harmonic extension]	\label{lemma:harmExtIdentities}
		For any $f : \T^2 \rightarrow \R$ sufficiently regular,
		\[
			\partial_3 \ext f = \ext \sqrt{-\Delta} f
			\quad\text{and}\quad
			\nablatwo \ext f = \ext \nablatwo f,
		\]
	where ${\brac{\sqrt{-\Delta} f}}^{\wedge} \brac{\bar{k}} \defeq 2 \pi \abs{\bar{k}}$ for all $\bar{k}\in\Z^2$.
	\end{lemma}
	\begin{proof}
		These results follow directly from short computations on the Fourier side.
	\end{proof}
	
	Next we record some useful estimates, starting with $L^2$ ones.
	\begin{lemma}[$L^2$ bound on the harmonic extension]\label{lemma:harmExtL2Bound}
	For any $f : \T^2 \rightarrow \R$ sufficiently regular,
	\[
		\normtypdom{\ext f}{L}{2}{\Omega}
		\leqslant \frac{1}{2\sqrt\pi} \normtypdom{f}{\dot{H}}{-1/2}{\T^2}
	\]
	where $\Omega = \T^2 \times (-b,0)$.
	\end{lemma}
	\begin{proof}
		To obtain this inequality we proceed as follows: employ Parseval's identity on the horizontal slices,
		then apply Tonelli's theorem so that we may integrate exactly along the vertical direction,
		and finally note that $1 - e^{-4\pi b \abs{\cdot}}\leqslant 1$.
	\end{proof}
	
	The $L^2$ bounds coupled with the identities for the derivatives of the harmonic extension lead to $H^s$ bounds.
	\begin{cor}[$H^s$ bounds on the harmonic extension]
	\label{cor:boundHarmExt}
	Recall that $\Omega = \T^2 \times (-b,0)$. For any $s \geqslant 0$, there exists $C_s > 0$ such that for any $f\in H^{s-1/2}\brac{\T^2}$,
	\[
		\normtypdom{\ext f}{H}{s}{\Omega}
		\lesssim \normtypdom{f}{H}{s-1/2}{\T^2},
	\]
	\end{cor}
	\begin{proof}
		This result follows from Lemmas \ref{lemma:harmExtIdentities} and \ref{lemma:harmExtL2Bound} when $s$ is an integer and a standard interpolation argument otherwise.
	\end{proof}
	
\subsection{Commutators with linear operators with multilinear dependence on their parameters}
	In this section we record how to compute commutators between partial derivatives and linear operators with multilinear dependence on their parameters.
	\begin{prop}\label{prop:commLinOpMultilinDepParam}
		Suppose that $L$ is a linear differential operator acting on functions $\eta : \cobrac{0,T}\times\T^2\to\R$ that can be written as
		$L = \hat{L}\brac{\pi_1,\dots,\pi_k}$ for some parameters $\pi_1,\dots,\pi_k:\cobrac{0,T}\times\T^2\to\R$, where $\hat{L}$ is multilinear.
		Then, for any multi-index $\alpha=\brac{\alpha_0, \bar{\alpha}}\in\N^3$ such that $\partial^\alpha = \partial_t^{\alpha_0} \partial_{\bar{x}}^{\bar{\alpha}}$, we have
		\[
			\sbrac{\partial^\alpha,L}
			= \sum_{\substack{
				\beta + \sum_{i=1}^{k} \gamma_i = \alpha\\
				\beta < \alpha
			}}
			\hat{L}\brac{\partial^{\gamma_1}\pi_1,\dots,\partial^{\gamma_k}\pi_k} \circ \partial^\beta.
		\]
	\end{prop}
	\begin{proof}
		For any $\eta:\T^2\to\R$, we compute directly, using the multilinearity of $\hat{L}$ in $\brac{\star}$ below:
		\begin{align*}
			\sbrac{\partial^\alpha, L} \eta
			&= \partial^\alpha \Bigbrac{
				\hat{L} \brac{\pi_1,\dots,\pi_k} \eta
			} - L\brac{\partial^\alpha \eta}
			\stackrel{\brac{\star}}{=}
			\brac{
				\sum_{\beta + \sum_{i=1}^{k} \gamma_i = \alpha}
				\hat{L} \brac{\partial^{\gamma_1} \pi_1, \dots, \partial^{\gamma_k} \pi_k}
				\partial^\beta \eta
			} - L \brac{\partial^\alpha \eta}\\
			&= \brac{
				\sum_{\substack{
					\beta + \sum_{i=1}^{k} \gamma_i = \alpha\\
					\beta \neq \alpha
				}}
				\hat{L}\brac{\partial^{\gamma_1}\pi_1,\dots,\partial^{\gamma_k}\pi_k} \circ \partial^\beta
			} \eta
		\end{align*}
		Since $\beta\neq\alpha$  above is equivalent to $\beta < \alpha$ in this context (since necessarily $\beta \leqslant \alpha$), we obtain the result desired.
	\end{proof}
	
\subsection{General recipe for controlling interactions with Sobolev norms}
We record here a general recipe for controlling interactions with Sobolev norms by combining the H\"{o}lder inequality and appropriate Sobolev embeddings. 
	\begin{prop}	\label{prop:controlInteractSobNorm}
		Let $n,k\in\N$ and let $s_1,\dots,s_k \geqslant 0$ be such that either
		\[
			\text{(i) }
			\sum_{i=1}^{k} \min\brac{s_i,\frac{n}{2}} > n\brac{\frac{k}{2}-1}
			\text{ or (ii) }
			\sum_{i=1}^{k} \min\brac{s_i,\frac{n}{2}} \geqslant n\brac{\frac{k}{2}-1}
			\text{ and }
			s_i \neq \frac{n}{2} \text{ for all } i
		\]
		holds. Then there exists $C>0$ such that for every $f_1 \in H^{s_1}\brac{\T^n}, \dots, f_k \in H^{s_k}\brac{\T^n}$,
		\[
			\vbrac{\int_{\T^n} f_1 \dots f_k} \leqslant C \normtypdom{f_1}{H}{s_1}{\T^n} \dots \normtypdom{f_k}{H}{s_k}{\T^n}.
		\]
	\end{prop}
	
\subsection{Product estimates in Sobolev spaces}
	In this section we record for which regularity indices $s,t,u$ it holds that $H^s \cdot H^t \hookrightarrow H^u$.
	Using Fourier analysis, these results boil down to:
	\begin{enumerate}
		\item	The following pointwise bound on the Fourier side:
			\[
				{\abrac{\cdot}}^s \vbrac{{\brac{fg}}^s}
				\lesssim
				{\abrac{\cdot}}^s \abs{\hat{f}} * \abs{\hat{g}}
				+ \abs{\hat{f}} * {\abrac{\cdot}}^s \abs{\hat{g}}
			\]
			for $f,g : \T^n \to \R$, which follows from the elementary observation that
			${\abrac{k}}^2 \lesssim {\abrac{k-l}}^2 + {\abrac{l}}^2$ for all $k,l\in\Z^n$.
		\item	Young's inequality for convolutions.
		\item	Using H\"{o}lder's inequality on the Fourier side to show that
			\[
				\normtypdom{{\abrac{\cdot}}^s \hat{f}}{l}{p}{\Z^n}
				\lesssim
				\normtypdom{f}{H}{s+\alpha}{\T^n}
			\]
			for the appropriate values $s$, $p$ and $\alpha$.
	\end{enumerate}
	
	\begin{prop}[$H^s$ is a Banach algebra when $s > s_*$] \label{prop:ProdEstSobSpaces}
		Let $D = \T^2,\, \Omega$ and correspondingly let $s_* = 1, \frac{3}{2}$.
		If $s > s_*$, then 
		\[
			H^s \brac{D} \cdot H^s \brac{D} \hookrightarrow H^s \brac{D}
		\]
		i.e. for every $s > s_*$ there exists $C > 0$ such that for every $f,g \in H^s \brac{D}$,
		the product $fg$ belongs to $H^s \brac{D}$ and satisfies the estimate
		\[
			\normtypdom{fg}{H}{s}{D}
			\leqslant C
			\normtypdom{f}{H}{s}{D}
			\normtypdom{g}{H}{s}{D}.
		\]
	\end{prop}
	
	\begin{prop}[$H^{s+\alpha}$ is a continuous multiplier on $H^s$ when $\alpha > s_*$]\label{prop:prodEstSobCtsMult}
		Let $D = \T^2,\, \Omega$ and correspondingly let $s_* = 1, \frac{3}{2}$.
		For every $s \geqslant 0$, if $\alpha > s_*$, then
		\[
			H^{s+\alpha} \brac{D} \cdot H^s \brac{D} \hookrightarrow H^s \brac{D}
		\]
		i.e. for every such $s$ and $\alpha$ there exists $C > 0$ such that for every $f \in H^{s+\alpha} \brac{D}$ and $g \in H^s\brac{D}$,
		the product $fg$ belongs to $H^s \brac{D}$ and satisfies the estimate
		\[
			\normtypdom{fg}{H}{s}{D}
			\leqslant C
			\normtypdom{f}{H}{s+\alpha}{D}
			\normtypdom{g}{H}{s}{D}.
		\]
	\end{prop}
	
	\begin{prop}[Borrowing regularity from both factors]\label{prop:prodEstSobBothFactors}
		Let $D = \T^2,\, \Omega$ and correspondingly let $s_* = 1, \frac{3}{2}$.
		For every $s \geqslant 0$ and $\alpha, \beta > 0$, if $s + \brac{\alpha + \beta} > s_*$,
		then
		\[
			H^{s+\alpha} \brac{D} \cdot H^{s+\beta} \brac{D} \hookrightarrow H^s \brac{D}
		\]
		i.e. for every such $s$, $\alpha$, and $\beta$ there exists $C > 0$ such that for every $f \in H^{s+\alpha} \brac{D}$ and $g \in H^{s+\beta} \brac{D}$,
		the product $fg$ belongs to $H^s \brac{D}$ and satisfies the estimate
		\[
			\normtypdom{fg}{H}{s}{D}
			\leqslant C
			\normtypdom{f}{H}{s+\alpha}{D}
			\normtypdom{g}{H}{s+\beta}{D}
		\]
	\end{prop}
	
\subsection{Post-composition estimates in Sobolev spaces}
	We record here conditions on $s$ for $H^s$ to be closed under post-composition by a sufficiently smooth function 
	(also known as a Nemytskii operator, or as a superposition operator).
	
	These post-composition estimates boils down to estimates of the multilinear terms involving derivatives of various orders which appear in the Fa\`{a} di Bruno formula
	(i.e. the chain rule for higher-order derivatives).
	The key observation is that these terms can be written as derivatives of polynomials.
	Coupling this observation with the fact that $H^s$ is an algebra for sufficiently large $s$ (c.f. Proposition \ref{prop:ProdEstSobSpaces}) thus yields the post-composition estimates.
	\begin{prop}\label{prop:postCompEstSob}
	Let $D = \T^2,\, \Omega$ and correspondingly let $s_* = 1, \frac{3}{2}$.
	Let $k\in\N$ and $\alpha\in\cobrac{0,1}$.
	If $k > s_*$ then for every $g \in H^{k+\alpha} \brac{D;\R}$ and for every $F \in C^{k,1}_{loc} \brac{\R;\R}$,
	$F \circ g \in H^{k+\alpha} \brac{D;\R}$ with
	\begin{align*}
		\normtypdom{F \circ g}{H}{k+\alpha}{D}
		&\lesssim
		\normtypdom{F \circ g}{L}{2}{D}
		+ \normtypdom{F}{C}{k,1}{\overline{
			B\brac{0,\norm{g}{\infty}}
		}}
		\brac{
			\normtypdom{g}{H}{k+\alpha}{D}
			+ \normtypdom{g}{H}{k+\alpha}{D}^{k+\ceil{\alpha}}
		}\\
		&\lesssim
		\normtypdom{F}{L}{\infty}{\overline{
			B\brac{0,\norm{g}{\infty}}
		}}
		+ \normtypdom{F}{C}{k,1}{\overline{
			B\brac{0,\norm{g}{\infty}}
		}}
		\brac{
			\normtypdom{g}{H}{k+\alpha}{D}
			+ \normtypdom{g}{H}{k+\alpha}{D}^{k+\ceil{\alpha}}
		}
	\end{align*}
	where $B\brac{R} = \brac{-R,R}$ and $\ceil{x}$ denotes the smallest integer greater than or equal to $x$.
\end{prop}

\subsection{Elliptic estimates for the Stokes problem}
In this section we record estimates for the Stokes problem.  We begin with the case of Dirichlet conditions.

	\begin{prop}[Estimates for the Stokes problem with Dirichlet boundary condition]\label{eqEstStokesProbDirBC}
		Let $s\geq 0$, let $f\in H^{s}\brac{\Omega}$, $g\in H^{s+1}\brac{\Omega}$, and $h\in H^{s+3/2}\brac{\partial\Omega}$ satisfy $\int_\Omega f = \int_{\partial\Omega} h\cdot\nu$,
		and let $\brac{u,p}$ solve
		\[
		\begin{cases}
			-\Delta u + \nabla p = f	&\text{in }\Omega,\\
			\nabla\cdot u = g		&\text{in }\Omega\text{, and}\\
			u = h				&\text{on }\partial\Omega.
		\end{cases}
		\]
		Then
		\begin{equation*}
			\norm{u}{H^{s+2}\brac{\Omega}} + \norm{\nabla p}{H^{s}\brac{\Omega}}
			\lesssim \norm{f}{H^{s}\brac{\Omega}} + \norm{g}{H^{s+1}\brac{\Omega}} + \norm{h}{H^{s+3/2}\brac{\partial\Omega}}.
		\end{equation*}
	\end{prop}
	
	Next we consider the Stokes problem with different boundary conditions.
	
	\begin{prop}[Estimates for the Stokes problem with mixed Dirichlet-Neumann boundary condition]\label{eqEstStokesProbMixBC}
		Let $s \geq 0$, let $f\in H^{s}\brac{\Omega}$, $g\in H^{s+1}\brac{\Omega}$, $h_1 \in H^{s+3/2}\brac{\Sigma}$, and $h_2 \in H^{s+1/2}\brac{\Sigma}$,
		and let $\brac{u,p}$ solve
		\[
		\begin{cases}
			-\Delta u + \nabla p = f			&\text{in }\Omega,\\
			\nabla\cdot u = g				&\text{in }\Omega,\\
			u \cdot e_3 = h_1				&\text{on }\Sigma,\\
			{\brac{\symgrad u \cdot e_3}}_{tan} = h_2	&\text{on }\Sigma\text{, and}\\
			u = 0						&\text{on }\Sigma_b
		\end{cases}
		\]
		where $v_{tan} \defeq \brac{I - e_3 \otimes e_3} v$, i.e. $v_{tan}$ is the \emph{tangential} part of $v$. Then
		\begin{equation*}
			\normtypdom{u}{H}{s+2}{\Omega} + \normtypdom{\nabla p}{H}{s}{\Omega}
			\lesssim \normtypdom{f}{H}{s}{\Omega} + \normtypdom{g}{H}{s+1}{\Omega} + \normtypdom{h}{H}{s+3/2}{\Sigma} + \normtypdom{h}{H}{s+1/2}{\Sigma}.
		\end{equation*}
	\end{prop}
	
\subsection{Dynamic boundary conditions}\label{sec:dynBC}

    We now turn our attention to estimates related to the dynamic boundary condition \eqref{NS_fixed_dyn}.  We begin with a definition.

	\begin{definition}
		Let $L$ be a linear differential operator acting on functions $\eta : \T^n \to \R$ and let $k\in\N$.
		We say that $L$ is a \emph{strictly elliptic $k$-th order differential operator on functions of average zero}
		if there exists $C>0$ such that $\hat{L} \geqslant C {\abs{\cdot}}^k$.
	\end{definition}
	
	Next we record elliptic estimates for such operators.
	\begin{prop}
		\label{prop:genDynBCEst}
		Let $L$ be a strictly elliptic $k$-th order differential operator on the $n$-torus.
		Then there exists $C>0$ such that for every $s\in\R$ and every $f\in H^s\brac{\T^n}$,
		if $\eta$ solves $L\eta = f$ on $\T^n$, then
		\[
			\normtypdom{\eta}{\dot{H}}{s+k}{\T^n} \leqslant C \normtypdom{f}{\dot{H}}{s}{\T^n}.
		\]
	\end{prop}
	\begin{proof}
		This result follows immediately from the assumption on $L$ of strict ellipticity over functions of average zero:
		\[
			\normtypdom{\eta}{\dot{H}}{s+k}{\T^n}^2
			= \normtypdom{
				{\abs{\cdot}}^{2s} \hat{\eta}
			}{l}{2}{\Z^n}^2
			\leqslant \frac{1}{C^2}
			\normtypdom{
				{\abs{\cdot}}^{2\brac{s-k}} \hat{L} \hat{\eta}
			}{l}{2}{\Z^n}^2
			= \frac{1}{C^2} \normtypdom{f}{\dot{H}}{s-k}{\T^n}^2.
		\]
	\end{proof}
	
	A byproduct of Proposition \ref{prop:genDynBCEst} is the following estimate, tailored to the dynamic boundary condition.
	
	\begin{cor}[Estimates for the dynamic boundary condition]	\label{eqEstDynBC}
		Let $g > 0$ and $f : \R^2 \times \R^{2\times 2} \to \R$, write $f = f\brac{\p,M}$ for $\brac{\p,M}\in\R^2 \times \R^{2 \times 2}$, and suppose that \eqref{f_assume_hessian} holds. Then for every $s\geqslant 0$ there exists $\tilde{C} > 0$ such that for every $f\in H^s\brac{\T^2}$, if $\eta$ satisfies $\int_{\T^2} \eta = 0$ and solves $\brac{\svr{0} + g} \eta = f$ on $\T^2$, then
		\[
			\normtypdom{\eta}{H}{s+4}{\T^n} \leqslant \tilde{C} \normtypdom{f}{\dot{H}}{s}{\T^n}.
		\]
	\end{cor}
	\begin{proof}
		The assumption \eqref{f_assume_hessian} tells us precisely that $\svr{0} + g$ is a strictly elliptic fourth-order operator over functions of average zero. Proposition \ref{prop:genDynBCEst} thus yields the desired result,	since on $\Z^n \setminus \cbrac{0}$, ${\abs{\cdot}}^s \asymp \jap{\cdot}{s}$,
		and hence for functions of average zero $\normtypdom{\cdot}{\dot{H}}{s}{\T^n} \asymp \normtypdom{\cdot}{H}{s}{\T^n}$.
	\end{proof}
	
	Next we consider Poincar\'{e}-type inequalities.
	\begin{prop}[Poincar\'{e}-type inequalities]\label{estPoincareType}
	The following hold.
		\begin{enumerate}
		\item	There exists $C^P > 0$ such that for every $\phi\in H^1\brac{\Omega}$,
			\begin{equation}
			\label{eqEstPoincareTypeTrace}
				\norm{\phi}{H^1\brac{\Omega}} \leq C^P \brac{
					\norm{\tr\phi}{L^2\brac{\Sigma}}
					+ \norm{\nabla\phi}{L^2\brac{\Omega}},
				}
			\end{equation}
		\item	For every $s \geq 0$, there exists $C^P_s > 0$ such that for every $\eta\in H^{s+1} \brac{\T^n}$ satisfying  $\int_{\T^n} \eta = 0$ we have that
			\begin{equation}
			\label{eqEstPoincareTypeHs}
				\norm{\eta}{H^{s+1}} \leq C^P_s \norm{\nabla\eta}{H^s}.
			\end{equation}	
		\end{enumerate}
	\end{prop}

	Korn's inequality, which we record now, is a sort of Poincar\'{e}-type inequality for the symmetrized gradient.
	See Lemma 2.7 in \cite{beale_1981} for a proof.
	\begin{prop}[Korn inequality]\label{eqEstKorn}
		There exist $C_K > 0$ such that for every $\phi\in H^1\brac{\Omega}$, if $\phi = 0$ on $\Sigma_b$, then
		\begin{equation*}
			\norm{\phi}{H^1 \brac{\Omega}} \leq C_K \norm{\symgrad\phi}{L^2 \brac{\Omega}}.
		\end{equation*}
	\end{prop}

\subsection{Linear algebra}
    In this section we record some simple facts from linear algebra.
	\begin{lemma}[Determinant of a rank 1 perturbation of the identity]	\label{lemma:linAlgDetRank1}
		Let $a,b\in\R^n$ and let $M = I + a \otimes b$. Then $\det M = 1 + a \cdot b$.  Moreover, if $a \cdot b \neq 1$ then $M$ is invertible and $M\inv = I - \frac{a \otimes b}{1 + a\cdot b}$.
	\end{lemma}
	
\subsection{Taylor's theorem}
	We record Taylor's theorem here in order to fix notation.
	\begin{prop}[Taylor's theorem with integral remainder]\label{prop:TaylorThm}
		For any $f\in C^{k+1} \brac{\R^d;\R}$ and any $z_0 \in \R^d$,
		\[
			f
            = \mathcal{P}_k \sbrac{f, z_0} + r_k \sbrac{f, z_0} \bullet {\brac{\cdot - z_0}}^{\otimes\brac{k+1}}
			= \mathcal{P}_k \sbrac{f,z_0}
			+ \mathcal{R}_k \sbrac{f,z_0},
		\]
		where, for any $z\in\R^d$,
		\[
			\mathcal{P}_k \sbrac{f,z_0} \brac{z}
			\defeq \sum_{l=0}^k
			\frac{1}{l!}
			\nabla^l f\brac{z_0}
			\bullet
			\brac{z-z_0}^{\otimes l},
		\]
		$\mathcal{R}_k \sbrac{f,z_0} \defeq r_k \sbrac{f,z_0} \bullet {\brac{\cdot - z_0}}^{\otimes\brac{k+1}}$, 
		and
		\[
			r_k \sbrac{f,z_0} \brac{z}
			\defeq \frac{1}{\brac{k+1}!}
			\int_0^1
			\brac{1-t}^k
			\nabla^{k+1} f \brac{
				\brac{1-t} z_0 + t z
			}
			\mathrm{d}t
		\]
	\end{prop}
	\begin{example}
		For example, when $k=2$ we have
		\[
			f\brac{z}
			=
			\underbrace{
				f\brac{0}
				+ \nabla f\brac{0} \cdot z
				+ \frac{1}{2} \nabla^2 f \brac{0} \cdot \brac{z \otimes z}
			}_{
				\mathcal{P}_2 \sbrac{f,0} \brac{z}
			}
			+ \underbrace{
				\frac{1}{6} \brac{
					\int_0^1 \brac{1-t}^2 \nabla^3 f \brac{tz} \mathrm{d}t
				} \bullet \brac{z \otimes z \otimes z}
			}_{
				\mathcal{R}_2 \sbrac{f,0} \brac{z}
			}.
		\]
	\end{example}
	
\addtocontents{toc}{\protect\setcounter{tocdepth}{0}}
\bibliographystyle{alpha-bis}
\bibliography{main}

\end{document}